\pdfoutput=1
\documentclass[11pt]{article}

\usepackage[T1]{fontenc}
\usepackage[utf8]{inputenc}
\usepackage{lmodern}

\usepackage{microtype}
\usepackage{geometry}

\usepackage{amsmath}
\usepackage{amssymb}
\usepackage{mathtools}
\usepackage{bbm}       
\usepackage{bm}         
\usepackage{mathrsfs}   
\usepackage{dsfont}     

\usepackage{graphicx}
\usepackage{xcolor}
\usepackage{float}
\usepackage{subcaption}
\usepackage{wrapfig}

\usepackage{booktabs}
\usepackage{array}
\usepackage{multirow}

\usepackage{enumitem}

\usepackage{algorithmic}

\usepackage{listings}
\usepackage{verbatim}

\usepackage{nameref}

\usepackage{xspace}

\usepackage[numbers,sort&compress]{natbib}

\usepackage{hyperref}
\hypersetup{
  colorlinks=true,
  linkcolor=blue,
  citecolor=blue,
  urlcolor=blue,
  pdfauthor={},
  pdftitle={}
}

\setlist{itemsep=2pt, topsep=4pt, parsep=0pt, partopsep=0pt}

\mathtoolsset{showmanualtags}

\reversemarginpar
\usepackage[colorinlistoftodos,prependcaption,textsize=tiny]{todonotes}

\newcommand{\mcl}{\mathcal}

\newcommand{\mbb}{\mathbb}

\newcommand{\iidsim}{\stackrel{iid}{\sim}}

\newcommand{\dd}{{\rm d}}

\newcommand{\N}{\mcl{N}}

\renewcommand{\H}{\mcl{H}}
\renewcommand{\L}{\mcl{L}}

\newcommand{\U}{\mcl{U}}

\newcommand{\V}{{\mcl{V}}}
\newcommand{\W}{{\mcl{W}}}

\newcommand{\PP}{\mbb{P}}

\newcommand{\R}{\mbb{R}}
\newcommand{\EE}{\mbb{E}}

\newcommand{\Law}{{\rm Law}\xspace}

\definecolor{darkred}{rgb}{.7,0,0}

\definecolor{darkgreen}{rgb}{.15,.55,0}

\definecolor{darkblue}{rgb}{0,0,0.7}

\usepackage{amsopn}
\usepackage{amssymb}

\reversemarginpar
\usepackage[colorinlistoftodos,prependcaption,textsize=tiny]{todonotes}

\newcommand{\einf}{\mathrm{ess}\inf}

\newcommand{\tr}{\mathrm{tr}}

\newcommand{\kl}{\text{D}_\text{KL}}

\newcommand{\Range}{\mathrm{Range}}

\newcommand{\psiclass}{\mathcal{P}_\psi(\RefM)}

\newcommand{\RefM}{\mu}
\newcommand{\TargetM}{\nu}
\newcommand{\pma}{\tilde{\eta}}

\newcommand{\Var}{\operatorname{Var}}
\newcommand{\Cov}{\operatorname{Cov}}

\usepackage{amsthm}
\usepackage{etoolbox}
\usepackage{algorithm}

\numberwithin{equation}{section}
\numberwithin{algorithm}{section}

\usepackage[capitalize,nameinlink]{cleveref}

\crefname{section}{section}{sections}
\crefname{subsection}{subsection}{subsections}
\Crefname{section}{Section}{Sections}
\Crefname{subsection}{Subsection}{Subsections}
\Crefname{figure}{Figure}{Figures}
\crefformat{equation}{\textup{#2(#1)#3}}
\crefrangeformat{equation}{\textup{#3(#1)#4--#5(#2)#6}}
\crefmultiformat{equation}{\textup{#2(#1)#3}}{ and \textup{#2(#1)#3}}
{, \textup{#2(#1)#3}}{, and \textup{#2(#1)#3}}
\crefrangemultiformat{equation}{\textup{#3(#1)#4--#5(#2)#6}}%
{ and \textup{#3(#1)#4--#5(#2)#6}}{, \textup{#3(#1)#4--#5(#2)#6}}{, and \textup{#3(#1)#4--#5(#2)#6}}
\Crefformat{equation}{#2Equation~\textup{(#1)}#3}
\Crefrangeformat{equation}{Equations~\textup{#3(#1)#4--#5(#2)#6}}
\Crefmultiformat{equation}{Equations~\textup{#2(#1)#3}}{ and \textup{#2(#1)#3}}
{, \textup{#2(#1)#3}}{, and \textup{#2(#1)#3}}
\Crefrangemultiformat{equation}{Equations~\textup{#3(#1)#4--#5(#2)#6}}%
{ and \textup{#3(#1)#4--#5(#2)#6}}{, \textup{#3(#1)#4--#5(#2)#6}}{, and \textup{#3(#1)#4--#5(#2)#6}}
\crefdefaultlabelformat{#2\textup{#1}#3}
\crefname{algorithm}{Algorithm}{Algorithms}

\theoremstyle{plain}
\newtheorem{theorem}{Theorem}[section]
\newtheorem{lemma}[theorem]{Lemma}
\newtheorem{corollary}[theorem]{Corollary}
\newtheorem{proposition}[theorem]{Proposition}
\newtheorem{definition}[theorem]{Definition}

\theoremstyle{remark}
\newtheorem{example}[theorem]{Example}
\newtheorem{remark}[theorem]{Remark}

\theoremstyle{plain}

\crefname{conjecture}{Conjecture}{Conjectures}
\crefname{fact}{Fact}{Facts}
\crefname{example}{Example}{Examples}
\crefname{remark}{Remark}{Remarks}
\crefname{assumption}{Assumption}{Assumptions}

\AddToHook{env/lemma/begin}{\crefalias{theorem}{lemma}}
\AddToHook{env/corollary/begin}{\crefalias{theorem}{corollary}}
\AddToHook{env/proposition/begin}{\crefalias{theorem}{proposition}}
\AddToHook{env/definition/begin}{\crefalias{theorem}{definition}}
\AddToHook{env/conjecture/begin}{\crefalias{theorem}{conjecture}}
\AddToHook{env/fact/begin}{\crefalias{theorem}{fact}}
\AddToHook{env/example/begin}{\crefalias{theorem}{example}}
\AddToHook{env/remark/begin}{\crefalias{theorem}{remark}}
\AddToHook{env/problem/begin}{\crefalias{theorem}{problem}}
\AddToHook{env/assumption/begin}{\crefalias{theorem}{assumption}}

\newenvironment{keywords}
  {\begin{quote}\small\noindent\textbf{Keywords.}\enspace}{\end{quote}}
\newenvironment{MSCcodes}
  {\begin{quote}\small\noindent\textbf{MSC codes.}\enspace}{\end{quote}}

\makeatletter
\newcommand{\preprint@appendix@seccnt}[1]{%
  \ifstrequal{#1}{section}
    {\appendixname\ \thesection.\quad}
    {\csname the#1\endcsname\quad}%
}
\g@addto@macro\appendix{\let\@seccntformat\preprint@appendix@seccnt}
\makeatother

\hypersetup{
  pdftitle={Approximating Measures on Function Spaces: Transport and Truncation},
  pdfauthor={Ricardo Baptista, Bamdad Hosseini, Alexander W. Hsu}
}
\graphicspath{{content/}}

\title{Approximating Measures on Function Spaces: Transport and Truncation}
\date{}
\author{Ricardo Baptista\thanks{Department of Statistical Sciences, University of Toronto, Toronto M5G 1X6, Canada}
\and Bamdad Hosseini\thanks{Department of Applied Mathematics, University of Washington, Seattle, USA}
\and Alexander W. Hsu\thanks{Courant Institute School of Mathematics, Computing, and Data Science
}}

\begin{document}
\maketitle

\begin{abstract}
Measures on function spaces arise throughout Bayesian inverse
problems and generative modeling, often with low-dimensional
structure relative to a tractable reference measure.
We introduce the class \(\psiclass\) of measures that differ from a
reference measure \(\RefM\) only through a finite-dimensional map
\(\psi\) while preserving the reference conditionals on its fibers.
Class members are determined by their \(d\)-dimensional
pushforwards under \(\psi\) and admit convenient block-triangular
transport map representations. Draws are taken from this
\(d\)-dimensional distribution and then completed to function
space through sampling of the reference conditionals.
For Gaussian references, these transport maps are finite rank
perturbations of the identity.
In contrast, optimal transport maps do not preserve this
low-dimensional structure. For covariance perturbations that are
trace class in the Cameron--Martin geometry of the reference, we show
the optimal map is a trace class perturbation of the identity.
Even finite rank perturbations yield corrections whose rank,
governed by an invariant subspace, is typically infinite.
We develop approximation theory for \(\psiclass\) and error
analysis for fitting within it, splitting the total error into an
irreducible class error and a finite-dimensional marginal term set by
the dimension of \(\psi\) rather than the ambient discretization.
We present numerical results including inference from
low-dimensional nonlinear observation maps, deconvolution under a
jump process prior, and state estimation for Navier--Stokes flows.

\end{abstract}



\section{Introduction}
In this work, we consider the problem of approximating measures on infinite-dimensional function spaces. These problems frequently arise in Bayesian approaches to inverse problems, spatial statistics, and imaging \cite{stuart2010inverse}, where the unknown quantities can be represented using functions. These approximations are also useful to build generative modeling algorithms, which aim to sample distributions, e.g., over high-resolution images and video, given only a limited set of samples from the distribution.
In both settings we represent the target distribution
as a low-dimensional transformation of a known reference measure.
We develop algorithms that modify a reference measure only through the distribution of a finite-dimensional observation \(w=\psi(u)\), while preserving the reference conditional law along the fibers of \(\psi\).

Let \(\U\) be a Polish space, and \(\RefM\) a probability measure on \(\U\), which we assume is known and computable. While this assumption is not always \emph{exactly} realizable computationally, it serves as a useful model for an idealized infinite-dimensional structure on \(\U\) computable to high accuracy. We will refer to \(\RefM\) as the \emph{reference measure}. Let \(\TargetM\) be a probability measure on \(\U\), which will be referred to as the \emph{target measure}. Our goal is to approximate \(\TargetM\) in the sense of building efficient algorithms for sampling from \(\TargetM\). We will also give special consideration to the case that we are targeting \emph{conditional measures}, \(\TargetM(u \mid y)\) where samples from the joint distribution \(\TargetM(y,u)\) are available and \(\TargetM(u \mid y)\) is approximated via simulation based inference.

For target distributions of the form \(\dd\TargetM \propto \exp(-\Phi(u))\dd\RefM\) with potential \(\Phi\),
we consider approximations with the structure
\begin{equation}\label{eqn:nuhat}
        \dd \widehat{\TargetM}(u)
    = \frac{1}{\widehat Z}\exp(-\phi(\psi(u)))\dd\RefM(u),
\end{equation}
where \(\psi:\mathcal{V}\to \mathbb{R}^d\) is a dimensionality reduction map that is continuous on a subset \(\V \subset \U\) upon which \(\RefM\) (and therefore also \(\TargetM\)) is concentrated.
\(\exp(-\phi \circ \psi(u))\) acts as a finite-dimensional modification to the infinite-dimensional reference measure \(\RefM\). Measures of this form satisfy
\begin{equation}
    \widehat{\TargetM}(\cdot \mid \psi(u) = w) = \RefM(\cdot \mid \psi(u) = w),
\end{equation}
which results in efficient methods for approximately sampling $\TargetM$.
In particular, for \(U \sim \widehat{\TargetM}\) and
\(W = \psi(U)\), we obtain that
\begin{align}
    \Law\left((W,U) \right)
    &= \widehat{\TargetM}_{\psi}(\dd w) \widehat{\TargetM}(\dd u \mid \psi(u) = w),\quad \widehat{\TargetM}_{\psi} := \psi \sharp \widehat{\TargetM} \nonumber \\
    &= \frac{1}{\widehat Z}\exp\left(-\phi(w)\right)\RefM_\psi(\dd w) \RefM(\dd u \mid \psi(u) = w),\quad \RefM_\psi = \psi \sharp \RefM. \label{eqn:two-stage}
\end{align}
The point of this decomposition is that if we can first sample \(W \sim \psi \sharp \widehat{\TargetM}\),
then by drawing a sample
\(U \sim \RefM(\dd u \mid \psi(u) = W)\), we obtain \(U \sim \widehat{\TargetM}\), which is intended to approximate \(\TargetM\). This second step only requires sampling from conditionals of the reference measure \(\RefM\), while preserving the infinite-dimensional structure of the target measure. This and other structural properties are discussed in further detail in \cref{sec:structure}.

For a reference measure \(\RefM\) and dimensionality reduction map \(\psi\), we define the space of such measures
\begin{equation}\label{eqn:psidef}
\psiclass :=
\left\{
\eta \in \mathcal P(\V)
:
\eta_\psi(\operatorname{supp}\RefM_\psi)=1,\quad
\eta(\cdot\mid \psi(u)=w)
=
\RefM(\cdot\mid \psi(u)=w)
\text{ for } \eta_\psi\text{-a.e. }w
\right\}.
\end{equation}

This formulation of dimension reduction includes projection-based subspace methods as one special case, but it also covers non-projection maps such as point evaluations, coarse-grid summaries, and sensor-derived observations. This flexibility is useful in function-space settings, where natural observations may be continuous only on a full-measure regularity class rather than on the ambient space in which error is measured.

\begin{example}
    A concrete example, considered in \cref{subsec:gp2}, is given by the distribution
\begin{equation}\label{eqn:gp2-intro}
    \nu(\dd u) \propto \exp\left(-\frac{1}{2\sigma^2}\sum_{i=1}^d (u(x_i)^2 - y_i)^2\right) \RefM(\dd u),
\end{equation}
where \(\RefM\) is a mean-zero Gaussian process with squared exponential covariance kernel \(k\), and \(x_i \in [0,1]\). This probability measure over functions is highly non-Gaussian and multimodal, but all of the deviation from the reference measure occurs in a low-dimensional subspace.
By taking \(\psi(u) = (u(x_1),\ldots,u(x_d))\), point evaluation at \((x_1,\ldots,x_d)\), we then rewrite \cref{eqn:gp2-intro} as
\begin{equation*}
    \nu(\dd u) \propto \exp\left(-\phi\circ \psi\right) \RefM(\dd u),\quad \phi(w) = \frac{1}{2\sigma^2}\sum_{i=1}^d (w_i^2 - y_i)^2.
\end{equation*}
Therefore \(\nu \in \psiclass\), and we have that
\begin{equation*}
    \nu_{\psi}(\dd w) \propto
    \exp\left(
    -\phi(w) - \frac{1}{2}w^T K^{-1} w
    \right),\quad K_{i,j} = k(x_i,x_j).
\end{equation*}
Draws \(U \sim \nu\) can be obtained by first sampling \(W\) from the
\(d\)-dimensional measure \(\nu_{\psi}\), then sampling the conditioned
Gaussian process \(U \sim \RefM(\cdot \mid u(x_i) = W_i)\).
In this case, we may take \(\V\) to be any reproducing kernel Hilbert space (RKHS) satisfying \(\RefM(\V) = 1\) to ensure that \(\psi\) is bounded.\footnote{The Cameron--Martin space \(\H\) associated to \(\RefM\) itself has \(\RefM(\H) = 0\) and does not satisfy this.} However, we may naturally wish to measure error on an ambient space such as \(L^2([0,1])\), where \(\psi\) is not bounded, which motivates our nesting of \(\V \subset \U\). We show illustrative results for a two-dimensional version built on two-point observations in \cref{fig:intro-gp2}, demonstrating the idea of sampling the distribution of \((U(0.25),U(0.75))\), before completing \(U\) to a function by sampling \(U \sim \RefM(\cdot \mid u(0.25) = W_1, u(0.75) = W_2)\).
\begin{figure}[htbp]
    \centering
    \includegraphics[width=\linewidth]{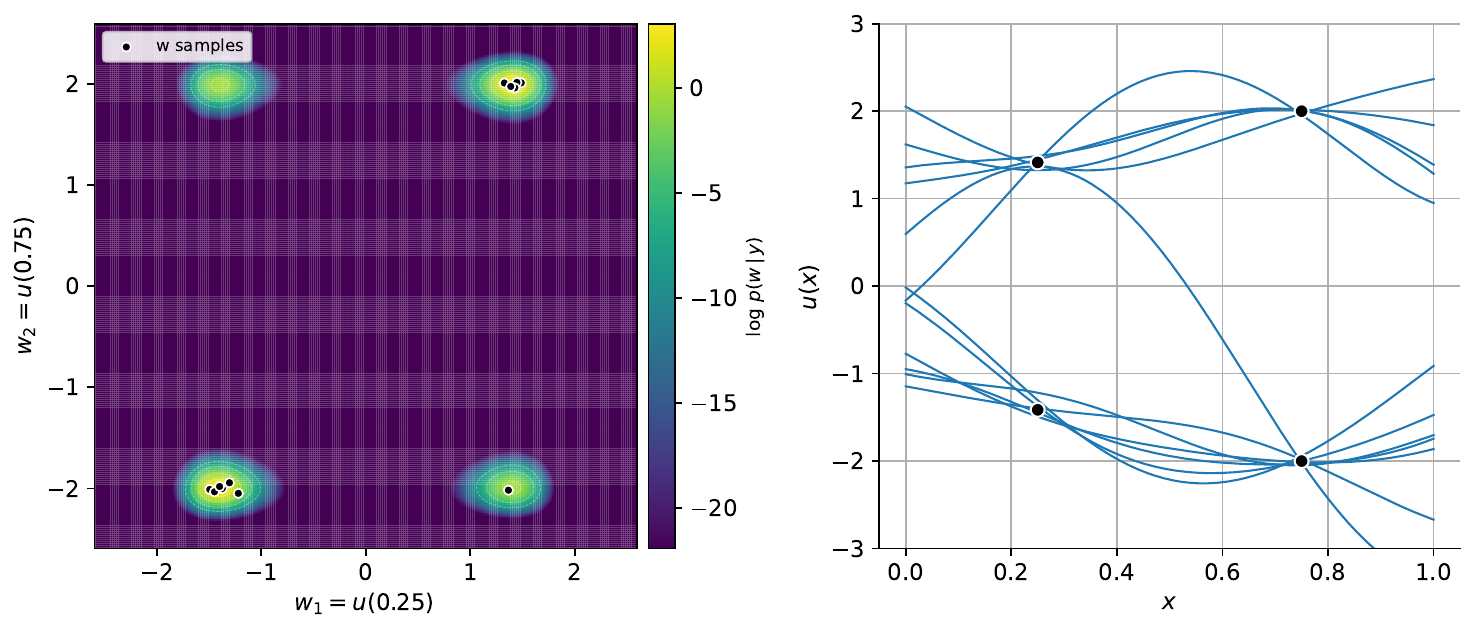}
    \caption{Left: the two-dimensional marginal density of \(W = (U(0.25), U(0.75))\) under the posterior, with samples \(W\) shown as black dots. Right: draws of \(U\) obtained by sampling the Gaussian process prior conditioned on the values \(u(0.25) = W_1\), \(u(0.75) = W_2\) drawn from the distribution on the left.}
    \label{fig:intro-gp2}
\end{figure}
\end{example}

The transport maps from \(\RefM\) to measures in \(\psiclass\) that underlie the sampling procedure \cref{eqn:two-stage} are structured in that they modify \(u\) only through \(\psi(u)\), moving mass along a fixed low-dimensional set of directions while preserving the reference conditional law along the fibers of \(\psi\). For Gaussian references,
these transport maps are finite rank perturbations of the identity.
This
structure is not shared by optimal transport (OT) maps. In the case of transport
maps between Gaussians, we show that for covariance perturbations that
are trace class in the Cameron--Martin geometry of the reference, the
OT map is also a trace class perturbation of the identity.
However, a finite rank perturbation of the
covariance does not necessarily lead to an OT map that is
also a finite rank perturbation of the identity.

\paragraph{Summary of contributions}
Our main contributions are as follows.
\begin{enumerate}
    \item We characterize the class \(\psiclass\) of measures through finite-dimensional representations of both their density with respect to \(\RefM\) (\cref{prop:cond-match}) and the transport maps from \(\RefM\) that generate them (\cref{thm:transport}).
    Within \(\psiclass\), we identify
    the best KL approximations of a given target measure (\cref{prop:best-approx}).

    \item We show that OT maps do not necessarily carry a low-rank structure for
    targets that are low-dimensional perturbations of
    a reference.
    Both maps can be written as the identity plus a correction.
    For Gaussian measures whose covariances differ by a low-rank
    perturbation, the OT correction may be infinite rank, while the
    two-stage correction remains finite rank.

    \item
    As a byproduct of contribution 2, we show a new and independently interesting result about
    OT maps between Gaussian measures. In particular, the OT map between a Gaussian reference and target measure is a trace class
    perturbation of the identity when the covariance of the target measure is a trace class perturbation relative to the Cameron--Martin norm
    (\cref{thm:ot-structure}).

    \item We introduce practical algorithms for approximating target measures by members
    of \(\psiclass\), where only the
    \(d\)-dimensional marginal $\nu_\psi$ is required.
    \(\nu_\psi\) can be approximated using generative models when samples from the
    target $\nu$ are available
    (\cref{alg:reduced-gen,alg:reduced-cond-gen}), and by variational
    inference when a likelihood can be evaluated (\cref{alg:vi}).

    \item We present error analysis quantifying the distance between an approximation \(\hat{\eta}\) and a target measure \(\TargetM\) under various assumptions, such as likelihood smoothness, conditional moments of the reference, and functional inequalities (\cref{sec:analysis}).

\end{enumerate}

\paragraph{Outline} In~\cref{sec:preliminaries}, we review related work, giving background on dimension reduction in generative modeling and Bayesian inference. In~\cref{sec:structure}, we present more details on the structure of measures in \(\psiclass\) based on their density and transport representations while contrasting these with properties of Gaussian OT maps. We present algorithms for computing these measures in
\cref{sec:methods}, and show numerical results and applications in \cref{sec:numerics}. Finally, in~\cref{sec:analysis}, we provide error analysis for the approximation by measures in \(\psiclass\). We collect proofs of our optimal transport results in \cref{app:ot-proofs}.

\paragraph{Notation}
For a probability measure \(\eta \in \mathcal{P}(\U)\) and a Borel function \(g\),
we write \(\eta(\cdot \mid g(u) = w)\) for any version of the regular conditional
measure associated to the \(\sigma\)-algebra generated by \(g\). For a
measurable map \(T\), \(T \sharp \eta := \eta \circ T^{-1} = \Law(T(X)), X\sim \eta\) denotes the
pushforward of \(\eta\). Throughout, \(\psi\)
denotes the dimensionality reduction map fixed by context, and we abbreviate
\(\eta^w := \eta(\cdot \mid \psi(u) = w)\) and \(\eta_\psi := \psi \sharp \eta\).
The two forms \(\eta^w\) and \(\eta(\cdot \mid \psi(u) = w)\) are used interchangeably. We refer to the pre-images of singleton sets \(\{w\}\) under \(\psi\) as the fibers of \(\psi\), \(\psi^{-1}(w):=\left\{u : \psi(u) = w\right\}\).

We write the disintegration of \(\eta\) along the fibers of \(\psi\) as
\(\eta(\dd u) = \eta_\psi(\dd w) \eta^w(\dd u)\),
which is shorthand for the identity
\(\eta(A \cap \psi^{-1}(B)) = \int_B \eta^w(A) \eta_\psi(\dd w)\)
holding for all Borel \(A \subset \U\) and \(B \subset \R^d\).
We use this product-of-differentials form freely to represent
integration against the marginal-then-conditional kernel.

Throughout, \(\RefM(\cdot\mid \psi(u)=w)\) denotes a fixed version of the reference regular conditional distribution. When \(\eta_\psi \ll \RefM_\psi\), this choice of version is immaterial. The advantage of fixing one is that it allows, for instance, marginals \(\psi\sharp \eta\) with atoms.

An operator \(A\) is positive when it is self-adjoint and
\(\langle Au, u\rangle \geq 0\) for every \(u\), and \(\preceq\)
denotes the Loewner order so that \(A \preceq B\) when \(B - A\) is
positive. We will write the trace norm and Hilbert--Schmidt norm of an
operator \(T\) as \(\|T\|_1 := \operatorname{tr}|T|\) and
\(\|T\|_{\mathrm{HS}} := (\operatorname{tr}T^*T)^{1/2}\), and \(T\) is
trace class, respectively Hilbert--Schmidt, when the corresponding
norm is finite. A Gaussian measure \(\N(m, \Sigma)\) is nondegenerate
when its covariance \(\Sigma\) is injective, and its Cameron--Martin
space is \(\Range(\Sigma^{1/2})\), equipped with the inner product
\(\langle \Sigma^{-1/2}u, \Sigma^{-1/2}v \rangle\). Throughout,
optimal transport (OT) refers to optimality under the quadratic cost
\(\|u - v\|^2\), with the norm determined by context and taken to be
that of \(\V\) unless stated otherwise.

\section{Background and related work}\label{sec:preliminaries}
Our methods are aimed at problems where the distribution of interest \(\TargetM\)
lies on a high- or infinite-dimensional space, a reference measure \(\RefM\) is available that captures its structural properties, and the departure of the target from that reference is low dimensional.
\Cref{subsec:lis} reviews dimension reduction in Bayesian inference, where an informative prior serves as a reference measure, and the data may cause only a low-dimensional departure from the prior.
\Cref{subsec:transport-gen} reviews transport maps and conditional generative modeling, which we apply throughout this work. \Cref{subsec:inf-dim} reviews variational inference and sampling on function spaces, where discretization invariance requires compatibility with the reference measure. Discussion of the literature on Gaussian OT maps accompanies our results in \cref{sec:structure}.

\subsection{Dimension reduction in Bayesian inference}\label{subsec:lis}

Gradient-based dimension reduction techniques are widely
used in active subspace methods
for forward uncertainty quantification
\cite{ConstantineActiveSubspace,constantine2015active}, where a
subspace is chosen using the leading eigenvectors of the averaged
gradient outer product \(\EE[\nabla f\, \nabla f^*]\) of a quantity
of interest \(f\).

Likelihood-informed subspaces (LIS) carry this idea to Bayesian
inversion \cite{lis}. With the posterior written in the form \(\dd\TargetM \propto
e^{-\Phi}\dd\RefM\) on \(\R^n\), \(u\) is split as
\(u = u_d + u_\perp\) along a subspace \(\W\) and its
orthogonal complement \(\W^\perp\). The
approximation keeps the prior conditional distribution on
\(\W^\perp\) and
performs inference only on the \(\W\)-marginal,
\begin{equation}\label{eqn:lis-form}
    \overline{\TargetM}(\dd u)
    = \overline{\TargetM}_d(\dd u_d)\, \RefM(\dd u_\perp \mid u_d),
\end{equation}
where \(\overline{\TargetM}_d\) is a surrogate for the marginal
posterior. In \cite{lis}, the complement subspace \(\W^\perp\) is chosen orthogonal with respect to the
inner product of the prior precision, so that the two components are
independent under the Gaussian prior and the conditional in
\cref{eqn:lis-form} is simply the prior marginal on the complement.
In \cite{certified,cui-LIS}, the approximation is directly written
as a density with respect to the prior depending only on the projected
state, which preserves the prior conditionals for any projector. Cui
et al.\ \cite{lis} use the prior-preconditioned Gauss--Newton Hessian of the
data misfit averaged over the posterior to choose \(\W\), and explore the
reduced posterior with Markov chain Monte Carlo (MCMC).
Under the assumption of a subspace log-Sobolev inequality,
 Zahm et al.\ \cite{certified} obtain error bounds in
Kullback--Leibler divergence for approximations of this form, using
the conditional expectation of the likelihood under the prior as a
low-dimensional surrogate. These bounds are based on the residual of
a prior-weighted low-rank approximation to the expected gradient
outer product matrix \(\EE_\TargetM[\nabla\Phi\, \nabla\Phi^*]\), and
minimizing this residual over subspaces reduces to a generalized
eigenproblem.
Under Poincar\'e inequalities on the reference conditionals, Cui and
Tong \cite{cui-LIS} give a unified analysis of gradient-based methods
of this type, treating the subspace truncation error together with
the errors from estimation of the diagnostic matrix and approximate
marginalization. Their bounds are in Hellinger distance and
Kullback--Leibler divergence, and cover a variety of surrogates and
choices of subspace. In the linear Gaussian setting, optimal low-rank
approximations of the posterior covariance on Hilbert spaces are
characterized in \cite{carere2024optimal}. We relate our error
analysis to results of this kind in \cref{subsec:functional}.

Measures of the form \cref{eqn:lis-form} are members of \(\psiclass\).
Let \(v_1, \dots, v_d\) be an orthonormal basis for \(\W\) and set
\(\psi(u) = (\langle v_1, u\rangle, \dots, \langle v_d, u\rangle)\)
and \(\psi^\dagger(w) = \sum_{i=1}^d w_i v_i\),
so that \(\ker\psi = \W^\perp\) and \(\psi^\dagger \circ \psi = P_\W\).
The measure \(\overline{\TargetM}\) in \cref{eqn:lis-form} keeps the reference conditionals on the
fibers of \(\psi\) and modifies only the \(\psi\)-marginal, which is
the defining property of \(\psiclass\). When the surrogate has a
density, \cref{prop:cond-match} gives the representation
\(\overline{\TargetM}(\dd u) \propto e^{-\phi\circ\psi(u)}\dd\RefM(u)\).
Within this class of measures, the optimal choice in
forward Kullback--Leibler divergence is the marginalization of
\cref{prop:best-approx}. Only the \(\sigma\)-algebra generated by
\(\psi\) matters here, so any projector and basis serve. In finite
dimensions and for linear \(\psi\), \(\psiclass\) is
precisely the class of approximations used by LIS methods.
For Gaussian references, the structure of such conditioned measures is
developed in \cite{chen2025gaussian}, which studies a Gaussian measure
conditioned on noisy nonlinear observations of finitely many bounded
linear functionals, exactly the class members with linear \(\psi\) and a
likelihood of the form \(F \circ \psi\). Their representer theorem
decomposes the conditioned measure into an explicitly identified
infinite-dimensional Gaussian and a finite-dimensional non-Gaussian
factor, and they develop small-noise consistency, conditional maximum a posteriori (MAP)
estimators, and simulation methods that focus computational effort on
the finite-dimensional factor, by MCMC or by Laplace and Gauss--Newton
approximations at the MAP. In our terminology, the decomposition is the
sampling characterization of \cref{prop:sampling} with the Gaussian
conditional \cref{eq:gauss-cond}, and the finite-dimensional factor is
the marginal that our algorithms fit with generative models and
transport maps.

We study this class in the function-space setting with a view
towards generative modeling, consider \(\psi\) beyond orthogonal
projections, and fit the marginal with a generative model or a
transport map, avoiding explicit marginalization of the likelihood.

\subsection{Transport maps and generative models}\label{subsec:transport-gen}
Measure transport methods represent a target measure as the
pushforward of a tractable reference, giving one approach to
generative modeling and a convenient parameterization of
measures by functions \cite{marzouk2016sampling}. Triangular maps of Knothe--Rosenblatt type
play a special role because their structure encodes conditioning
\cite{bogachev2005triangular,friendly-triangular}. In particular,
block-triangular maps that push a reference to a joint law produce
exact samplers for its conditionals \cite{mgan}. \cref{thm:transport}
is an instance of this principle for \(\psiclass\), where the
conditionals are imputed by the reference
measure to preserve infinite-dimensional structure.

A related line of work studies conditional OT,
examining triangular maps through the lens of transport costs.
Conditional Brenier maps were introduced as a notion of multivariate
conditional quantiles in \cite{carlier2016vector}. Further,
triangular maps themselves arise as the limit of OT
problems with quadratic costs whose coordinate weights successively
dominate \cite{carlier2010knothe,bonnotte2013knothe}. Data-driven formulations of the
general problem appear in \cite{tabak2021conditional}, and the theory
has since been developed on function spaces
\cite{hosseini2025conditional}, with work on the computation of these maps in
\cite{manupriya2024consistent,baptista2025conditionalSimulation,baptista2025knothe,wang2025conditional,dynamic_cot}.
Chemseddine et al.\ \cite{chemseddine2025conditional} further study
conditional Wasserstein distances by restricting to couplings
compatible with the conditioning variable. These quantities evaluate
the expected Wasserstein distance between the conditionals, the same
quantity that appears in our bounds of \cref{subsec:wasserstein}.

Conditional generative models can be applied to solve Bayesian inference problems using the framework of simulation-based inference. In simulation-based inference, a Bayesian model is accessible through draws from a joint distribution \(\TargetM(y, u)\) of data and parameter, and the model may even be defined through the sampling procedure itself. A conditional generative model is trained on these joint draws \((Y, U) \sim \TargetM\) and samples the posterior \(\TargetM(u \mid y)\) directly \cite{mgan,hosseini2025conditional,dynamic_cot}. All of the numerical examples in this paper operate in this setting, which is addressed by \cref{alg:reduced-cond-gen}.

The marginal models in \cref{alg:reduced-gen,alg:reduced-cond-gen}
are ordinary generative models on \(\R^d\), transports fitted from
marginal samples. Our experiments use flow
matching and stochastic interpolants
\cite{lipman2023flow,albergo2025stochastic,dynamic_cot,functional_flow}, for which
quantitative error bounds are available \cite{benton2023error}.
Normalizing flows \cite{rezende2015variational} and diffusion models
\cite{sohl-dickstein2015noneq,ho2020ddpm,song2019score,song2021score_sde}
would be similarly applicable, since any Wasserstein or Kullback--Leibler
guarantee for the marginal model transfers to function space through
the results of~\cref{sec:analysis}.

When the likelihood can be evaluated, the transport may instead be
fit variationally by minimizing the reverse Kullback--Leibler
divergence over a parameterized family of maps
\cite{moselhy2012bayesian,marzouk2016sampling}. \cref{alg:vi} outlines
this approach in our dimension-reduced setting, with the transport parameterized through its action on
the \(\psi\)-marginal.

A parallel literature builds generative models directly on function
space. Diffusion models have been formulated on Hilbert spaces with
trace class noise
\cite{inf_dim_diffusion,score_function_space,franzese2023functional}.
Well-posedness holds, for instance, when the data measure has a density
with respect to a Gaussian reference, and \cite{inf_dim_diffusion}
proves dimension-independent convergence rates. 
These methods seek to learn the
entire measure, including its infinite-dimensional structure. In our work, the infinite-dimensional structure is
carried exactly by the reference conditionals and the learning task is reduced to a problem on
\(\R^d\) at the cost of an irreducible class error quantified in
\cref{sec:analysis}.

A related family builds transports from maps that act through finitely
many directions. Lazy maps \cite{lazy-map} compose transports confined
to low-dimensional subspaces, each chosen by minimizing the bound of
\cite{certified}, and iterate on the pullback of the target. A single
such step produces a member of \(\psiclass\), and we make this
connection precise in \cref{subsec:vi}. Functional normalizing flows
\cite{func_norm_flow} compose finite rank perturbations of the
identity whose ranges lie in the Cameron--Martin space of the prior,
a constraint that keeps each layer equivalent to the prior with a
computable density. Transports of this type, perturbations of the
identity with Cameron--Martin-valued displacement, exist for any
target absolutely continuous with respect to the Gaussian reference
\cite{bogachev2005triangular}. We show in \cref{rem:finite-rank-layers} that such
layers are again transports into \(\psiclass\), for an enlarged choice
of \(\psi\).

\subsection{Variational inference and sampling on function spaces}\label{subsec:inf-dim}

Variational inference approximates an intractable target
measure \(\TargetM\) by
minimizing a divergence over a chosen family
\cite{jordan1999introduction,blei2017variational}. On function spaces,
the family must be compatible with the reference measure
\(\RefM\) for the divergence to remain finite.
Theory for best approximation by Gaussian measures on function spaces with respect to Kullback--Leibler divergence is established in \cite{pinski2015kullback} and corresponding algorithms are developed in \cite{pinski2015algorithms}. The approximation
studied in this paper is variational inference over the structural
family \(\psiclass\). Its best elements
are characterized exactly in \cref{prop:best-approx}, and
\cref{alg:vi} minimizes the reverse divergence within the class.

Markov chain Monte Carlo on function space requires similar
compatibility with the reference measure, as samplers designed in
finite dimensions often degrade as the discretization is refined.
Proposals that preserve the Gaussian reference
measure, such as the preconditioned Crank--Nicolson method, are well
defined in the infinite-dimensional limit and remain stable under mesh
refinement \cite{cotter2013mcmc}.

\section{Low-dimensional modifications of a reference measure}\label{sec:structure}
In this section, we derive fundamental properties of the class of measures \(\psiclass\) defined in \cref{eqn:psidef}, i.e., measures \(\eta\) on \(\V\)
whose conditionals along the fibers of \(\psi\) agree with those of the reference measure \(\RefM\). In \cref{subsec:character}, we characterize \(\psiclass\) from three perspectives, through sampling, densities, and measure transport. In \cref{subsec:gauss-transport}, we specialize to Gaussian reference measures, where the transports can be written explicitly as finite rank perturbations of the identity.

In \cref{subsec:approx}, we characterize the best approximation of a measure \(\TargetM\) by a measure in \(\psiclass\).

\subsection{Three characterizations of \texorpdfstring{$\psiclass$}{P-psi(mu)}}\label{subsec:character}
\cref{prop:sampling} gives a sampling perspective, \cref{prop:cond-match}
identifies members of $\psiclass$ via their Radon--Nikodym derivative with
respect to \(\RefM\), and \cref{thm:transport} exposes a block-triangular
map structure underlying measures in \(\psiclass\). Each characterization shows that elements of $\psiclass$ are effectively finite dimensional in relation to \(\RefM\).

\begin{proposition}[Sampling perspective]\label{prop:sampling}
    Let \(\eta \in \psiclass\), and let \(\eta_\psi = \psi \sharp \eta\).
    If
    \begin{equation}
        W \sim \eta_\psi, \text{ and } U \sim \RefM(\cdot \mid \psi(u) = W),
    \end{equation}
    then \( \Law(U) = \eta\).
\end{proposition}
\begin{proof}
    The joint distribution \(\Law((W,U)) = \eta_\psi(\dd w)\RefM(\dd u \mid \psi(u) = w)\).
    Because
    \(\RefM(\dd u \mid \psi(u) = w) = \eta(\dd u \mid \psi(u) = w)\) \(\eta_\psi\)-almost everywhere,
    \(\Law((W,U)) = \eta_\psi(\dd w)\eta(\dd u \mid \psi(u) = w) = (\psi,\mathrm{Id}) \sharp \eta\).
    Hence, \(\Law(U) = \eta\).
\end{proof}

This characterization amounts to the fact that any measure \(\eta \in \psiclass\) is determined by the marginal distribution \(\psi \sharp \eta\), and outlines a simple procedure for sampling, which is the main motivation for using \(\psiclass\) in practice.

Next, we show that under an additional absolute continuity assumption on $\psi \sharp \eta$,
membership in \(\psiclass\) is equivalent to $\eta$ having a density with respect to $\RefM$ that
depends only on $\psi(u)$. In this sense, \(\eta\) is finite dimensional in relation to \(\RefM\).

\begin{proposition}[Density perspective]\label{prop:cond-match}
    Let \(\eta \in \mathcal{P}(\V)\), and assume that \(\psi \sharp \eta \ll \psi \sharp \RefM\).
    Then \(\eta \in \psiclass\) if and only if there exists a measurable function
    \(\phi: \R^d \to \R \cup \{+\infty\}\), with
    \(\int \exp(-\phi \circ \psi(u)) \dd \RefM(u) \in (0,\infty)\),
    such that
    \begin{equation}\label{eqn:density-rep}
        \dd \eta(u) = \frac{1}{Z}\exp(-\phi \circ \psi(u))\dd\RefM(u),
        \quad Z := \int \exp(-\phi \circ \psi(u))\dd\RefM(u) .
    \end{equation}
\end{proposition}
\begin{proof}
    First, assume \(\eta\) has a density representation given by \cref{eqn:density-rep}.
    Let \(\RefM^w(\dd u), \eta^w(\dd u)\) be disintegrations of \(\RefM, \eta\) respectively
    along the fibers \(\psi(u) = w\). Then
    \(\eta(A \cap \psi^{-1}(B)) = \int_{B}\eta^w(A)\eta_\psi(\dd w)\),
    where \(\eta_\psi = \psi \sharp \eta\).
    Using the density that defines \(\eta\) and the fact that
    \(\eta_\psi(\dd w) = \frac{1}{Z}\exp(-\phi(w))\RefM_\psi(\dd w)\), we have
    \begin{align*}
        \eta(A \cap \psi^{-1}(B))
        &= \frac{1}{Z}\int_{A\cap \psi^{-1}(B)}\exp(-\phi \circ \psi(u))\dd\RefM(u) \\
        &= \frac{1}{Z}\int_{B}\RefM^w(A)\exp(-\phi(w))\RefM_\psi(\dd w),
        \quad \RefM_\psi = \psi \sharp \RefM \\
        &= \int_{B}\RefM^w(A)\eta_\psi(\dd w).
    \end{align*}
    Thus
    \(\int_{B}\eta^w(A)\eta_\psi(\dd w) = \int_{B}\RefM^w(A)\eta_\psi(\dd w)\),
    and therefore \(\RefM^w = \eta^w\) a.e.\ with respect to \(\eta_\psi\),
    showing that \(\eta \in \psiclass\).

    For the opposite direction, assume that \(\eta \in \psiclass\),
    so that
    \(\eta(\cdot \mid \psi(u) = w) = \RefM(\cdot \mid \psi(u) = w)\)
    \(\eta_{\psi}\)-almost everywhere.
    Disintegrating \(\eta\) along the fibers of \(\psi\), we have
    \begin{equation}
        \eta(A \cap \psi^{-1}(B)) = \int_{B}\eta^w(A) \eta_\psi(\dd w)
        = \int_{B}\RefM^w(A) \eta_\psi(\dd w).
    \end{equation}
    By the assumption that \(\psi \sharp \eta \ll \psi \sharp \RefM\) and the Radon--Nikodym theorem,
    there exists a measurable \(\phi\) (possibly extended valued) such that
    \(\eta_\psi(\dd w) = \frac{1}{Z}\exp(-\phi(w))\RefM_\psi(\dd w)\),
    which gives
    \begin{equation}
        \eta(A \cap \psi^{-1}(B)) = \frac{1}{Z}\int_{B}\RefM^w(A) \exp(-\phi(w))\RefM_\psi(\dd w)
        = \frac{1}{Z}\int_{A \cap \psi^{-1}(B)} \exp(-\phi\circ\psi(u))\dd\RefM(u),
    \end{equation}
    establishing the density representation \cref{eqn:density-rep}.
\end{proof}

We now consider transport maps from the reference measure \(\RefM\) to
measures in \(\psiclass\).
The sampling representation of \cref{prop:sampling} gives such maps a natural
block-triangular structure \cite{hosseini2025conditional,mgan}.
For a target \(\eta \in \psiclass\), the map consists of a finite-dimensional
transport between \(\RefM_{\psi}\) and \(\eta_{\psi}\), composed with an
infinite-dimensional map, determined by \(\psi\) and \(\RefM\), that carries
one reference conditional to another.

\begin{theorem}\label{thm:transport}
    Let \(\eta \in \psiclass\).
    Let \(F:\R^d \to \R^d\) be a map such that \(F\sharp\RefM_\psi = \eta_\psi\). Let \(G: \R^d \times \R^d \times \U \to \U\) be jointly measurable, with each \(G(v,w,\cdot): \U \to \U\) a transport map
     such that
     \begin{equation}\label{eq:G-cond}
         G(v,w,\cdot)\sharp\RefM(\cdot\mid \psi(u) = w)=\RefM(\cdot \mid \psi(u) = v).
     \end{equation}
     Then with \(T(u) = G(F(\psi(u)),\psi(u),u)\), we have \(T\sharp \RefM = \eta\).
\end{theorem}
\begin{remark}
    \(G\) does not depend at all on \(\eta\) or \(F\), and is a property of the reference measure, mapping one conditional of the reference measure to another.
\end{remark}
\begin{proof}
    This proof is nearly identical to those showing that constructions of (block) triangular transport maps are correct; see, for example, \cite[Theorem~2.4]{mgan}. Let \(f:\U \to \mathbb{R}\) be a continuous test function. Then
    \begin{align}
        \int_{u\in \U}f(u)(T\sharp\RefM)(\dd u)
        &= \int_{u\in \U} f(G(F(\psi(u)),\psi(u),u)) \RefM(\dd u)\nonumber\\
        &=\int_{w\in \R^d} \int_{u\in \U} f(G(F(w),w,u)) \RefM^{w}(\dd u) \RefM_{\psi}(\dd w) \nonumber\\
        &=\int_{w\in \R^d} \int_{u\in \U} f(u) \RefM^{F(w)}(\dd u) \RefM_{\psi}(\dd w) \label{eq:fw}\\
        &=\int_{w\in \R^d} \int_{u\in \U} f(u) \RefM^{w}(\dd u) (F\sharp\RefM_{\psi})(\dd w)\nonumber\\
        &=\int_{\U} f(u)\eta(\dd u)\label{eq:fsharp},
    \end{align}
    where we used \cref{eq:G-cond} in \cref{eq:fw}, and in \cref{eq:fsharp} that \(F\sharp \RefM_\psi = \eta_\psi\) together with \(\RefM^w = \eta^w\) for \(\eta_\psi\)-almost every \(w\), since \(\eta \in \psiclass\).
    \end{proof}

\subsection{Explicit transports for Gaussian references}\label{subsec:gauss-transport}
When the reference measure is Gaussian and \(\psi\) is linear, simple maps satisfying
\cref{thm:transport} can be written explicitly. Let \(\RefM = \N(m, \Sigma_{uu})\)
be a Gaussian measure on \(\V\) with injective covariance operator
\(\Sigma_{uu}\), let \(\psi:\V\to\R^d\) be a surjective bounded linear map, and
write \(\Sigma_{u\psi} := \Sigma_{uu}\psi^*\),
\(\Sigma_{\psi u} := \Sigma_{u\psi}^* = \psi\Sigma_{uu}\), and
\(\Sigma_{\psi\psi} := \psi\Sigma_{uu}\psi^*\), which is invertible. The
conditionals of \(\RefM\) are again Gaussian, with
\begin{equation}\label{eq:gauss-cond}
    \RefM(\cdot\mid \psi(u)=w)
    = \N\!\left(
        m + \Sigma_{u\psi}\Sigma_{\psi\psi}^{-1}(w - \psi(m)),
        \Sigma_{uu} - \Sigma_{u\psi}\Sigma_{\psi\psi}^{-1}\Sigma_{\psi u}
    \right).
\end{equation}
A classical property is that the \emph{conditional covariance
does not depend on the conditioning value} \(w\), so the shift matching the
conditional means,
\begin{equation}\label{eq:matheron}
    G(v,w,u) = u + \Sigma_{u\psi}\Sigma_{\psi\psi}^{-1}(v - w),
\end{equation}
satisfies \cref{eq:G-cond}. In particular, when \(U \sim \RefM\), we have
\(\Law(G(v,\psi(U),U)) = \RefM(\cdot \mid \psi(u) = v)\). This sampling
procedure for conditionals is often referred to as Matheron's update
\cite{pathwise_gp}. Given an efficient algorithm for generating and representing
prior samples, Matheron's update produces samples conditioned on linear
measurements by applying a sample-dependent shift; we use this for
conditioned sampling in \cref{subsubsec:gauss-cond}.

With this choice of \(G\), the map of \cref{thm:transport} becomes
\begin{equation}\label{eqn:gauss-TF}
    T_F(u) = u + \Sigma_{u\psi}\Sigma_{\psi\psi}^{-1}\left(F(\psi(u)) - \psi(u)\right),
\end{equation}
the identity plus a perturbation acting along the fixed directions
\(\Range(\Sigma_{u\psi})\), with coefficients depending only on \(\psi(u)\).
The transport is therefore as low dimensional as the modification it
implements, moving mass only in a fixed \(d\)-dimensional subspace.
Since \(\RefM_\psi\) is a nondegenerate Gaussian, a map \(F\) with
\(F\sharp\RefM_\psi = \eta_\psi\) exists for every \(\eta\in\psiclass\), so
every measure in \(\psiclass\) is the pushforward of \(\RefM\) under a map of
this form. Conversely, for Gaussian references, any perturbation along
the directions \(\Range(\Sigma_{u\psi})\) with coefficients depending on \(\psi(u)\)
pushes \(\RefM\) into \(\psiclass\).

\begin{corollary}[Finite rank perturbations of the identity]\label{cor:finite-rank}
    Let \(\RefM = \N(m,\Sigma_{uu})\) be as above, let \(\psi:\V\to\R^d\) be a
    surjective bounded linear map, and let \(\rho:\R^d\to\R^d\) be measurable.
    Then the map
    \begin{equation*}
        T(u) = u + \Sigma_{u\psi}\rho(\psi(u))
    \end{equation*}
    satisfies \(T\sharp\RefM \in \psiclass\).
\end{corollary}
\begin{proof}
    Since \(\psi(T(u)) = \psi(u) + \Sigma_{\psi\psi}\rho(\psi(u))\), defining
    \(F(w) := w + \Sigma_{\psi\psi}\rho(w)\) gives
    \(\psi\circ T = F\circ\psi\). Substituting
    \(\rho(\psi(u)) = \Sigma_{\psi\psi}^{-1}\left(F(\psi(u)) - \psi(u)\right)\)
    into the definition of \(T\) shows that \(T = T_F\) in
    \cref{eqn:gauss-TF}, that is, \(T(u) = G(F(\psi(u)),\psi(u),u)\) with
    \(G\) the update \cref{eq:matheron}. \cref{thm:transport} then identifies
    \(T\sharp\RefM\) as the measure in \(\psiclass\) with \(\psi\)-marginal
    \(F\sharp\RefM_\psi\).
\end{proof}

\begin{remark}[Choice of space and Cameron--Martin directions]\label{rem:adapted-spaces}
    The transport of \(\RefM\) under a finite rank perturbation of the
    identity can admit a density with respect to \(\RefM\) if and only if
    the perturbation directions lie in the Cameron--Martin space
    \(\Range(\Sigma_{uu}^{1/2})\) \cite{bogachev2015gaussian}. When the space \(\V\) may be chosen for
    the transport at hand, this is in turn equivalent to the setting of
    \cref{cor:finite-rank}, so the boundedness of \(\psi\) costs no
    generality. For any finite collection of such directions, a suitable
    space can be built by grouping the Karhunen--Lo\`eve expansions of the
    directions into blocks with summable variances, which makes the
    associated conditioning functionals bounded. The adaptation of \(\V\)
    to the transport directions would be necessary in the most general case, as any Hilbert space \(\V\) whose dual contains the entire dual of the Cameron--Martin space would itself have to be contained in
    the Cameron--Martin space, which has \(\RefM\)-measure zero. Without
    continuity on some \(\V\), the map \(\psi\) is defined only up to
    \(\RefM\)-null sets, so that its fibers and the reference conditionals along them are no longer necessarily defined in a pointwise sense.
\end{remark}

\begin{remark}[Compositions and layered transports]\label{rem:finite-rank-layers}
    The maps of \cref{cor:finite-rank} are closed under composition. For
    layers with maps \(\psi_1\) and \(\psi_2\), the value
    \(\psi_2(T_1(u))\) is a function of \((\psi_1(u),\psi_2(u))\), so the
    composition is again of the same form for the concatenated map. The
    same concatenation covers perturbations whose directions and
    coefficient functionals differ. Transports layered in this way appear in the
    lazy maps of \cite{lazy-map}, which compose maps \cref{eqn:gauss-TF} and
    so produce measures in the class of their concatenated subspaces, and in
    the functional normalizing flows of \cite{func_norm_flow}. There, a
    typical layer is the planar map
    \(T(u) = u + v_1\tanh(\langle w_1, u\rangle + b_1)\), which is the case
    \(\psi = (\langle w_1,\cdot\rangle, \langle z_1,\cdot\rangle)\) with
    \(\Sigma_{uu}z_1 = v_1\) when \(v_1 \in \Range(\Sigma_{uu})\). For
    any other direction \(v_1\) in the Cameron--Martin space of the prior,
    the same holds after adapting the space \(\V\), as discussed in \cref{rem:adapted-spaces}.
    In either case, a flow of \(k\) such layers produces a measure in
    \(\psiclass\) for a \(\psi\) of dimension at most \(2k\).
\end{remark}

In contrast to the maps described in this section, OT maps
do not, in general, retain low-dimensional structure.
Conversely, the maps of this section will generally neither recover nor approximate
the OT map from \(\RefM\) to \(\eta\) without further assumptions on the compatibility of \(\psi\), \(\RefM\), and the norm used to measure the OT cost.
This holds even when the marginal transport from \(\RefM_\psi\) to \(\eta_\psi\) is itself optimal and \(\eta \in \psiclass\).
From a computational point of view, maps of the form in \cref{eqn:gauss-TF} are much more convenient than OT maps for sampling measures in \(\psiclass\). They require significant computation only on the reduced space, along with a way to apply the covariance operator of the reference measure. In contrast, OT maps will generally mix the spaces together and require significant computation on the full space unless the reference distribution \(\RefM\) is an isotropic Gaussian, or the transport cost is taken to be the squared Cameron--Martin norm associated to a reference Gaussian.

Consider transport maps of the form \(T(u) = Au\) from a centered Gaussian measure with covariance \(\Sigma\) to another centered Gaussian measure with covariance \(\Sigma'\). These maps must satisfy the equation \(A \Sigma A^* = \Sigma'\), which has many solutions. The two-stage structure of the maps of this section imposes that \(A\) is block lower-triangular with variables corresponding to \(\psi\) ordered first, while OT maps impose that \(A\) is positive self-adjoint \cite{knott1984optimal,cuesta1996lower}, leading to possibly very different structures. For Gaussian measures on infinite-dimensional spaces, these maps and
the transport geometry they induce on covariance operators are
developed in \cite{masarotto2019procrustes}. The maps are in general
unbounded operators, and assumptions leading to boundedness are
developed in
\cite{yun2025gaussian,masarotto2024transportation,santoro2025large}.
At the level of transport distances rather than maps,
\cite{masarotto2019procrustes,minh2022finite} establish
convergence of certain
finite-dimensional approximations of the Wasserstein distance between
Gaussian processes.

For the OT map \(T(u) = Au\) when \(\Sigma' = \Sigma + \Gamma\), we
relate the size and support of the correction \(S := A - I\) to those of
the perturbation \(\Gamma\). We show with \cref{thm:ot-structure} that when
\(\Gamma\) is trace class in the Cameron--Martin geometry of the
reference, meaning \(\Gamma = \Sigma^{1/2}H\Sigma^{1/2}\) with \(H\)
trace class, the correction \(S\) is trace class with its trace norm
controlled by that of \(H\), and derive an integral representation for
\(S\). \cref{thm:ot-rank} uses this representation to
characterize the range and rank of \(S\) for \emph{finite rank}
perturbations of the covariance. The proofs of both theorems are
given in \cref{app:ot-proofs}.

\begin{theorem}[Structure of Gaussian optimal transport]\label{thm:ot-structure}
    Let \(\RefM = \N(0, \Sigma)\) be a nondegenerate Gaussian measure
    on \(\V\), and let \(\eta = \N(0, \Sigma')\) with
    \begin{equation*}
        \Sigma' = \Sigma + \Sigma^{1/2} H \Sigma^{1/2}
    \end{equation*}
    for a nonzero self-adjoint trace class operator \(H\) such that
    \(I + H\) is injective. Then \(\eta\) is a nondegenerate Gaussian
    measure equivalent to \(\RefM\), there is a unique bounded positive self-adjoint operator \(A\) with \(A\Sigma A = \Sigma'\), and
    \(T(u) = Au\) is the OT map from \(\RefM\) to
    \(\eta\) for the quadratic cost.

    Moreover, with
    \(R_t := (\Sigma^2 + t^2)^{-1}\) and \(G_t := \Sigma R_t \Sigma\),
    \begin{equation}\label{eqn:ot-structure-int}
        A - I = \frac{2}{\pi}\int_0^\infty t^2
        R_t\Sigma^{1/2} H(I + G_t H)^{-1} \Sigma^{1/2}R_t
        \dd t
    \end{equation}
    and
    \begin{equation}\label{eqn:norm-bound}
        \|A - I\|_1 \leq \frac{\kappa}{2} \|H\|_1, \quad \kappa:= \max\left(1,\|(I + H)^{-1}\|\right) < \infty.
    \end{equation}
\end{theorem}

\begin{theorem}[Gaussian optimal transport under finite rank covariance perturbations]\label{thm:ot-rank}
    Let \(\RefM = \N(0, \Sigma)\) be a nondegenerate Gaussian measure
    on \(\V\), let \(\psi:\V\to\R^d\) be a surjective bounded linear map,
    and let \(\eta = \N(0, \Sigma')\) be nondegenerate, with
    \(\Sigma' = \Sigma + \Sigma_{u\psi}C\Sigma_{\psi u}\) for a
    nonzero symmetric matrix \(C\).\footnote{By \cref{prop:cond-match} and the
    Woodbury identity, measures \(\eta\) of this form are the centered Gaussian
    members of \(\psiclass\) equivalent to \(\RefM\). }
    Then the OT map from $\mu$ to $\eta$
    takes the form \(T(u) = Au\) where \(A\) is the unique bounded positive self-adjoint operator satisfying \(A \Sigma A = \Sigma'\). Further, let \(\mathcal{K}\)
    be the smallest closed $\Sigma$-invariant subspace containing $\Range(\psi^*)$.
    Then the following statements hold:
    \begin{enumerate}
        \item \(A - I\) vanishes on \(\mathcal{K}^\perp\) and maps into
        \(\mathcal{K}\), so
        \(\operatorname{rank}(A - I) \leq \dim\mathcal{K}\);
        \item if \(C\) is positive or negative definite, then the closure
        of the range of \(A - I\) is \(\mathcal{K}\), so
        \(\operatorname{rank}(A - I) = \dim\mathcal{K}\);
        \item \(\operatorname{rank}(A - I) \geq
        \frac{1}{2}\operatorname{rank} C\);
        \item if \(\Range(\psi^*)\) is \(\Sigma\)-invariant and \(\psi\psi^* = I\), we have \(A = T_F\) as in
        \cref{eqn:gauss-TF}, with \(F\) the OT map between
        the marginals \(\RefM_\psi\) and \(\eta_\psi\).
    \end{enumerate}
\end{theorem}
\begin{proof}[Proof sketch]
    The finite rank perturbation satisfies the hypotheses of
    \cref{thm:ot-structure}, and specializing the representation
    \cref{eqn:ot-structure-int} through the push-through identity
    collapses each integrand to rank at most \(d\), so that
    \begin{equation*}
        S = \frac{2}{\pi}\int_0^\infty t^2
        \left(R_t\Sigma\psi^*\right) W_t^{-1} \left(\psi\Sigma R_t\right)\dd t,
        \qquad
        W_t^{-1} = C\left(I + \psi\Sigma^{3/2}R_t\Sigma^{3/2}\psi^*C\right)^{-1}.
    \end{equation*}
    Each factor
    \(R_t\Sigma\psi^*\) maps into \(\mathcal{K}\), which gives the first
    claim. When \(C\) is definite, the matrices \(W_t^{-1}\) are definite
    with a common sign, so no cancellation can occur in the integral,
    and the closure of the range of \(S\) is the closed span of the
    vectors \(R_t\Sigma\psi^* a\) over \(t > 0\) and \(a \in \R^d\).
    This span is exactly \(\mathcal{K}\), as resolvents applied to a vector span
    the same closed subspace as the Krylov space it generates (\cref{lem:resolvent-span}). The third claim follows from the identity
    \(S\Sigma + \Sigma S + S\Sigma S = \Sigma_{u\psi}C\Sigma_{\psi u}\),
    and the fourth from the product structure that
    \(\Sigma\)-invariance of \(\Range(\psi^*)\) induces on both
    measures.
\end{proof}

\begin{example}[The optimal transport map is not low rank]\label{ex:ot-rank}
    On \(\R^2\), let \(\RefM = \N(0, \Sigma)\) with
    \(\Sigma = \begin{pmatrix} 1 & 3 \\ 3 & 10 \end{pmatrix}\), and let
    \(\psi(u) = u_1\) observe the first coordinate, so that
    \(\Sigma_{\psi\psi} = 1\). Since \(\Sigma e_1 = (1, 3)\) is not
    proportional to \(e_1\), the invariant subspace is
    \(\mathcal{K} = \R^2\). Let
    \(\eta = \N(0, \Sigma + cvv^*) \in \psiclass\) with
    \(v = \Sigma e_1\) and \(c = -\frac{99}{100}\), so that
    \begin{equation*}
        \Sigma + cvv^*
        = \frac{1}{100}\begin{pmatrix} 1 & 3 \\ 3 & 109 \end{pmatrix}.
    \end{equation*}
    Computing the transport map \(T_F\) of \cref{eqn:gauss-TF} where
    \(F(s) = s/10\) is the OT map between the marginals, and comparing to the
    OT map \(A\), we obtain that
    \begin{equation*}
        T_F - I = -\frac{9}{10}\begin{pmatrix} 1 & 0 \\ 3 & 0 \end{pmatrix},
        \qquad
        A - I = \begin{pmatrix} -0.699 & -0.080 \\ -0.080 & -0.646 \end{pmatrix}.
    \end{equation*}
    The perturbation \(T_F - I\) reads only the observed coordinate and
    displaces along the fixed direction \(\Sigma e_1\), so it has rank
    one, while \(A - I\) has eigenvalues \(-0.757\) and \(-0.588\), of
    comparable size, as predicted by \cref{thm:ot-rank}: the OT map
    contracts every direction of \(\R^2\), and is far from any rank-one
    perturbation of the identity.
\end{example}

\begin{remark}[Non-Gaussian targets]\label{rem:ot-rank-general}
    Gaussianity of the target enters \cref{thm:ot-rank} only through the
    rank statements (2) and (3), while claims (1) and (4) extend to every
    \(\eta \in \psiclass\) with finite second moment. Since
    \(\mathcal{K}\) is \(\Sigma\)-invariant, disintegrating gives
    \(\eta = \eta_\psi(\dd w)\,\RefM^w(\dd u)\), and because
    \(\mathcal{K}\) reduces \(\Sigma\) the reference conditionals factor as
    \(\RefM_{\mathcal K}(\cdot \mid \psi = w) \otimes \RefM_{\mathcal K^\perp}\),
    so every such \(\eta\) shares the factor of \(\RefM\) on \(\mathcal{K}^\perp\). The quadratic cost splits
    across this decomposition, so the OT map factorizes
    as well: it exists and is unique
    \cite[Theorem~6.2.10]{ambrosio2008gradient}, restricts to the
    identity on \(\mathcal{K}^\perp\), and its displacement lies in
    \(\mathcal{K}\) and depends only on \(P_{\mathcal{K}}u\). In the
    same way, when \(\Range(\psi^*)\) is \(\Sigma\)-invariant, the proof
    of the final claim applies unchanged, and the OT map
    from \(\RefM\) to any member of \(\psiclass\) is \cref{eqn:gauss-TF},
    with \(F\) the OT map between the marginals. The rank
    equality (2), by contrast, relies on the closed form of the Gaussian
    transport map, and characterizing the span of the displacement for
    general members of the class appears to be open.
\end{remark}

The rank and approximation of the optimal correction \(A - I\) in
\cref{thm:ot-rank} are connected to the literature on low-rank updates and
approximations to matrix functions, as the proof expresses \(A\)
through a resolvent integral representation of operator square roots.
In that literature, \cite{beckermann2018low,beckermann2021rational} construct Krylov subspace
approximations of \(f(M + UCV^*) - f(M)\) that are exact for
polynomial and rational \(f\). For rank-one perturbations to the input matrix, \cite{bernstein2000rational} bound
the rank of the update by the degree when \(f\) is rational. For the
square root, \cite{fasi2023square} study the case of low-rank
perturbations of a scaled identity, where the correction is exactly
low rank. Closed forms of this type are used to numerically compute
OT maps between high-dimensional Gaussians in
\cite{bouveyron2026scaling}, under the restriction that every
covariance is a low-rank perturbation of a multiple of the identity.
In general, \cite{shmueli2024lowrank} establish spectral decay
bounds for the correction (recall that \(A - I\) is
generally trace class and therefore approximable by finite rank
operators, even when not exactly low rank). \cref{thm:ot-rank} shows
that the correction is exactly low rank when the
invariant subspace \(\mathcal{K}\) is low dimensional. From a
statistical point of view, the spiked transport model of
\cite{niles2022estimation} assumes displacements confined to a low-dimensional
subspace, and \cref{thm:ot-rank} characterizes the
Gaussian pairs satisfying this hypothesis, with the spike dimension
given by \(\mathcal{K}\) rather than by the rank of the covariance
perturbation.

Conditions related to the invariant subspace \(\mathcal{K}\) in \cref{thm:ot-rank} appear as compatibility hypotheses in linearized
and sliced OT \cite{moosmuller2023linear,li2024slice},
where one asks that composition with a structured map preserve
optimality. For Gaussian families, this amounts to simultaneous
diagonalizability of the covariance matrices. \cref{thm:ot-rank} gives conditions at the level of subspaces and quantifies their failure
through the rank of the optimal correction. It also characterizes when
the subspace-detour Monge--Knothe transports of
\cite{muzellec2019subspace}, which can coincide with \cref{eqn:gauss-TF} for
targets in \(\psiclass\), are optimal.

Finally, we remark that OT maps do preserve the low-rank structure
when the Cameron--Martin norm of the reference is used in the transport
cost. In finite dimensions this is the Mahalanobis norm
\(\langle \delta, \Sigma^{-1}\delta \rangle\), and the transport between
two Gaussians amounts to whitening with respect to the reference,
transporting, and undoing the whitening. In this case, any subspace is
invariant with respect to the identity covariance operator of the
whitened reference, and the OT map reduces to
\cref{eqn:gauss-TF}, with \(F\) the OT map between the
marginals for the quadratic cost induced by \(\Sigma_{\psi\psi}^{-1}\).
In the infinite-dimensional setting this cost must be treated with
care: the Cameron--Martin space has measure zero under the reference,
the whitening is no longer realizable as a map on \(\V\), and the cost
is infinite for almost every independent pair, so that couplings of
finite cost must concentrate on Cameron--Martin shifts. Optimal
transport for this cost is developed on Wiener space in
\cite{feyel2004monge} (see also \cite{bogachev2012monge}). For
targets at finite transport distance from the reference, the optimal
map is the shift of the identity by the Cameron--Martin gradient of a
convex potential. The displacement of \cref{eqn:gauss-TF} lies in
\(\Range(\Sigma_{u\psi})\), inside the Cameron--Martin space, and is of
exactly this form.

\subsection{Best approximation in KL divergence by measures in \texorpdfstring{$\psiclass$}{P-psi(mu)}}\label{subsec:approx}
Recall our goal of approximating the target measure \(\TargetM\) with a measure \(\eta \in \psiclass\). It is natural to ask whether a \emph{best} approximation exists and whether it can be
identified explicitly. We answer both affirmatively when
the approximation error is measured with respect to forward and reverse Kullback--Leibler divergences.

\begin{proposition}\label{prop:best-approx}
    Let \(\TargetM\in\mathcal{P}(\U)\), assume that there exists some \(\eta\in\psiclass\) such that \(\kl(\TargetM,\eta)<\infty\), and define
    \begin{equation}\label{eqn:opt-forward}
        \tilde{\eta}(\dd u):=(\psi\sharp\TargetM)(\dd w)\RefM(\dd u\mid \psi(u)=w).
    \end{equation}
    Then \(\tilde{\eta}\in\psiclass\) by definition, and \(\kl(\TargetM,\eta)\geq\kl(\TargetM,\tilde{\eta})\) for all \(\eta\in\psiclass\).
    If there exists \(\eta\in\psiclass\) such that \(\kl(\eta,\TargetM)<\infty\), define
    \(
        \phi^{*}(w):=\kl(\RefM^w,\TargetM^w)\)
    and \(Z^{*}:=\int_{\U}\exp(-\phi^{*}\circ\psi(u))\tilde{\eta}(\dd u)\).
    Then \(0<Z^{*}\leq 1\), and with
    \begin{equation}\label{eqn:opt-reverse}
        \eta^{*}(\dd u)=\frac{1}{Z^{*}}\exp(-\phi^{*}\circ\psi(u))\tilde{\eta}(\dd u),
    \end{equation}
    we have \(\eta^{*}\in\psiclass\) and \(\kl(\eta,\TargetM)\geq\kl(\eta^{*},\TargetM)\) for all \(\eta\in\psiclass\).
\end{proposition}
\begin{proof}
    For \cref{eqn:opt-forward}, we make use of the chain rule
    \begin{equation}
        \kl(\TargetM,\eta) = \kl(\psi\sharp \TargetM,\psi\sharp \eta) + \int_{\R^d} \kl(\TargetM^w,\eta^w)(\psi \sharp \TargetM)(\dd w).
    \end{equation}
    If \(\kl(\TargetM,\eta)=\infty\), the inequality is trivial.
    Otherwise \(\TargetM_\psi\ll\eta_\psi\), so for \(\eta \in \psiclass\), \(\eta^w = \RefM^w\) holds \(\TargetM_\psi\)-almost everywhere, and the second term is constant for all measures in \(\psiclass\).
    Hence, minimizing \(\kl(\TargetM,\eta)\) over \(\eta \in \psiclass\) is equivalent to minimizing \(\kl(\psi\sharp \TargetM,\psi\sharp \eta) \), so the minimizer must satisfy \(\psi\sharp \TargetM = \psi\sharp \eta\). Finally, in order to ensure that \(\eta \in \psiclass\), we must have that \(\eta(\cdot \mid \psi(u) = w) = \RefM(\cdot \mid \psi(u) = w)\). Combining these two conditions yields the result that
    \(
        \tilde{\eta}(\dd u)
        =
        (\psi\sharp \TargetM)(\dd w)\RefM(\dd u \mid \psi(u) = w)
    \).

    For \cref{eqn:opt-reverse}, we reverse the chain rule and use that \(\eta \in \psiclass\) to obtain, with \(\phi^{*}(w) := \kl(\RefM^w,\TargetM^w)\) as in the statement,
    \begin{equation}\label{eq:reverse-chain}
        \kl(\eta,\TargetM)
        = \kl(\eta_\psi,\TargetM_\psi)
        + \int_{\R^d} \phi^{*}(w) \eta_\psi(\dd w).
    \end{equation}

    Define
    \begin{equation}
        Z^{*} := \int_{\R^d} e^{-\phi^{*}(w)} \TargetM_\psi(\dd w).
    \end{equation}
    If \(Z^{*}=0\), then \(\phi^{*}=\infty\) \(\TargetM_\psi\)-a.e., and \(\kl(\eta,\TargetM)=\infty\) for every \(\eta\in\psiclass\), so the claim is trivial.
    Assume therefore that \(Z^{*}>0\), and set
    \begin{equation}
        \eta^{*}_\psi(\dd w) := \frac{e^{-\phi^{*}(w)}}{Z^{*}} \TargetM_\psi(\dd w).
    \end{equation}
    Now fix \(\eta\in\psiclass\), and assume that \(\kl(\eta,\TargetM)<\infty\).
    By \cref{eq:reverse-chain} we have \(\eta_\psi \ll \TargetM_\psi\) and \(\int \phi^{*} \dd\eta_\psi<\infty\).
    Together with \(\phi^{*} \geq 0\), we have \(\eta_\psi(\{\phi^{*}=\infty\})=0\). On \(\{\phi^{*}<\infty\}\), the density \(\dd\eta^{*}_\psi / \dd\TargetM_\psi = e^{-\phi^{*}}/Z^{*}\) is positive, so \(\eta_\psi \ll \eta^{*}_\psi\).

    Applying the Radon--Nikodym chain rule gives
    \begin{equation*}
        \frac{\dd \eta_\psi}{\dd \TargetM_\psi}
        =
        \frac{\dd \eta_\psi}{\dd \eta^{*}_\psi}
        \frac{\dd \eta^{*}_\psi}{\dd \TargetM_\psi}
        =
        \frac{\dd \eta_\psi}{\dd \eta^{*}_\psi}\frac{e^{-\phi^{*}}}{Z^{*}}.
    \end{equation*}
    Taking logs,
    \begin{equation*}
        \log \left(\frac{\dd \eta_\psi}{\dd \TargetM_\psi}\right)+\phi^{*}
        =
        \log \left(\frac{\dd \eta_\psi}{\dd \eta^{*}_\psi}\right)-\log Z^{*}.
    \end{equation*}
    Integrating against \(\eta_\psi\) yields
    \begin{equation}
        \kl(\eta_\psi,\TargetM_\psi) + \int_{\R^d}\phi^{*}(w) \eta_\psi(\dd w)
        = \kl(\eta_\psi,\eta^{*}_\psi)-\log Z^{*}
        \ge -\log Z^{*}.
    \end{equation}
    Using \cref{eq:reverse-chain}, we obtain that \(\kl(\eta,\TargetM)\ge -\log Z^{*}\) for all \(\eta \in \psiclass\).

    Setting
    \begin{equation*}
        \eta^{*}(\dd u)=\eta^{*}_\psi(\dd w)\RefM(\dd u\mid \psi(u)=w)
    \end{equation*}
    and applying \cref{eq:reverse-chain} with \(\eta=\eta^{*}\) we get
    \begin{equation*}
        \kl(\eta^{*},\TargetM)
        = \kl(\eta^{*}_\psi,\TargetM_\psi)+\int \phi^{*} \dd\eta^{*}_\psi
        = -\log Z^{*}.
    \end{equation*}
    Therefore \(\kl(\eta,\TargetM)\geq \kl(\eta^{*},\TargetM)\) for all \(\eta\in\psiclass\), establishing the optimality of \(\eta^{*}\).
\end{proof}

\section{Computational methods}\label{sec:methods}
In this section we develop algorithms for approximating a target measure by a measure in \(\psiclass\).
Because a measure in \(\psiclass\) is determined by its marginal \(\eta_\psi\) (\cref{prop:sampling}), it suffices to model this finite-dimensional marginal, and \cref{prop:best-approx} identifies the optimal marginals to target.
In \cref{subsec:sample-approx}, we fit the marginal with a generative model trained on samples drawn from the target.
In \cref{subsec:vi}, we instead fit the marginal given access to the likelihood, parameterizing the approximation through the transport maps of \cref{thm:transport} and directly minimizing the reverse KL divergence.
Both routes require sampling from the reference conditionals \(\RefM(\dd u \mid \psi(u) = w)\), which we address in \cref{subsec:cond-sampling}.
\subsection{Generative approximation from samples}\label{subsec:sample-approx}
When approximating a measure \(\TargetM\) from samples, it is convenient, and optimal in forward KL divergence, to target the marginal, taking \(\eta_{\psi} \approx \TargetM_\psi := \psi \sharp \TargetM\) when constructing a measure in \(\psiclass\).
With samples \(u_1,\ldots,u_N \iidsim \TargetM\), this can be done by fitting a generative model to the samples \(w_i = \psi(u_i)\).
In most settings of interest, samples from \(\TargetM\) are not directly available. However, it is often the case (particularly in Bayesian inverse problems) that the target distribution is a conditional of a joint measure, and samples from a \emph{joint distribution} \(\TargetM(y,u)\) are available through simulation \cite{friendly-triangular,hosseini2025conditional,mgan}.
In these cases, \(u\) represents the parameter of interest, \(y\) represents the observed data, and we have access to joint samples \(\{(y_i,u_i)\}_{i=1}^N\). A \emph{conditional} generative model may be fit to approximate the conditional measures \(\TargetM_{\psi}(\cdot \mid y)\) from \(\{(y_i,w_i)\}_{i=1}^N\), where \(w_i = \psi(u_i)\) as before.
Further, these approaches are particularly applicable when samples are only available at a fixed resolution, where we may regard \(w_i = \psi(u_i)\) as coarse observations of \(u_i\) themselves, in which case we may approximate \(\TargetM_\psi\) directly.

In all of these cases, the generative modeling problem is reduced to the finite-dimensional problem of approximating \(\TargetM_{\psi}\) from samples.
\begin{algorithm}
\caption{Reduced generative model}\label{alg:reduced-gen}
\begin{algorithmic}[1]
\REQUIRE Samples \(u_1,\ldots,u_N \iidsim \TargetM\), dimensionality reduction map \(\psi\)
\FOR{$i = 1,\ldots,N$}
    \STATE Compute \(w_i = \psi(u_i)\)
\ENDFOR
\STATE Fit a generative model \(\hat{\eta}_\psi\) using the samples \(w_1,\ldots,w_N\)
\STATE Set \(\hat{\eta}(\dd u) = \hat{\eta}_\psi(\dd w) \RefM(\dd u \mid \psi(u) = w)\)
\end{algorithmic}
\end{algorithm}

\begin{algorithm}
\caption{Reduced conditional generative model}\label{alg:reduced-cond-gen}
\begin{algorithmic}[1]
\REQUIRE Samples \((y_1,u_1),\ldots,(y_N,u_N) \iidsim \TargetM(y,u)\), dimensionality reduction map \(\psi\)
\FOR{$i = 1,\ldots,N$}
    \STATE Compute \(w_i = \psi(u_i)\)
\ENDFOR
\STATE Fit a conditional generative model \(\hat{\eta}_\psi(w \mid y)\) using the samples \((y_i,w_i)\)
\STATE Set \(\hat{\eta}(\dd u \mid y) = \hat{\eta}_\psi(\dd w \mid y) \RefM(\dd u \mid \psi(u) = w)\)
\end{algorithmic}
\end{algorithm}

\subsection{Variational inference with likelihood access}\label{subsec:vi}

When the potential \(\Phi\) can be evaluated and samples are not available, a natural approach is to
minimize the reverse KL divergence over measures in \(\psiclass\) \cite{jordan1999introduction,blei2017variational,pinski2015kullback,pinski2015algorithms,marzouk2016sampling,func_norm_flow}. For
\(\phi:\R^d\to\R\) normalized so that
\(\eta(\phi)(\dd u) := \exp(-\phi\circ\psi(u))\dd\RefM(u)\) is a
probability measure, we seek to minimize
\begin{equation}\label{eqn:reverse-kl-vi}
\kl\left(\eta(\phi), \TargetM\right)
=\log(Z)+\mathbb{E}_{U\sim\eta(\phi)}\left[\Phi(U)-\phi(\psi(U))\right]
\end{equation}
with respect to \(\phi\), where \(Z\) is the normalizing constant of \(\TargetM\).

This variational problem may be directly targeted using the triangular map structure of transport
maps into \(\psiclass\).
Recall from \cref{thm:transport} that for any \(\eta\in\psiclass\), if \(F\sharp\RefM_{\psi}=\eta_{\psi}\)
and \(G(v,w,\cdot)\sharp\RefM(\cdot\mid\psi(u)=w)=\RefM(\cdot\mid\psi(u)=v)\),
then with \(T_{F}(u)=G(F\circ\psi(u),\psi(u),u)\),
we have \(\eta=T_{F}\sharp\RefM\),
so that \(F\) is the only free part of the transport. As \(\log(Z)\) does not depend on \(F\), we may
omit it, and obtain the negative of the evidence lower bound as our objective.
\begin{align}
\L\left(F\right)&=\kl\left(T_{F}\sharp\RefM,\TargetM\right)-\log(Z) \nonumber\\
&=\mathbb{E}_{W\sim\eta_{\psi}}\left[\mathbb{E}\left[\Phi(U)\mid\psi(U)=W\right]-\phi(W)\right]. \label{eqn:vi-objective-phi}
\end{align}
Using that
\(\exp(-\phi(w))\RefM_{\psi}(\dd w)=\left(F\sharp\RefM_{\psi}\right)(\dd w)\)
and assuming that \(\RefM_{\psi}\) has a density \(\rho(w)\) and that \(F\) is a diffeomorphism, we have
\begin{equation}\label{eqn:phi-from-F}
\phi(w)=\log\rho(w)-\log\rho\left(F^{-1}(w)\right)+\log\left|\det DF\left(F^{-1}(w)\right)\right|.
\end{equation}

Substituting \cref{eqn:phi-from-F} into \cref{eqn:vi-objective-phi} and pulling back by \(F\sharp\RefM_{\psi}=\eta_{\psi}\), so that expectations are taken under the reference,
\begin{align}
\L(F) &=\mathbb{E}_{W\sim\RefM_{\psi}}\left[\mathbb{E}\left[\Phi(G(F(W),W,U))\mid\psi(U)=W\right]+\log\frac{\rho\left(W\right)}{\rho(F(W))}-\log\left|\det DF\left(W\right)\right|\right]. \label{eqn:vi-pullback}
\end{align}

Drawing \(U\sim\RefM\) and setting \(W=\psi(U)\), we may replace the conditional expectation with a single-sample estimate, giving the Monte Carlo amenable objective
\begin{equation}\label{eqn:vi-mc}
\L(F) = \mathbb{E}_{\RefM}\left[\Phi(G(F(W),W,U))+\log\frac{\rho\left(W\right)}{\rho(F(W))}-\log\left|\det DF\left(W\right)\right|\right].
\end{equation}
This objective requires evaluating \(\Phi\) at arbitrary points of the entire space. The
term involving the conditional transport map \(G\) may equivalently be expressed as
\begin{equation}
\mathbb{E}\left[\Phi(G(F(W),W,U))\mid\psi(U)=W\right]=\mathbb{E}_{U\sim\RefM^{F(W)}}\left[\Phi(U)\right],
\end{equation}
so that any method for sampling from the reference conditional \(\RefM(\cdot \mid \psi(u) = F(W))\)
may be substituted for the explicit transport map \(G\).

\begin{algorithm}
\caption{Variational inference in \(\psiclass\)}\label{alg:vi}
\begin{algorithmic}[1]
\REQUIRE Potential \(\Phi\), dimensionality reduction map \(\psi\), reference density \(\rho = \dd\RefM_\psi / \dd w\), conditional transport \(G\) satisfying \(G(v,w,\cdot)\sharp\RefM(\cdot \mid \psi(u)=w) = \RefM(\cdot \mid \psi(u) = v)\)
\STATE Minimize over diffeomorphisms \(F:\R^d \to \R^d\),
\begin{equation*}
\L(F) = \mathbb{E}_{\RefM}\left[\Phi(G(F(W),W,U))+\log\frac{\rho(W)}{\rho(F(W))}-\log|\det DF(W)|\right]
\end{equation*}
where \(U \sim \RefM\), \(W = \psi(U)\)
\STATE Set \(\hat{\eta}_\psi = F\sharp\RefM_\psi\)
\STATE Set \(\hat{\eta}(\dd u) = \hat{\eta}_\psi(\dd w) \RefM(\dd u \mid \psi(u) = w)\)
\end{algorithmic}
\end{algorithm}

For Gaussian \(\RefM\) and linear \(\psi\), the transports parameterized
in \cref{alg:vi} coincide with the lazy maps of \cite{lazy-map}. There,
such maps are composed greedily to iteratively improve the variational
approximation, with each active subspace selected by minimizing a
Kullback--Leibler upper bound of the type recalled in
\cref{subsec:functional}. The functional normalizing flows of
\cite{func_norm_flow} minimize the same reverse divergence over
compositions of the finite rank maps of \cref{rem:finite-rank-layers},
with all layers trained jointly rather than greedily.

\subsection{Conditioned sampling of reference measures}\label{subsec:cond-sampling}

The sampling procedures in \cref{alg:reduced-gen,alg:reduced-cond-gen,alg:vi} all
require drawing from the reference conditional \(\RefM(\cdot \mid \psi(u) = w)\).
This step is what preserves the infinite-dimensional structure of the reference measure
while restricting modification to the finite-dimensional component captured by \(\psi\).
We now discuss how this conditioned sampling step may be carried out concretely for
several classes of reference measures.

\subsubsection{Linearly conditioned Gaussian reference measures}\label{subsubsec:gauss-cond}
For a Gaussian reference measure and bounded linear \(\psi\), the conditional
\(\RefM(\cdot \mid \psi(u) = w)\) is the Gaussian \cref{eq:gauss-cond}, and
Matheron's update \cref{eq:matheron} samples it by applying a sample-dependent
shift to a prior draw.
We demonstrate inference based on nonlinear observations of a Gaussian process in \cref{subsec:gp2}.

\paragraph{Sampling Gaussian processes}
When \(\RefM\) is a Gaussian process and \(\psi(u) = (u(x_1),\ldots,u(x_d))\) consists of
pointwise evaluations, the formula \cref{eq:gauss-cond} specializes to standard
Gaussian process (GP) conditioning. A GP sample can be evaluated \emph{exactly} at any
finite collection of points by drawing from the corresponding joint Gaussian distribution.
Furthermore, given an existing sample evaluated at points \(\{x_1,\ldots,x_d\}\),
the values at additional points \(\{x_{d+1},\ldots,x_{d+k}\}\) can be filled in
by drawing from the conditional Gaussian at \(O((d+k)^2)\) cost per new point, extending the Cholesky factor of the covariance as points are added.
This means that one need not commit to a fixed discretization as the sample
can be lazily extended at progressively finer resolution, with each
refinement being an exact draw from the conditional.
In \cite{pathwise_gp}, Matheron's update is combined with random Fourier feature sampling of the prior \cite{rahimi2007randomfeatures}, producing conditioned samples that retain a random feature functional representation.

\paragraph{Nested Brownian bridge and the virtual Brownian tree}
Brownian motion on \([0,T]\) admits particularly simple conditioned sampling.
By the Markov property, the process on any subinterval \([s,t]\),
conditioned on \(B_s\) and \(B_t\), is a Brownian bridge:
\begin{equation}\label{eq:bb}
    B_r \mid B_s, B_t
    \sim \N\!\left(B_s + \frac{r-s}{t-s}(B_t - B_s), \frac{(r-s)(t-r)}{t-s}\right),
    \quad r\in(s,t),
\end{equation}
and is conditionally independent of the process outside \([s,t]\).
Given any collection of previously observed values,
the process between two adjacent observations is therefore again a Brownian bridge
and may be sampled by conditioning only on those two neighbors.
Each new evaluation point therefore requires a single univariate Gaussian draw
at \(O(1)\) cost (plus \(O(\log n)\) to locate the neighbors among \(n\) existing points).

The application of this recursive refinement to dyadic points is closely related to the
L\'evy--Ciesielski representation via the Haar basis,
and to the virtual Brownian tree construction
applied to adaptive simulation of stochastic differential equations
\cite{2025singleseedgenerationbrownianpaths,pmlr-scalable-gradients}.
It is a special case of \cref{eq:gauss-cond} for the covariance
\(K(s,t) = \min(s,t)\), where the Markov property reduces the
general GP conditioning update to dependence on only two neighboring points.

By \cref{eq:matheron}, a sample (with the associated functional representation) of a Brownian motion or Brownian bridge may be conditioned on finitely many pointwise observations by adding a piecewise-linear function.

\subsubsection{Processes with independent increments}\label{subsubsec:bpjp}
When the reference measure is the law of a process with independent increments
and \(\psi(u) = (u(x_1),\ldots,u(x_d))\) consists of pointwise evaluations
at ordered points \(0 = x_0 < x_1 < \cdots < x_d < x_{d+1} = T\),
the conditional \(\RefM(\cdot \mid \psi(u) = w)\) decomposes into
\(d+1\) independent bridge problems, one on each subinterval \([x_j, x_{j+1}]\)
with prescribed endpoint values.
It is often possible to sample the process within these subintervals conditioned on these endpoint values.

As a non-Gaussian example that still admits efficient conditioned sampling,
consider the sum of a Brownian motion and a compound Poisson process with Gaussian jump sizes:
\begin{equation}\label{eq:bpjp}
    X_t = X_0 + B_t + \sum_{i=1}^{N_t} J_i,
    \qquad J_i \iidsim \mathcal N(0,\sigma_{\mathrm{jump}}^2),
\end{equation}
where \(B_t\) is a Brownian motion with \(\Var(B_t) = \sigma_{\mathrm{bm}}^2 t\) and \(N_t \sim \mathrm{Pois}(\lambda t)\), with \(B_t\), \(N_t\), and the jump sizes mutually independent. This is a L\'evy process, with stationary and independent increments, and is the log of a particular Merton jump diffusion process \cite{Merton1976JumpDiffusion}; we will refer to it as a Brownian Poisson jump process.
The process is càdlàg with finitely many jumps on every compact interval.

To sample the \emph{bridge} of this process conditioned on \(X_0 = a\) and \(X_T = b\),
set \(\Delta = b - a\) and proceed in three steps.

\paragraph{Step 1: sample the jump count}
Since \(X_T - X_0 \mid (N_T=k) \sim \N(0 , \sigma_{\mathrm{bm}}^2 T + k\sigma_{\mathrm{jump}}^2)\), and the prior distribution on the number of jumps is Poisson, Bayes' rule gives that
\begin{equation}\label{eq:count-probs}
    \PP(N_T = k \mid X_T - X_0 = \Delta)
    \propto \exp(-\lambda T)\frac{(\lambda T)^k}{k!}\varphi(\Delta; 0, \sigma_{\mathrm{bm}}^2 T + k\sigma_{\mathrm{jump}}^2),
\end{equation}
for \(k \geq 0\), where \(\varphi\) denotes the Gaussian density with the specified mean and variance.
These probabilities are cheap to compute, and the sum may be safely truncated owing to the exponential decay of the Poisson weights.

\paragraph{Step 2: sample the jump times}
Conditional on \(N_T = K\), the jump times are sampled as the order
statistics of \(K\) independent uniforms on \([0,T]\), giving
\(\tau_1 < \cdots < \tau_K\).

\paragraph{Step 3: sample the jump sizes}
After conditioning on \(K\), and before conditioning on the endpoint value, \(B_T,J_1,\ldots,J_K\) are independent Gaussians. Thus, in order to sample
\begin{equation}
    B_T,J_1,\ldots,J_K \mid B_T + \sum_{i=1}^K J_i = \Delta,
\end{equation}
we independently sample
\[
\widetilde{B}_T \sim \mathcal N(0,\sigma_{\mathrm{bm}}^2 T),
\qquad
\widetilde{J}_1,\ldots,\widetilde{J}_K \iidsim \mathcal N(0,\sigma_{\mathrm{jump}}^2)
\]
and apply Matheron's update to obtain a conditional sample as an inverse-covariance-weighted projection onto the conditioning set \cite{pathwise_gp}. With \(v^2 = \sigma_{\mathrm{bm}}^2 T + K \sigma_{\mathrm{jump}}^2\) and \(\widetilde{S} = \widetilde{B}_T + \sum_{i=1}^K \widetilde{J}_i\), set
\begin{equation}
    B_T = \widetilde{B}_T + (\Delta - \widetilde{S})\frac{\sigma_{\mathrm{bm}}^2 T}{v^2},\quad
    J_i = \widetilde{J}_i + (\Delta - \widetilde{S})\frac{\sigma_{\mathrm{jump}}^2}{v^2}.
\end{equation}

A Brownian bridge is constructed conditioned on the value of \(B_T\), and the process may be evaluated at an arbitrary point \(t\in [0,T]\) as
\begin{equation}
    X_t = a + B_t + \sum_{i:\tau_i\leq t} J_i.
\end{equation}
We may also compute the density \(p_T\) of the increment \(X_T - X_0\) using \cref{eq:count-probs} and summing over the probabilities of jump counts,
\begin{equation}
    p_{T}(x) = \exp(-\lambda T)\sum_{k=0}^\infty
    \frac{(\lambda T)^k}{k!}\varphi(x; 0, \sigma_{\mathrm{bm}}^2 T + k\sigma_{\mathrm{jump}}^2).
\end{equation}
In particular, the law of the process is not log-concave. Over a
short interval, an increment either contains no jump, giving a narrow
Gaussian, or a single jump, giving a much wider one, and a mixture of
this kind is not log-concave in general.

The Markov property allows us to condition on values \(w_1, \ldots, w_m\) at an arbitrary collection of points \(0=x_1<\cdots<x_m=T\) by applying this procedure to each subinterval independently, sampling the outer endpoint values \(w_1, w_m \iidsim q\) for a boundary distribution \(q\) of our choosing. Along with another application of Bayes' rule, this allows us to compute the joint density of a sequence of observations as the product of the densities of each independent increment.
If we were to take the density of the Brownian Poisson jump process sampled at discrete points without conditioning on the right endpoint, we would have
\begin{equation}
    p(w_2,\ldots,w_{m-1},w_m \mid w_1) = \prod_{i=1}^{m-1}p_{x_{i+1} - x_i}(w_{i+1} - w_{i}),
\end{equation}
so, using \(p(w_m \mid w_1) = p_{x_m - x_1}(w_m - w_1)\),
\begin{equation}
    p(w_2,\ldots,w_{m-1} \mid w_1, w_m) =
    \frac{\prod_{i=1}^{m-1}p_{x_{i+1} - x_i}(w_{i+1} - w_{i})}{
        p_{x_m -x_1}(w_m - w_1)
    }.
\end{equation}
Finally, imputing the boundary distribution \(q\) for the endpoint values,
\begin{equation}
    p(w_1,\ldots,w_{m}) = q(w_1)q(w_m)
    \frac{\prod_{i=1}^{m-1}p_{x_{i+1} - x_i}(w_{i+1} - w_{i})}{
        p_{x_m - x_1}(w_m - w_1)}.
\end{equation}
Numerical results with this prior are demonstrated in \cref{subsec:jump-deconv}.

\subsubsection{Conditioned sampling via an auxiliary generative model}\label{subsubsec:aux-gen}
When exact conditioned sampling from \(\RefM(\cdot \mid \psi(u) = w)\) is
not tractable, one may instead approximate this conditional using a
learned generative model. Concretely, let \(\hat{\RefM}(\cdot \mid w)\)
denote a conditional generative model trained to approximate
\(\RefM(\cdot \mid \psi(u) = w)\).

The auxiliary model \(\hat{\RefM}\) can serve as a learned reference for
the conditional fine-scale structure, and need not be adapted to
any particular target \(\TargetM\). Its role is to approximate
\(\RefM(\cdot \mid \psi(u) = w)\) accurately, with no requirement on the distribution of \(\psi(U)\).
In practice, it may be trained on samples from a broad data distribution
that serves as an artificial reference.
For instance, \(\RefM\) could be a distribution over a diverse
collection of high-resolution natural images, \(\TargetM\) a distribution
specific to a particular dataset, and \(\psi\) a coarsening operator
to low resolution.
The task-specific model then operates only on the low-dimensional
component \(w = \psi(u)\), while the auxiliary model handles
the conditional generation of fine-scale detail.

This two-stage structure---adapting the coarse component to the target,
then drawing the fine-scale residual from a learned conditional---is
closely related to the framework of \cite{debias-coarsely},
which composes a coarse-scale debiasing step with diffusion-based
conditional sampling for high-resolution generation.
In our setting, the coarse step corresponds to sampling
\(w \sim \eta_\psi\), and the fine step to drawing from
\(\hat{\RefM}(\cdot \mid w) \approx \RefM(\cdot \mid \psi(u) = w)\).

This decomposition admits a natural form of transfer learning based on multi-resolution sampling:
the auxiliary model, trained once on a large generic dataset,
may be reused across different targets sharing the same reference,
while only the low-dimensional coarse model is refit per task.
The sample complexity of the task-specific model depends on the
dimension of \(\psi(u)\) rather than the ambient dimension,
while the cost of modeling fine-scale conditional structure is
amortized, with the model trained on a possibly broader dataset.

We demonstrate this approach in \cref{subsec:navier-stokes}, performing state estimation from sparse sensor measurements on data from the Navier--Stokes equations.

\section{Numerical results}\label{sec:numerics}
We present three numerical examples. In the first, a Gaussian process
is observed through a nonlinear function of pointwise values. The
likelihood depends on \(u\) exactly through \(\psi(u)\), and the
reference conditionals are Gaussian. Both the dimension reduction and
the conditional sampling are exact. In the second, we consider
deconvolution under a jump process prior, where the likelihood is
only approximately determined by \(\psi(u)\), while the conditionals
are sampled exactly by the bridge constructions of
\cref{subsec:cond-sampling}. In the third, we consider a filtering
problem for solutions of the Navier--Stokes equations, where both the
marginal and the reference conditionals are learned, following
\cref{subsubsec:aux-gen}.

\subsection{Nonlinear inference on a Gaussian process}\label{subsec:gp2}
Let $\Omega = [0, 1]$ and let $\mu \in \mathcal{P}(L^2(\Omega))$ be a mean-zero Gaussian process with covariance kernel $k(x,x') = \exp(-\frac{(x-x')^2}{2 \cdot (0.15^2)})$, concentrated on any of the Sobolev spaces $\V = H^s(\Omega)$, $s > 1/2$, on which the point evaluations defining $\psi$ are bounded linear functionals. Fix $d = 5$ observation locations $\{x_i\}_{i=1}^d \subset \Omega$ and noise variance $\sigma^2 = 10^{-2}$. We take $\psi$ to be pointwise evaluation at the observation locations, $\psi(u) := (u(x_1), \ldots, u(x_d)) \in \R^d$, and with this choice, the dimension reduction is exact, as the likelihood depends only on these values. The observational model is
\begin{equation}
    y_i = u(x_i)^2 + \epsilon_i, \qquad i = 1, \ldots, d,
\end{equation}
where $\epsilon_i \iidsim \N(0, \sigma^2)$. This posterior distribution is multimodal as each observation gives information about $u(x_i)^2$, making the sign ambiguous at each of the observation locations and leaving up to \(2^5 = 32\) modes. We train a conditional flow matching model with three hidden layers of size 128 on 50,000 joint draws $(Y, \psi(U))$, giving an approximation of \(\psi \sharp \TargetM(\cdot \mid y)\) for any observation $y$, and then sample the Gaussian process prior conditioned on these values directly. The model captures both the multimodality and the functional structure of the data; see \cref{fig:GP2}. Each of the 32 modes corresponds to a sequence of five signs indicating whether the function is positive or negative at each of the five observation sites. Using this to partition the probability mass into these modes, we obtain good agreement with the ground truth mode probabilities; see \cref{fig:prob-compare}. We emphasize, however, that this performance rests on the quality of the 5-dimensional generative model, and that the dimension reduction here is exact. This example illustrates the sampling perspective on \(\psiclass\) in \cref{prop:sampling}. We note that this setting fits that of \cite{chen2025gaussian}, whose representer theorem describes such conditioned Gaussians.

\begin{figure}[htbp]
    \centering
    \includegraphics[width=\linewidth]{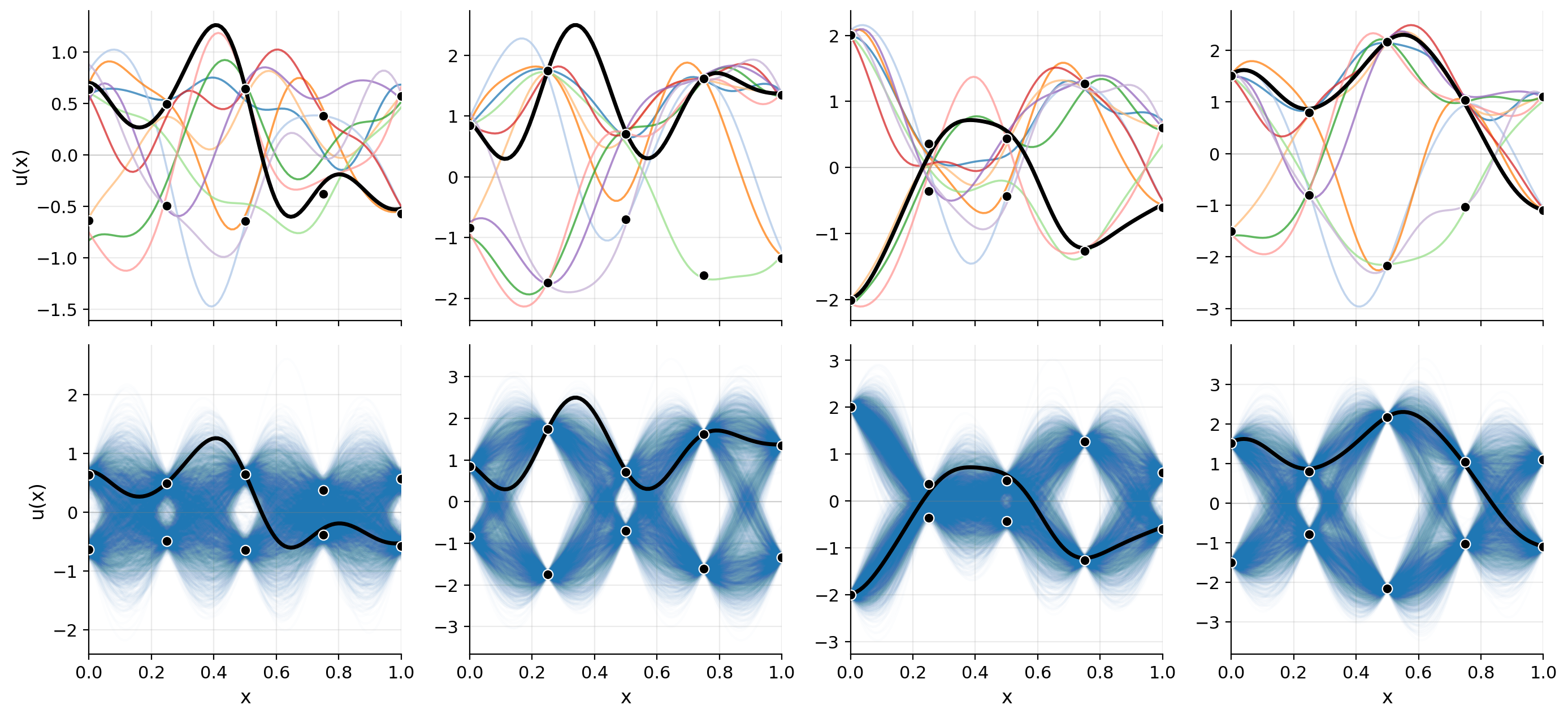}
    \caption{Example posteriors for the squared pointwise observations of a GP, with draws generated using the flow matching model. The black line is the true function $u(x)$; the black dots are at $(x_i, \pm\sqrt{|y_i|})$. Top row: ten draws shown in different colors to illustrate the pathwise distribution. Bottom row: 4000 draws shown in blue to illustrate the marginal distribution at each point \(x\) in space.}
    \label{fig:GP2}
\end{figure}

\begin{figure}[htbp]
    \centering
    \includegraphics[width=\linewidth]{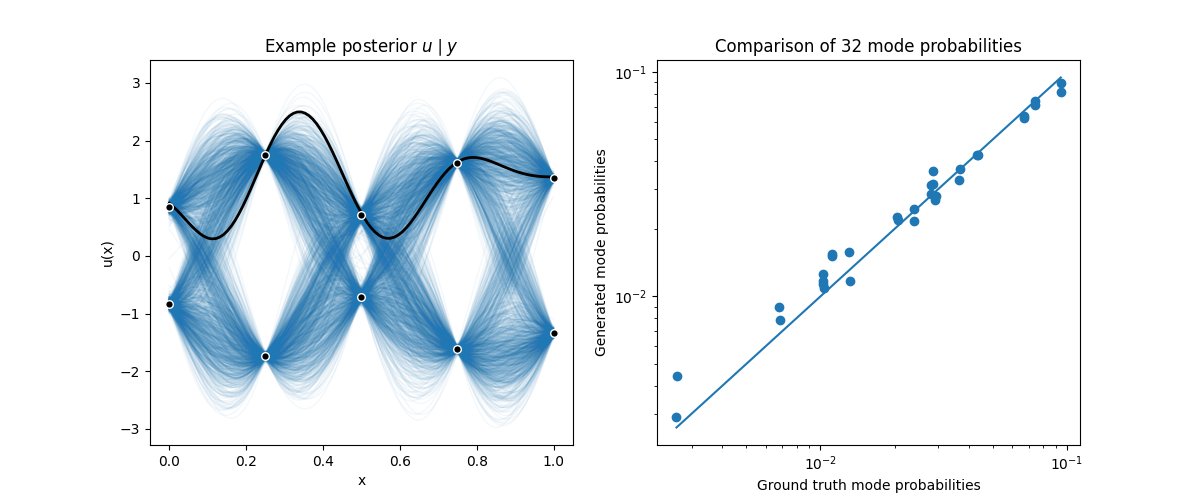}
    \caption{Left: an example posterior distribution with draws in blue and the truth in black. Right: a comparison of the 32 mode-wise probabilities, with numerically computed mode probabilities (treated as ground truth) on the x-axis and the generative model on the y-axis. Assignment of posterior draws to modes is determined by taking the sign of the draw at each of the five observation sites. The ground truth mode probabilities are computed by approximating the posterior as a mixture of Laplace approximations centered at each of the modes, with probabilities given by the relative evidence corrected by importance sampling. These are accurate because the modes are peaked and well separated, and can be enumerated using the knowledge that they correspond roughly to sign patterns.}
    \label{fig:prob-compare}
\end{figure}

\subsection{Nonlinear deconvolution with a jump prior}\label{subsec:jump-deconv}
Let $\Omega = [0,1]$. For each coarse grid
$\{z_i\}_{i=1}^d$ used below, we take
\begin{equation}\label{eqn:special-space}
    \begin{aligned}
        \U = \V = \mathcal{H}_{\psi}
        &:= L^2\!\left(\Omega,\dd t+\sum_{i=1}^d\delta_{z_i}\right),
        \\
        \|u\|_{\mathcal{H}_{\psi}}^2
        &= \|u\|_{L^2(\Omega,\dd t)}^2 + \|\psi(u)\|_2^2.
    \end{aligned}
\end{equation}
This is a separable Hilbert space on which the grid evaluation map
$\psi$ is bounded linear. Let $\mu\in\mathcal{P}(\mathcal{H}_{\psi})$
be the law of the Brownian Poisson jump process with parameters
\(\sigma_{\mathrm{bm}}^2=0.5\), \(\lambda=4\),
\(\sigma_{\mathrm{jump}}^2=1\), and endpoint values drawn
independently from \(\N(0,2.25)\) (see \cref{subsubsec:bpjp}).
We regard each c\`adl\`ag sample path as an element of
$\mathcal{H}_{\psi}$ through its Lebesgue equivalence class and
its actual values at the grid points.
With $M = 25$ observation locations $\{x_i\}_{i=1}^M \subset \Omega$, noise variance $\sigma^2 = 10^{-2}$, and a convolution kernel \(\varphi\) set to
\begin{equation}
    \varphi(\Delta) = \frac{1}{2 \sigma_K \sqrt{\pi}}\, \exp\!\left(-\left(\frac{\Delta}{2\sigma_K}\right)^{\!2}\right), \qquad \sigma_K = 0.1,
\end{equation}
we consider observations, with the convolution taken over $\Omega$,
\begin{equation}
    y_i = \bigl(\varphi * \tanh(3 u)\bigr)(x_i) + \epsilon_i, \qquad i = 1, \ldots, M,
\end{equation}
where $\epsilon_i \iidsim \N(0, \sigma^2)$. The goal is to recover the posterior distribution of \(u\) given the observations $Y = (y_1, \ldots, y_M)$.
We take the dimensionality reduction map $\psi$ to be pointwise evaluation on an equispaced grid, \(\psi(u) := (u(z_1), \ldots, u(z_d)) \in \R^d\), and set \(d = 30\). There is good agreement with a ground truth computed via MCMC, which is slow but mixes well; see \cref{fig:jump-deconv-comparison}. The 24 NUTS chains initialized from independent prior draws converge to a common distribution, with split $\hat{R} \le 1.004$ and bulk effective sample size above $6{,}500$ in every coordinate. The setup and example posteriors are shown in
\cref{fig:jump-deconv-setup,fig:jump-deconv-samples}. Varying the
grid resolution \(d\), the prior-normalized squared kernel Stein discrepancy of the generated
posterior decreases until about \(d = 30\), after which no further improvements are observed, and the increased difficulty of marginal fitting dominates the reduction in class error from approximation in \(\psiclass\) for larger \(d\)
(\cref{fig:jump-deconv-setup}, right). The discrepancy is computed with an inverse multi-quadric kernel \cite{gorham2017measuring} and a median-heuristic bandwidth computed from the prior.

\begin{figure}[htbp]
    \centering
    \begin{minipage}[t]{0.48\linewidth}
        \centering
        \vspace{0pt}
        \includegraphics[width=\linewidth]{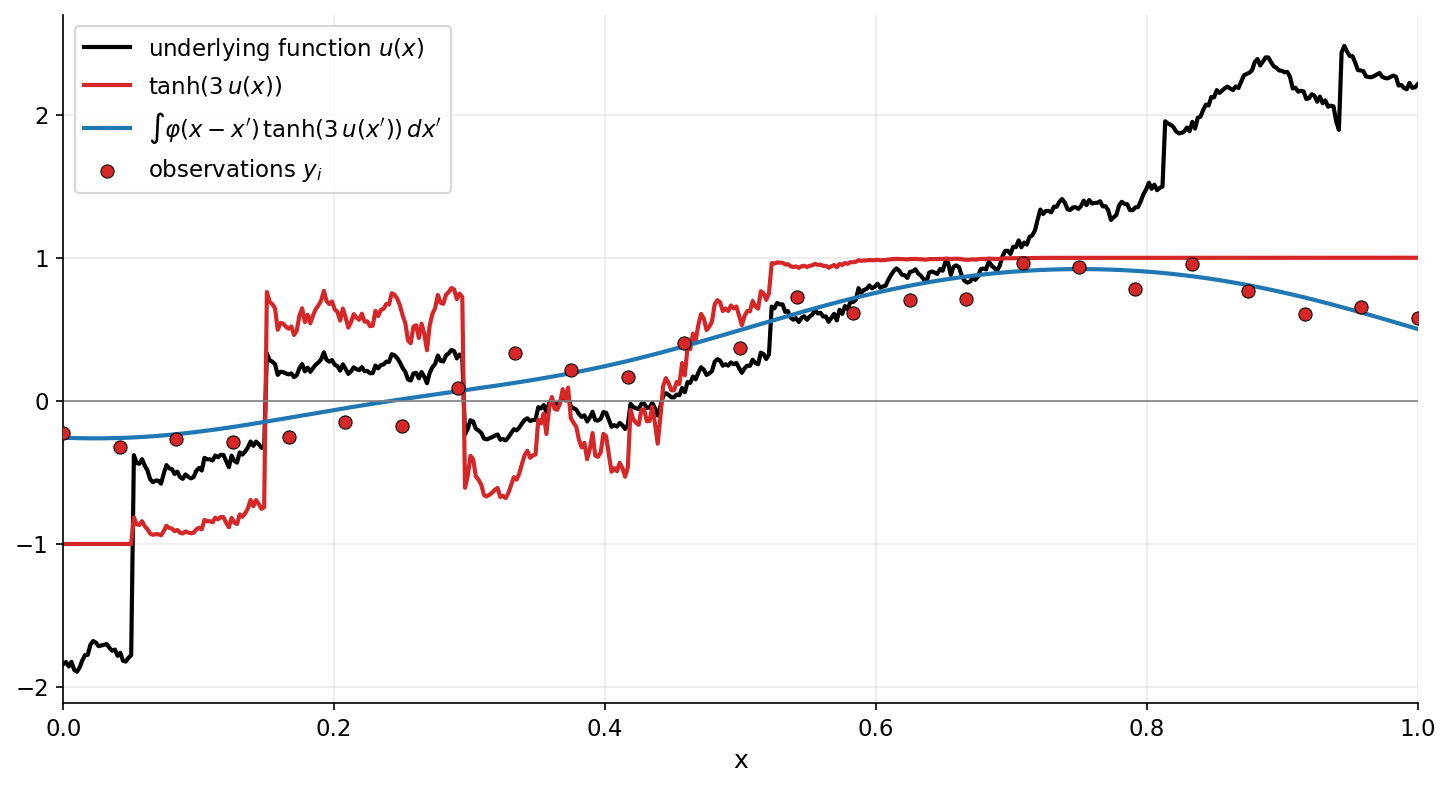}
    \end{minipage}%
    \hfill
    \begin{minipage}[t]{0.48\linewidth}
        \centering
        \vspace{0pt}
        \includegraphics[width=\linewidth]{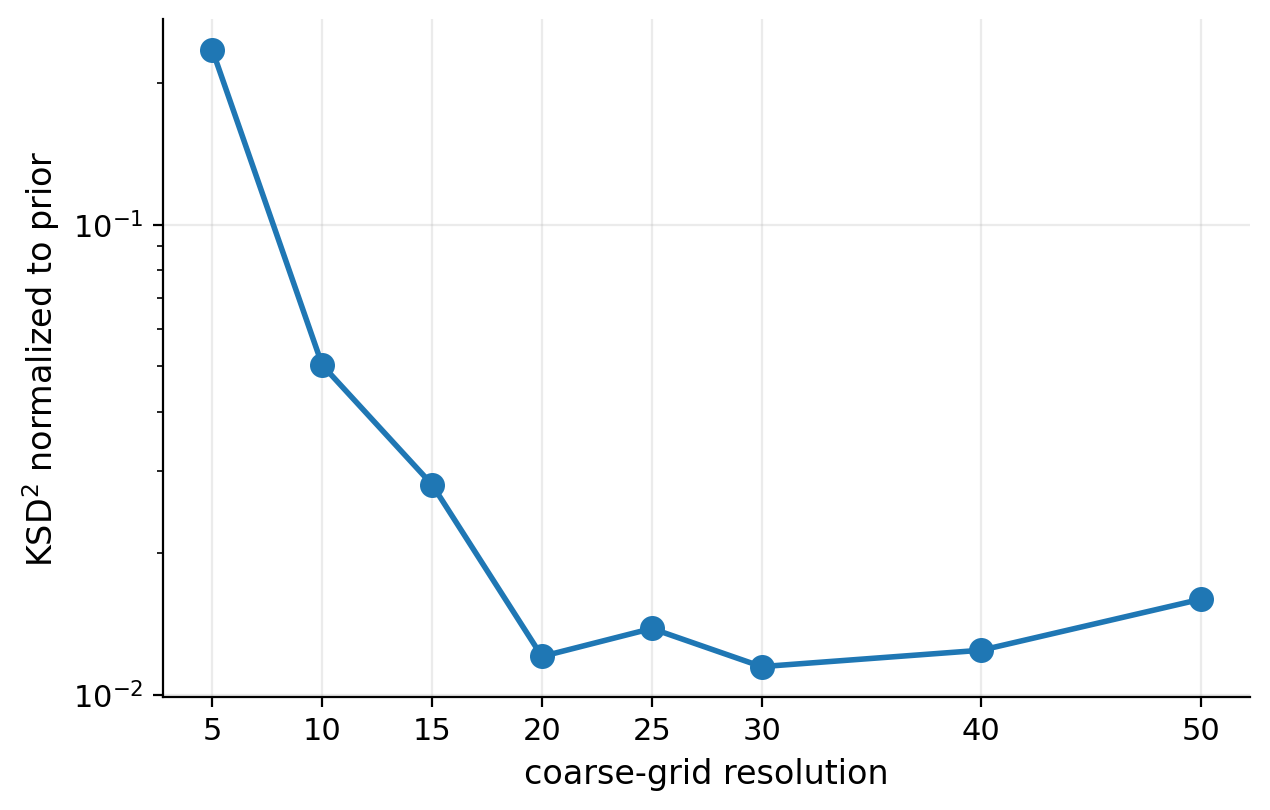}
    \end{minipage}
    \caption{
        Left: problem setup. Underlying draw \(u(x)\) in black, \(\tanh(3 u(x))\) in red, \(\int \varphi(x - x') \tanh(3u(x'))\dd x'\) in blue, and red dots representing noisy samples of \(\int \varphi(x - x') \tanh(3u(x'))\dd x'\) at the observation points.
        Right: the squared kernel Stein discrepancy of the generated posterior,
        normalized by that of the prior as the coarse-grid resolution
        is varied.
    }
    \label{fig:jump-deconv-setup}
\end{figure}

\begin{figure}[htbp]
    \centering
    \includegraphics[width=\linewidth]{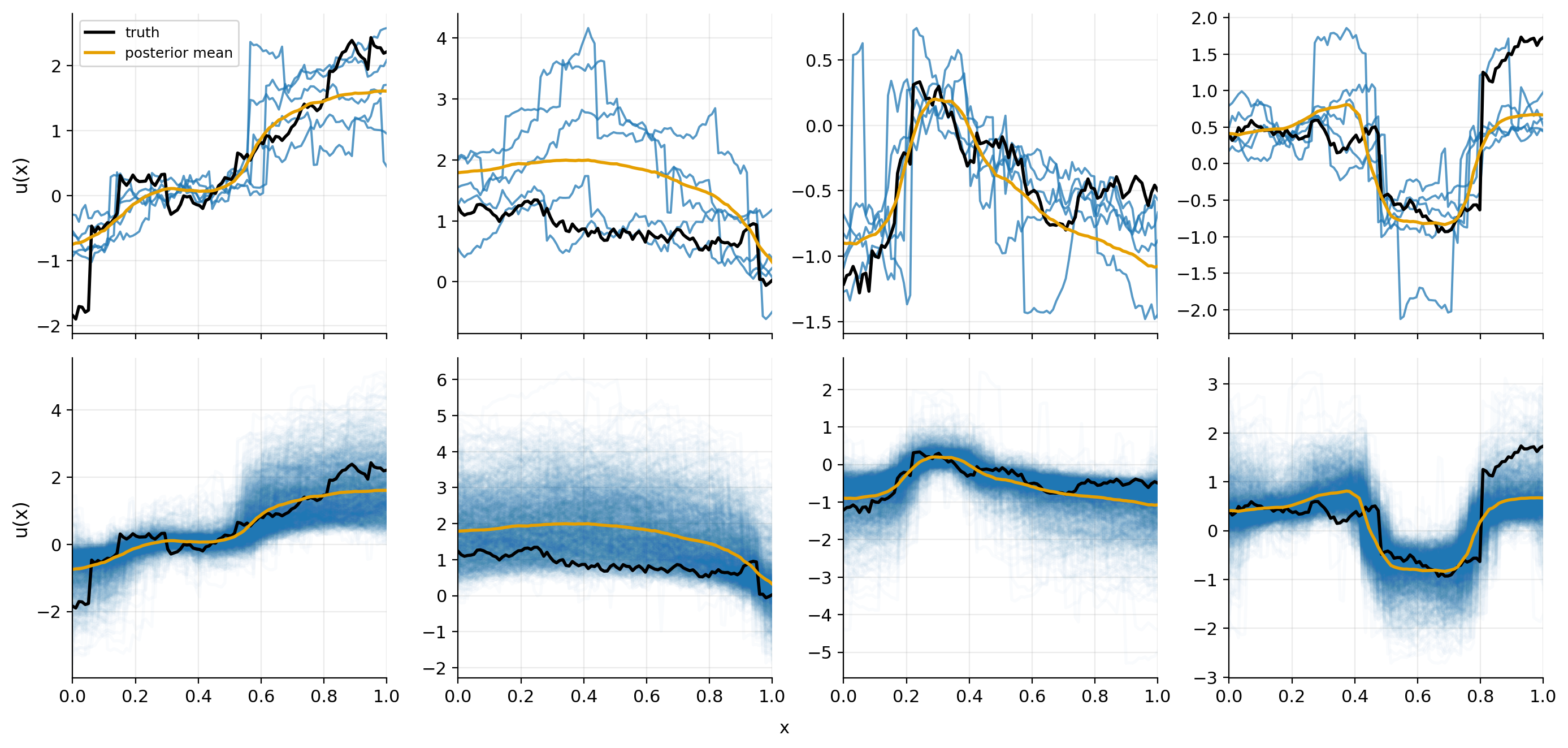}
    \caption{
        Example draws from representative posteriors. The true function \(u(x)\) is shown in black, draws from the posterior are shown in blue, and the posterior mean is shown in orange. Top row: five draws shown to illustrate the pathwise distribution. Bottom row: 1000 draws shown to illustrate the marginal distributions. Draws are computed using the conditional generative model with coarse-grid resolution 30, with the rest of the function sampled from the prior conditioned on these 30 point evaluations.
    }
    \label{fig:jump-deconv-samples}
\end{figure}

\begin{figure}[htbp]
    \centering
    \includegraphics[width=\linewidth]{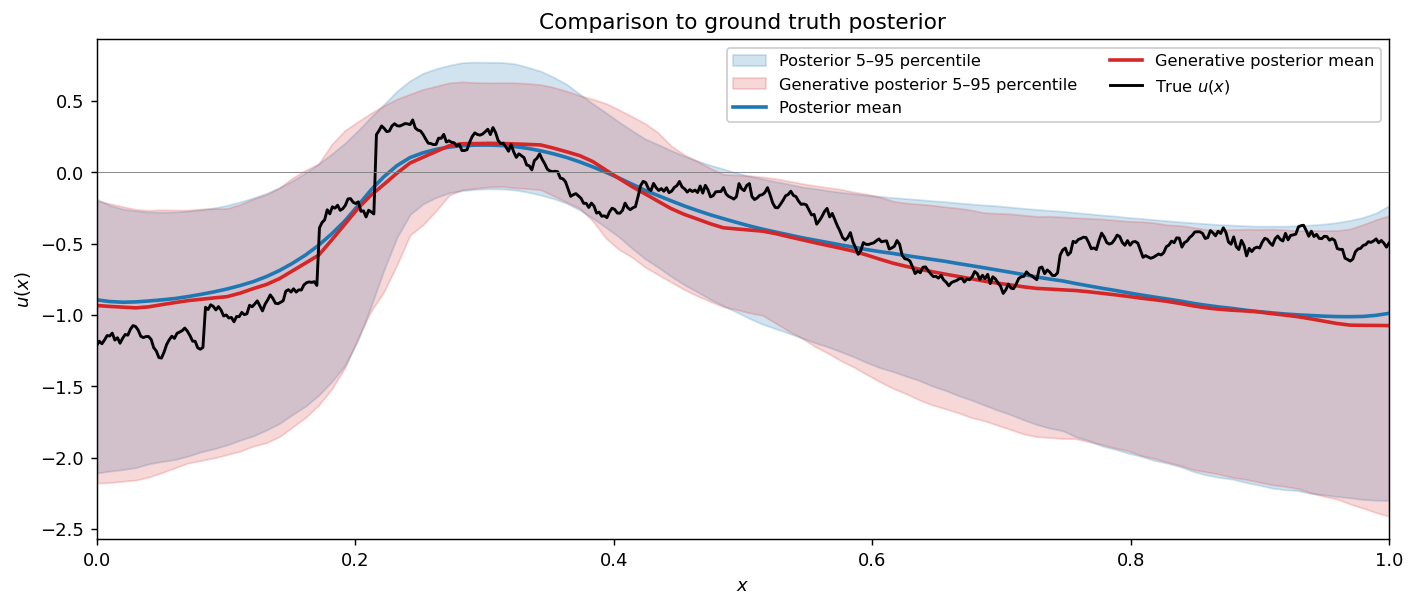}
    \caption{
        Comparison to ground truth posterior computed by MCMC applied to a fine discretization. Underlying draw \(u(x)\) in black, posterior mean in blue and generated posterior mean in red, along with the associated 5--95\% bands. The ground truth posterior is computed using multiple long MCMC chains with the No-U-Turn Sampler \cite{NUTS} implemented with BlackJax \cite{cabezas2024blackjax}.
    }
    \label{fig:jump-deconv-comparison}
\end{figure}

\subsection{Navier--Stokes filtering}\label{subsec:navier-stokes}
We consider posterior inference on the terminal vorticity of a 2D incompressible Navier--Stokes flow on the torus, given a sparse and noisy history of pointwise observations.

Let $\Omega = [0, 2\pi]^2$ with periodic boundary conditions. The vorticity field $u: \Omega \times [0,T] \to \R$ evolves according to the 2D incompressible Navier--Stokes equation in vorticity form,
\begin{equation}\label{eq:nse}
\partial_t u + (v \cdot \nabla) u = \gamma \Delta u + f,
\qquad (x, t) \in \Omega \times (0, T],
\end{equation}
with viscosity $\gamma = 5 \times 10^{-4}$ and forcing $f: \Omega \to \R$ constant in time. The velocity $v: \Omega \times [0,T] \to \R^2$ is the unique divergence-free, zero-mean field with $\mathrm{curl}\, v = u - \bar u$, where $\bar u$ denotes the spatial mean of $u$. The initial condition $u(\cdot, 0)$ is drawn from a Mat\'ern-$\tfrac52$ Gaussian random field on $\Omega$ with length scale $\rho = \pi/5$ and unit pointwise variance; the forcing $f$ is drawn once per trajectory from a band-limited Gaussian field. We take the reference measure $\RefM$ to be the law of the terminal vorticity $u(\cdot, T)$ on $L^2_{\mathrm{per}}(\Omega)$ induced by the dynamics pushing forward the distribution over initial conditions and forcings.

We observe the vorticity at $M = 16$ fixed sensor locations $\{x_m\}_{m=1}^{M} \subset \Omega$ over $K = 60$ uniformly spaced times $t_k = kT/K \in (0, T]$ with $T = 15$,
\begin{equation}\label{eq:nse-obs}
Y_{m,k} = u(x_m, t_k) + \epsilon_{m,k},
\qquad
\epsilon_{m,k} \iidsim \N(0, \sigma_\epsilon^2),
\quad \sigma_\epsilon = 0.05.
\end{equation}
Simulation provides joint draws of the sensor history and the terminal vorticity from $\TargetM$, and the target is the posterior $\TargetM(\cdot \mid Y) = \Law(u(\cdot, T) \mid Y)$. Sensor locations are the same across trajectories; only the noise and the underlying flow differ. Since the observation window ends at the inference time $T$, the target is the filtering distribution at time $T$. See \cref{fig:ns-setup}.

We take the dimensionality reduction map $\psi := \mathrm{Coarsen}_{16 \times 16} : L^2_{\mathrm{per}}(\Omega) \to \R^{16 \times 16}$, which replaces a field by its averages over the cells of a uniform $16 \times 16$ partition of $\Omega$, computed from values on the $128 \times 128$ simulation grid. Since $\RefM$ has no closed-form description, exact sampling from $\RefM(\cdot \mid \psi(u) = w)$ is intractable, so we follow \cref{subsubsec:aux-gen} and train two models: a reference conditional model whose samples approximate $\RefM(\cdot \mid \psi(u) = w)$, and a marginal model $\hat{\eta}_\psi(w \mid Y) \approx \TargetM_\psi(\cdot \mid Y)$. We denote the law of the reference conditional model's samples by $\hat{\RefM}(\cdot \mid w)$.

The approximation is built in two stages, with separately trained networks:
\begin{enumerate}
    \item The marginal model targets the coarse-grid values $w = \psi(u(\cdot, T))$ and generates \(\hat{\eta}_\psi(w \mid Y)\), conditioned on the sensor history $Y$.
    \item The reference conditional model targets the fine-grid values $u(\cdot, T)$, generates \(\hat{\RefM}(\cdot \mid w)\) conditioned on the coarse-grid values, and has no access to the sensor measurements \(Y\).
\end{enumerate}
The two-stage sampler then produces posterior samples of $u(\cdot, T) \mid Y$ by drawing $\hat{w} \sim \hat{\eta}_\psi(\cdot \mid Y)$ from the marginal model, then $\hat{u} \sim \hat{\RefM}(\cdot \mid \hat{w})$ from the reference conditional model.

Both models are FiLM-conditioned U-Nets trained by conditional flow
matching \cite{dynamic_cot,lipman2023flow,albergo2025stochastic} on
$39{,}000$ simulated trajectories. The marginal model
($\approx 0.6$M parameters) operates at the coarse $16\times16$
resolution, conditioned on the sensor history through a temporal
convolutional encoder and broadcast conditioning channels. The reference conditional model
($\approx 1.0$M parameters) operates at the fine $128\times128$
resolution, conditioned on $w$ through upsampled input channels.

At the coarse level, the marginal model achieves a continuous ranked
probability score (CRPS) of $0.30$, with a $66\%$ variance reduction
relative to the prior, that is, a mean squared error $66\%$ below that of
prediction by the prior mean. The CRPS is a proper scoring rule comparing
predicted marginal distributions with realized values
\cite{gneiting2007strictly}. Given the true coarse field $w$, the
reference conditional model attains a fine-grid mean squared error of
$0.026$. For the two-stage sampler, evaluated on $200$ test
trajectories with $64$ posterior samples each, the fine-grid CRPS is
$0.36$, the mean squared error is $0.42$, and the variance reduction
is $59\%$. The MSE is that of prediction using the posterior sample mean, and both metrics are averaged over pixels and test trajectories. The gap between the error given the true $w$ ($0.026$) and
that of the two-stage sampler ($0.42$) largely reflects the
posterior uncertainty in $w \mid Y$. Quantifying this uncertainty is
part of the inference problem.
In terms of calibration, the empirical coverage of the
pixelwise credible intervals deviates from the nominal level by
\(4.2\%\), averaged over coverage levels and pixels.
The two-stage sampler's posterior samples track the power spectrum of the truth (\cref{fig:ns-coarse-sweep}, right).

\cref{fig:ns-posterior} shows posterior samples and summaries for three test trajectories.

We further probe the role of $\psi$ by varying the coarse resolution, replacing $\psi$ with $\psi_d := \mathrm{Coarsen}_{d \times d} : L^2_{\mathrm{per}}(\Omega) \to \R^{d \times d}$ and sweeping $d$, retraining both models at each value.
\cref{fig:ns-coarse-sweep}~(left) shows the end-to-end CRPS across the
sweep, which steadily improves up to $d = 16$ with only a small marginal gain at \(d = 32\). 

\begin{figure}[htbp]
    \centering
    \includegraphics[width=0.95\linewidth]{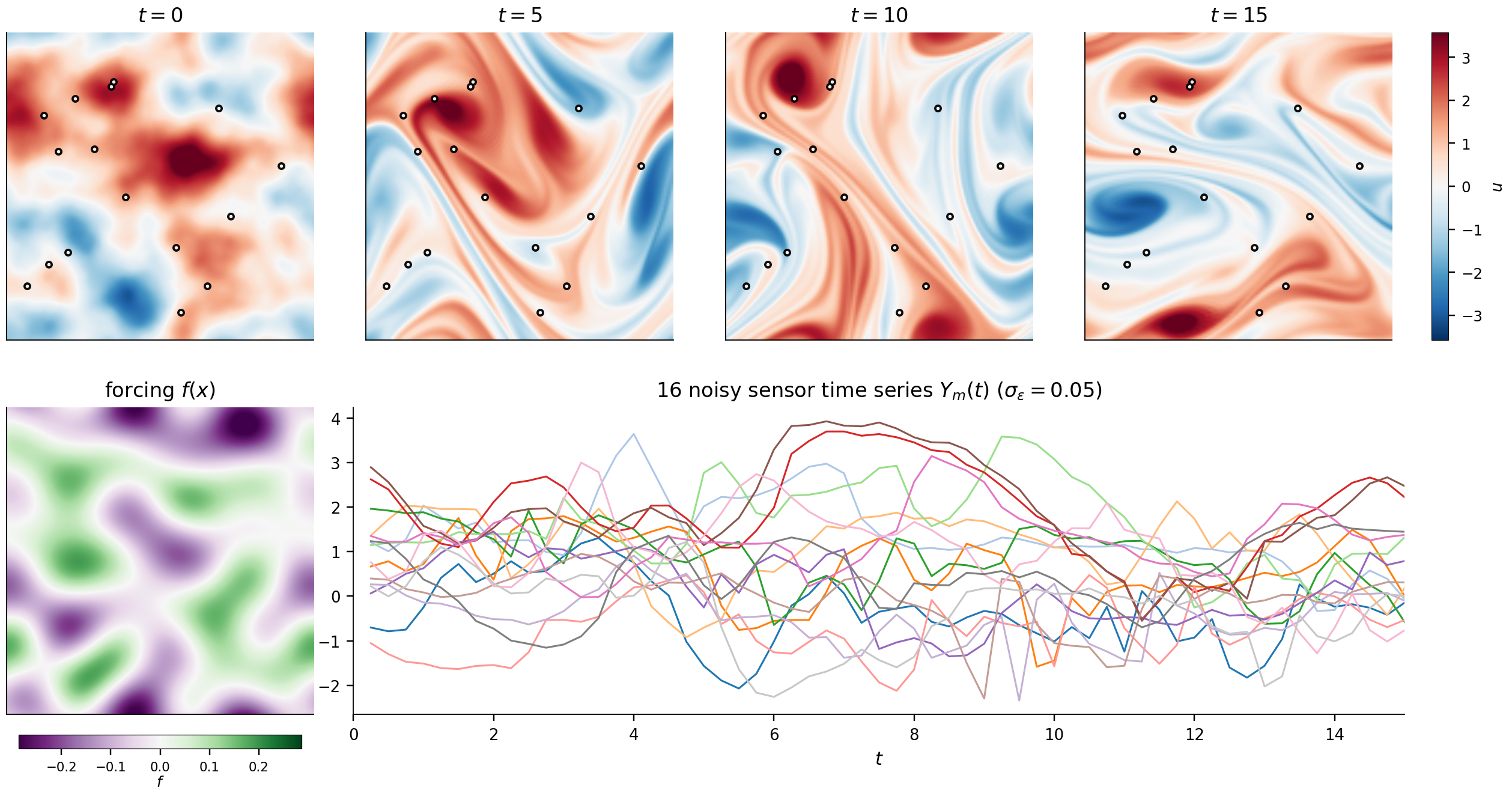}
    \caption{
        Navier--Stokes setup. Top row: snapshots of $u(\cdot, t)$ at $t \in \{0, 5, 10, 15\}$ for one trajectory, with the $M = 16$ fixed sensor locations marked. Bottom left: the time-constant forcing $f$ for this trajectory. Bottom right: the $16$ noisy sensor time series $Y_{m,\cdot}$ over $[0, T]$.
    }
    \label{fig:ns-setup}
\end{figure}

\begin{figure}[htbp]
    \centering
    \includegraphics[width=0.95\linewidth]{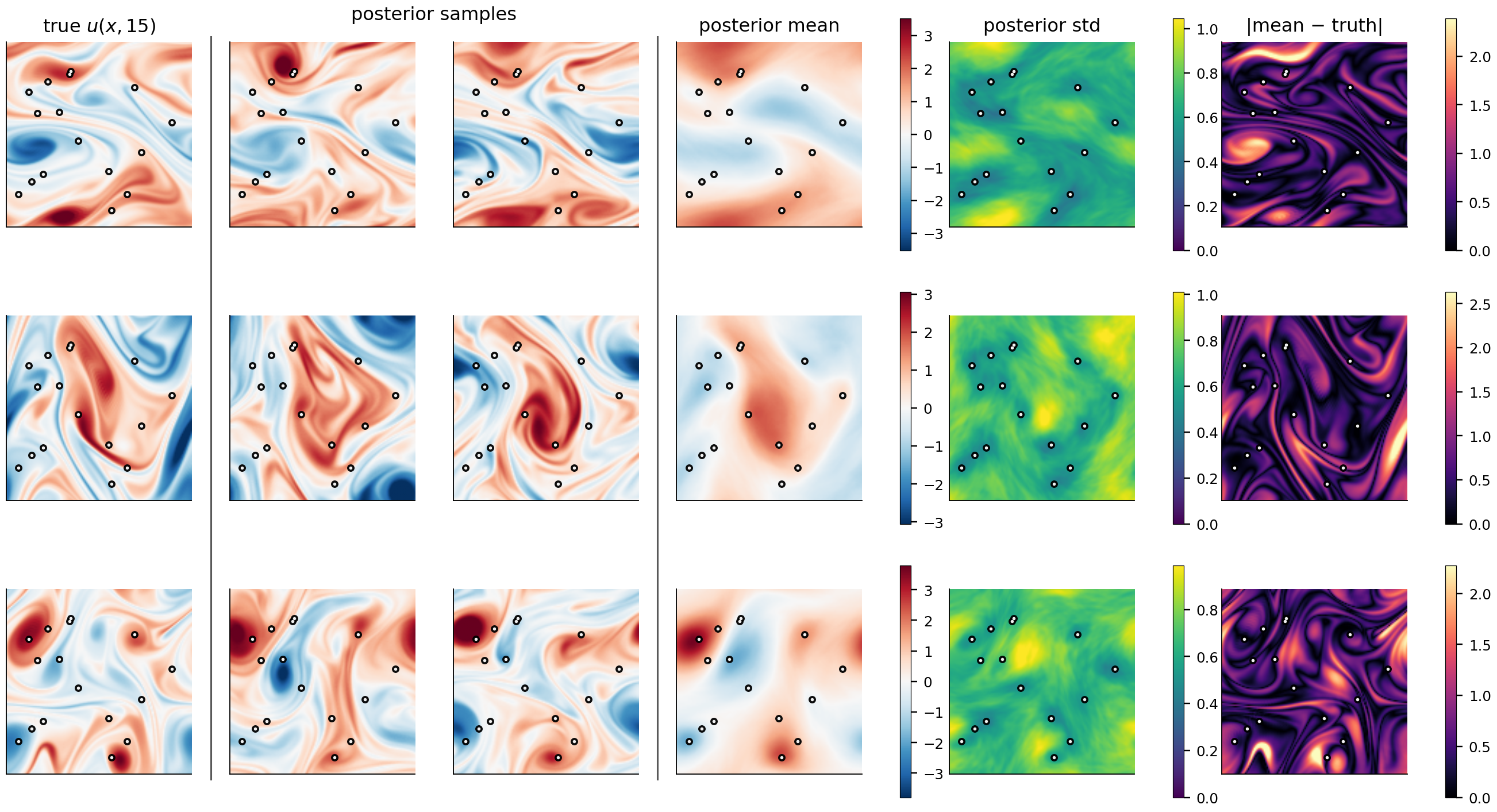}
    \caption{
        Posterior samples on three test trajectories (rows). Columns, left to right: true terminal vorticity $u(\cdot, T)$, two posterior draws from the two-stage sampler, posterior mean, pointwise posterior standard deviation, and pixelwise $|\text{mean} - \text{truth}|$, with sensor locations marked.
    }
    \label{fig:ns-posterior}
\end{figure}

\begin{figure}[htbp]
    \centering
    \begin{minipage}[t]{0.48\linewidth}
        \centering
        \vspace{0pt}
        \includegraphics[width=\linewidth]{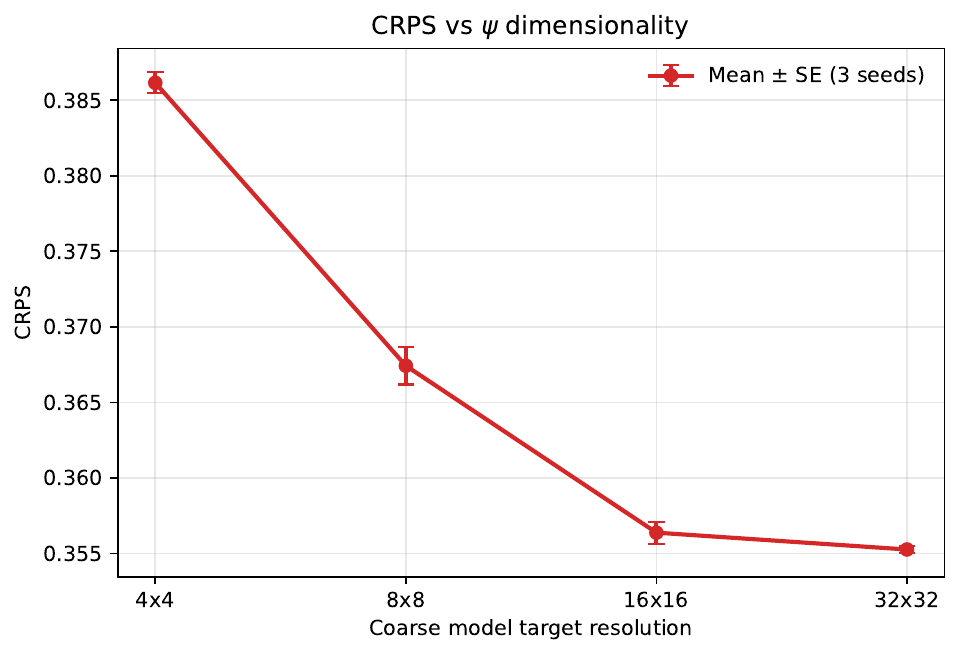}
    \end{minipage}%
    \hfill
    \begin{minipage}[t]{0.48\linewidth}
        \centering
        \vspace{0pt}
        \includegraphics[width=\linewidth]{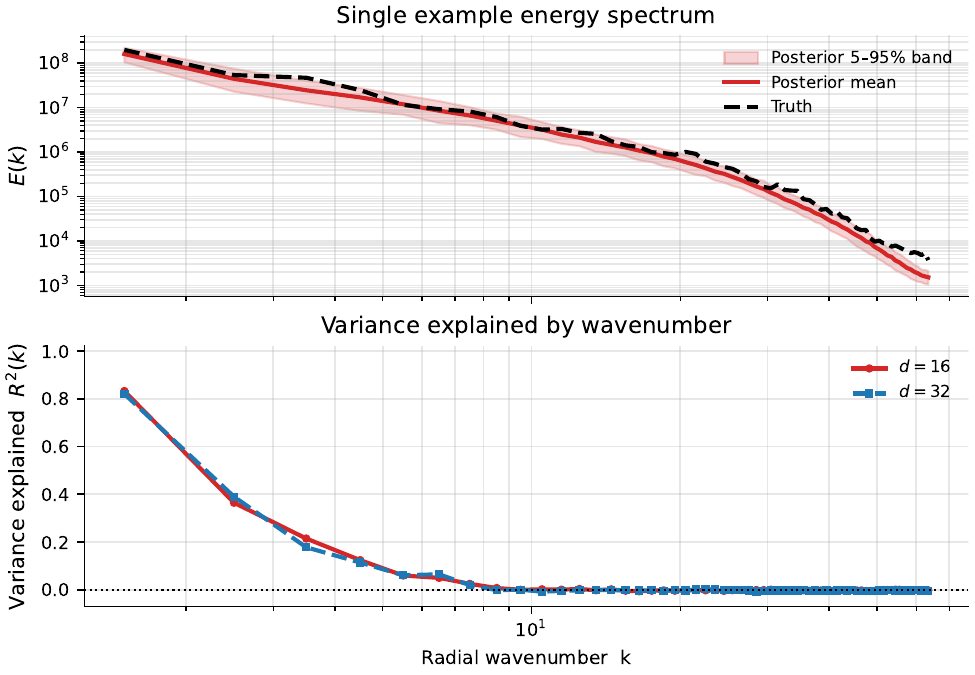}
    \end{minipage}
    \caption{
        Left: end-to-end CRPS at the fine resolution as a function of the coarse target resolution $d \times d$, sweeping $d \in \{4, 8, 16, 32\}$. Both networks were retrained at each value. The marginal model \(\hat{\eta}_{\psi}(w \mid y)\) was trained  on three different random seeds for each resolution, and we show the mean and \(\pm 1\) SE across these seeds. CRPS decreases between $d = 4$ and $d = 16$, with only a small further gain at $d = 32$.
        Right: radially averaged power spectrum of the posterior samples (mean and $5$--$95\%$ band) against the truth for a single example at \(d = 16\) (top), and the per-wavenumber variance-explained statistic $R^2(k)$ (bottom) for models with coarse resolution \(d = 16\) and \(d = 32\). The explained variance is concentrated at low wavenumbers and falls to near zero by $k \approx 8$ for both models, suggesting that wavenumbers above this are no longer informed by the data. Since the $d = 16$ coarse field resolves wavenumbers up to $k = 8$ and a $d = 32$ field up to $k = 16$, the absence of any gain at $d = 32$ for high wavenumbers indicates that this cutoff reflects the data rather than the choice of $\psi$.
    }
    \label{fig:ns-coarse-sweep}
\end{figure}

\section{Error analysis}\label{sec:analysis}

We now consider error analysis for approximations based on the marginalized measure \(\pma\). The measures produced by
\cref{alg:reduced-gen,alg:reduced-cond-gen} combine a
generative model fitted to the \(d\)-dimensional marginal \(\psi \sharp
\TargetM\) with the conditionals of the reference measure. We
decompose the error into the quality of the marginal approximation and
an irreducible error associated with approximation by measures in
\(\psiclass\). In \cref{subsec:wasserstein} we bound Wasserstein
distances through a nested OT problem, assuming stability of conditionals with respect to perturbations of the conditioning variable.
In \cref{subsec:PMA-error}, we consider a natural decomposition of the Kullback--Leibler divergence which splits into marginal error and irreducible error. We develop bounds on the irreducible part based on low-order conditional moments of the reference measure.
In \cref{subsec:functional}, we consider cases where the reference
conditionals satisfy functional inequalities, recovering the
subspace-selection diagnostic of likelihood-informed subspace methods.
In each case the infinite-dimensional structure is carried by the
reference conditionals, and the error from the marginal term
is that of a finite-dimensional generative modeling problem.
Convergence guarantees for generative models on \(\R^d\) therefore
transfer directly to the function-space approximation.

We note that the Kullback--Leibler divergence does not depend on the norm defining the ambient space
\(\U\), and neither does the class \(\psiclass\), while the upper bounds of
this section do. An ambient norm enters as the metric of the Wasserstein
distances in \cref{subsec:wasserstein}, through the Lipschitz
constants and conditional moments in \cref{subsec:PMA-error}, and
through the gradient norms in \cref{subsec:functional}. Each bound
holds for any norm under which its hypotheses are satisfied, so for
the Kullback--Leibler bounds, the norm is a free parameter, with weaker
norms demanding more smoothness from the likelihood and less from the reference.
However, the weighted inequalities of \cref{rem:weighted-fi} absorb this dependence
into an operator \(\Gamma\) which depends on the reference measure,
giving bounds that are invariant to the choice of ambient space.

Throughout this section, we adopt the normalization \(\einf_\RefM(\Phi) = 0\), so that \(\dd\TargetM/\dd\RefM = e^{-\Phi}/Z \leq 1/Z\).
\subsection{Wasserstein bounds}\label{subsec:wasserstein}

We bound Wasserstein distances between a measure \(\eta \in \psiclass\)
and the target \(\TargetM\). The bound has the structure of a nested
OT problem, consisting of an outer coupling of the \(\psi\)-marginals on
\(\R^d\), whose transport cost is itself an OT distance between conditionals. When the marginal measures match, this reduces to a conditional OT problem \cite{hosseini2025conditional,chemseddine2025conditional,manupriya2024consistent,baptista2025knothe,baptista2025conditionalSimulation}.

\begin{proposition}[Nested optimal transport bound]\label{prop:Wp-nested}
    Let \(p \in [1,\infty)\), and let \(\eta, \TargetM \in \mathcal{P}(\U)\)
    have finite \(p\)-th moment. Then
    \begin{equation}\label{eqn:nested-ot}
        W_{p}^{p}(\eta,\TargetM)
        \leq \inf_{\pi_{\psi}\in\Pi(\eta_{\psi},\TargetM_{\psi})}
        \int_{\R^d \times \R^d} W_{p}^{p}\left(\eta^w,\TargetM^z\right)
        \pi_{\psi}(\dd w, \dd z).
    \end{equation}
    In particular, if \(\eta \in \psiclass\), then
    \begin{equation}\label{eqn:nested-ot-psi}
        W_{p}^{p}(\eta,\TargetM)
        \leq \inf_{\pi_{\psi}\in\Pi(\eta_{\psi},\TargetM_{\psi})}
        \int_{\R^d \times \R^d} W_{p}^{p}\left(\RefM^w,\TargetM^z\right)
        \pi_{\psi}(\dd w, \dd z).
    \end{equation}
\end{proposition}
\begin{proof}
    Let \(\pi_\psi \in \Pi(\eta_{\psi},\TargetM_{\psi})\) be an arbitrary
    coupling, and for each \((w,z) \in \R^d \times \R^d\), let
    \(\pi^{w,z} \in \Pi(\eta^w, \TargetM^z)\) be an optimal coupling for
    \(W_p\), chosen so that \((w,z) \mapsto \pi^{w,z}\) is measurable
    \cite[Corollary~5.22]{villani2009optimal}.
    Define the measure
    \begin{equation}
        \pi(\dd u, \dd u') :=
        \int_{\R^d\times\R^d} \pi^{w,z}(\dd u, \dd u')\pi_\psi(\dd w, \dd z)
    \end{equation}
    on \(\U \times \U\). Disintegrating along the fibers of \(\psi\), the first
    marginal of \(\pi\) is \(\int \eta^w \eta_\psi(\dd w) = \eta\), and
    similarly the second marginal is \(\TargetM\), so
    \(\pi \in \Pi(\eta, \TargetM)\). Therefore
    \begin{equation}
        W_p^p(\eta, \TargetM)
        \leq \int_{\U\times\U} \|u - u'\|^p \pi(\dd u, \dd u')
        = \int_{\R^d\times\R^d} W_p^p\left(\eta^w, \TargetM^z\right)
        \pi_\psi(\dd w, \dd z),
    \end{equation}
    and taking the infimum over \(\pi_\psi\) gives \cref{eqn:nested-ot}.
    When \(\eta \in \psiclass\), we have \(\eta^w = \RefM^w\) for
    \(\eta_\psi\)-a.e.\ \(w\), giving \cref{eqn:nested-ot-psi}.
\end{proof}

In \cref{eqn:nested-ot-psi}, the choice of \(\eta \in \psiclass\) only affects
the outer coupling through the marginal \(\eta_\psi\), while the inner cost
\(W_p^p(\RefM^w, \TargetM^z)\) is fixed by the reference and target. For the
marginalized measure \(\tilde\eta\) of \cref{eqn:opt-forward}, the outer
problem degenerates: \(\tilde\eta_\psi = \TargetM_\psi\), so the identity
coupling is admissible and the bound collapses to an average of conditional
transport costs along the diagonal.

\begin{corollary}\label{cor:marginal-Wp}
    Let \(\tilde\eta\) be the marginalized measure of \cref{eqn:opt-forward}
    and \(p \in [1,\infty)\). Then
    \begin{equation}
        W_{p}^{p}(\tilde\eta, \TargetM)
        \leq \EE_{W\sim\TargetM_\psi}
        \left[W_{p}^{p}\left(\RefM^W, \TargetM^W\right)\right].
    \end{equation}
\end{corollary}
\begin{proof}
    Since \(\tilde\eta_\psi = \TargetM_\psi\), the identity coupling
    \((\mathrm{Id},\mathrm{Id})\sharp\TargetM_\psi
    \in \Pi(\tilde\eta_\psi, \TargetM_\psi)\) is admissible in
    \cref{eqn:nested-ot-psi}, and the outer transport cost concentrates on
    the diagonal \(w = z\).
\end{proof}

The identity coupling in \cref{cor:marginal-Wp} is available because the marginalized measure matches \(\TargetM_\psi\) exactly. When the marginal is instead \emph{fitted}, as in \cref{alg:reduced-gen,alg:reduced-cond-gen}, the nested structure of \cref{prop:Wp-nested} propagates the marginal fitting error through the lift, provided the reference conditionals vary smoothly with the conditioning value.

\begin{proposition}[Stability of the lift]\label{prop:lift-stability}
    Let \(p \in [1,\infty)\) and \(\eta \in \psiclass\) and suppose the reference conditionals are \(K\)-Lipschitz in Wasserstein distance, that is, \(W_p(\RefM^w, \RefM^z) \leq K\|w - z\|\) for all \(w, z \in \operatorname{supp}\RefM_\psi\).
    Then
    \begin{equation*}
        W_p(\eta, \TargetM)
        \leq K W_p(\eta_\psi, \TargetM_\psi)
        + \left(\EE_{W\sim\TargetM_\psi}\left[W_p^p\left(\RefM^W, \TargetM^W\right)\right]\right)^{1/p}.
    \end{equation*}
\end{proposition}
\begin{proof}
    By the triangle inequality, \(W_p(\eta,\TargetM) \leq W_p(\eta, \tilde\eta) + W_p(\tilde\eta, \TargetM)\), and the second term is bounded by \cref{cor:marginal-Wp}. For the first, both \(\eta\) and \(\tilde\eta\) lie in \(\psiclass\), with \(\eta^w = \RefM^w\), \(\tilde\eta^z = \RefM^z\), and \(\tilde\eta_\psi = \TargetM_\psi\), so \cref{eqn:nested-ot} gives
    \begin{equation*}
        W_p^p(\eta, \tilde\eta)
        \leq \inf_{\pi_\psi \in \Pi(\eta_\psi, \TargetM_\psi)} \int_{\R^d\times\R^d} W_p^p\left(\RefM^w, \RefM^z\right)\pi_\psi(\dd w, \dd z)
        \leq K^p W_p^p(\eta_\psi, \TargetM_\psi).
    \end{equation*}
\end{proof}

\begin{remark}[Gaussian references]\label{rem:gauss-K}
    For Gaussian \(\RefM\), Matheron's update \cref{eq:matheron} transports \(\RefM^w\) to \(\RefM^z\) \emph{exactly} by the shift \(u \mapsto u + \Sigma_{u\psi}\Sigma_{\psi\psi}^{-1}(z-w)\), so that
    \(W_p(\RefM^w,\RefM^z) \leq \|\Sigma_{u\psi}\Sigma_{\psi\psi}^{-1}\|\|w-z\|\) for every \(p\), and \cref{prop:lift-stability} holds with the explicit constant \(K = \|\Sigma_{u\psi}\Sigma_{\psi\psi}^{-1}\|\).
\end{remark}

\begin{corollary}[Learned reference conditionals]\label{cor:learned-cond}
    Let \(\hat{\eta}\) denote the law of the two-stage sampler of \cref{subsubsec:aux-gen}: \(W \sim \hat{\eta}_\psi\), then \(U \sim \hat{\RefM}(\cdot \mid W)\), where \(\hat{\RefM}(\cdot \mid w)\) approximates \(\RefM^w\). Under the hypotheses of \cref{prop:lift-stability}, and assuming \(\hat{\eta}_\psi(\operatorname{supp}\RefM_\psi) = 1\),
    \begin{equation*}
    \begin{split}
        W_p(\hat{\eta}, \TargetM)
        \leq \left(\EE_{W\sim\hat{\eta}_\psi}\left[W_p^p\left(\hat{\RefM}(\cdot \mid W), \RefM^W\right)\right]\right)^{1/p}
        &+ K W_p(\hat{\eta}_\psi, \TargetM_\psi) \\
        &+ \left(\EE_{W\sim\TargetM_\psi}\left[W_p^p\left(\RefM^W, \TargetM^W\right)\right]\right)^{1/p}.
    \end{split}
    \end{equation*}
\end{corollary}
\begin{proof}
    Let \(\eta \in \psiclass\) be the measure with marginal \(\eta_\psi = \hat{\eta}_\psi\) and conditionals \(\RefM^w\). Coupling \(\hat{\eta}\) and \(\eta\) by drawing the same \(W \sim \hat{\eta}_\psi\) and then coupling \(\hat{\RefM}(\cdot\mid W)\) with \(\RefM^W\) optimally gives
    \(W_p^p(\hat{\eta}, \eta) \leq \EE_{W\sim\hat{\eta}_\psi}[W_p^p(\hat{\RefM}(\cdot\mid W), \RefM^W)]\); conclude by the triangle inequality and \cref{prop:lift-stability} applied to \(\eta\).
\end{proof}

Each conditional Wasserstein distance appearing in the integrals of
\cref{cor:marginal-Wp,prop:lift-stability} compares \(\RefM^w\)
with \(\TargetM^w(\dd u) = \frac{1}{Z_w}\exp(-\Phi(u))\RefM^w(\dd u)\),
that is, it compares \(\frac{1}{Z_w}\exp(-\Phi(u))\) to a constant likelihood. This simple structure allows us to obtain bounds based on low-order moments of \(\RefM^w\) using duality.

\begin{proposition}[Conditional Wasserstein bound]\label{prop:fiber-w1}
    For \(\RefM_\psi\)-almost every
    \(w\), let \(Z_w := \EE_{\RefM^w}[e^{-\Phi}]\), and for independent
    \(U, V \sim \RefM^w\) let
    \(\sigma_w^2 := \tfrac12 \EE\big\|U - V\big\|^2 < \infty\). Then
    \begin{equation*}
        W_1(\RefM^w, \TargetM^w)
        \leq \sigma_w \sqrt{\frac{1 - Z_w}{Z_w}}.
    \end{equation*}
    If, in addition, \(\Phi\) is \(C(w)\)-Lipschitz along the fiber \(\{\psi(u) = w\}\) then
    \begin{equation*}
        W_1(\RefM^w, \TargetM^w)
        \leq \frac{C(w) \sigma_w^2}{Z_w}.
    \end{equation*}
\end{proposition}
\begin{proof}
    By Kantorovich--Rubinstein duality on the fiber and
    \(\dd\TargetM^w/\dd\RefM^w = e^{-\Phi}/Z_w\),
    \begin{equation*}
        W_1(\RefM^w, \TargetM^w)
        = \sup_{\mathrm{Lip}(f) \leq 1}
        \left( \EE_{\TargetM^w}[f] - \EE_{\RefM^w}[f] \right)
        = \frac{1}{Z_w} \sup_{\mathrm{Lip}(f) \leq 1}
        \Cov_{\RefM^w}\left(f,\ e^{-\Phi}\right).
    \end{equation*}
    For independent \(U, V \sim \RefM^w\) and square-integrable
    \(f, g\),
    \begin{equation*}
        \Cov(f, g) = \tfrac12
        \EE\left[(f(U) - f(V))(g(U) - g(V))\right].
    \end{equation*}
    Cauchy--Schwarz gives
    \(\Cov_{\RefM^w}(f, e^{-\Phi}) \leq
    \sqrt{\Var_{\RefM^w}(f)} \sqrt{\Var_{\RefM^w}(e^{-\Phi})}\),
    and for \(1\)-Lipschitz \(f\) the covariance identity gives
    \(\Var_{\RefM^w}(f) \leq \tfrac12 \EE\|U - V\|^2 = \sigma_w^2\),
    so
    \begin{equation*}
        W_1(\RefM^w, \TargetM^w)
        \leq \frac{\sigma_w}{Z_w}
        \sqrt{\Var_{\RefM^w}\left(e^{-\Phi}\right)}.
    \end{equation*}
    Since \(e^{-2\Phi} \leq e^{-\Phi}\) pointwise,
    \(\Var_{\RefM^w}(e^{-\Phi}) = \EE[e^{-2\Phi}] - Z_w^2 \leq
    Z_w - Z_w^2\), which gives the first bound. When \(\Phi\) is
    \(C(w)\)-Lipschitz, since \(\Phi \geq 0\) and
    \(t \mapsto e^{-t}\) is \(1\)-Lipschitz on \([0, \infty)\),
    \begin{equation*}
        \Var_{\RefM^w}(e^{-\Phi})
        = \tfrac12 \EE\left[\left(e^{-\Phi(U)} - e^{-\Phi(V)}\right)^2\right]
        \leq \tfrac{C(w)^2}{2} \EE\|U - V\|^2 = C(w)^2 \sigma_w^2,
    \end{equation*}
    which gives the second.
\end{proof}

The two bounds are tight in opposite regimes. For
\(\RefM^w = \tfrac12(\delta_0 + \delta_r)\) and
\(\Phi(u) = M|u|\), the first bound is tight as
\(Mr \to \infty\), saturating at \(W_1 \to r/2\), while the
Lipschitz bound matches \(W_1\) with constant one as
\(Mr \to 0\).

Integrating over the conditioning variable costs only the global
normalizing constant, since
\(\dd\TargetM_\psi/\dd\RefM_\psi(w) = Z_w/Z\):
\begin{equation*}
    \EE_{W \sim \TargetM_\psi}\left[W_1(\RefM^W, \TargetM^W)\right]
    \leq \frac{1}{Z} \EE_{W \sim \RefM_\psi}\left[
    C(W) \sigma_W^2 \right],
\end{equation*}
which, through \cref{cor:marginal-Wp} at \(p = 1\), bounds
\(W_1(\pma, \TargetM)\), and likewise bounds the second term of \cref{prop:lift-stability}; the
smaller of the two fiber bounds may be used at each \(w\).

\subsection{Kullback--Leibler bounds}\label{subsec:PMA-error}
The forward KL best approximation $\tilde{\eta}$ from \cref{eqn:opt-forward}
matches the $\psi$-marginal of $\TargetM$ while retaining the conditionals
of $\RefM$, and is targeted in practice by
\cref{alg:reduced-gen,alg:reduced-cond-gen}. Its Radon--Nikodym derivative
with respect to $\RefM$ is

\begin{equation}\label{eqn:pma-def}
    \dd\pma(u) = \frac{1}{\tilde{Z}}\exp(-\tilde{\phi}(\psi(u)))\dd\RefM(u), \quad
    \tilde{\phi}(w) := -\log\EE_\RefM\left[\exp(-\Phi(U)) \mid \psi(U) = w\right].
\end{equation}

We first note that $\tilde{Z} = Z$. Indeed, by the tower property of conditional expectation,
\begin{align*}
    \tilde{Z} &= \EE_\RefM\left[\exp(-\tilde{\phi}(\psi(U)))\right]
    = \EE_\RefM\left[\EE_\RefM\left[\exp(-\Phi(U)) \mid \psi(U)\right]\right]
    = \EE_\RefM\left[\exp(-\Phi(U))\right] = Z.
\end{align*}

The error of \emph{any} approximation in \(\psiclass\) decomposes around \(\pma\). Recall the chain rule for the KL divergence from the proof of \cref{prop:best-approx}: for \(\eta \in \psiclass\), the conditionals of \(\eta\) agree with those of \(\RefM\) (either both sides below are infinite, or \(\TargetM_\psi \ll \eta_\psi\) and the conditionals agree \(\TargetM_\psi\)-almost everywhere), so
\begin{equation}\label{eqn:kl-chain-psi}
    \kl(\TargetM, \eta) = \kl(\TargetM_\psi, \eta_\psi)
    + \EE_{W \sim \TargetM_\psi}\left[\kl\left(\TargetM^W, \RefM^W\right)\right].
\end{equation}
Taking \(\eta = \pma\), the first term vanishes since \(\pma_\psi = \TargetM_\psi\), which identifies the second term as
\begin{equation}\label{eqn:kl-irreducible}
    \kl(\TargetM, \pma) = \EE_{W\sim\TargetM_\psi}\left[\kl\left(\TargetM^W, \RefM^W\right)\right],
\end{equation}
so that for every \(\eta \in \psiclass\),
\begin{equation}\label{eqn:kl-pythagoras}
    \kl(\TargetM,\eta) = \kl(\TargetM_\psi, \eta_\psi) + \kl(\TargetM, \pma).
\end{equation}
The first term in \cref{eqn:kl-pythagoras} is the error of the \(d\)-dimensional generative model fitted to the marginal \(\TargetM_\psi\), which is the only quantity that \cref{alg:reduced-gen,alg:reduced-cond-gen} control, while the second is irreducible within \(\psiclass\). In the finite-dimensional setting with \(\psi\) a linear projection, the identity \cref{eqn:kl-pythagoras} appears in \cite[Section~2.1]{certified}. The reverse divergence decomposes analogously,
\begin{equation}\label{eqn:kl-chain-reverse}
    \kl(\eta, \TargetM) = \kl(\eta_\psi, \TargetM_\psi) + \EE_{W\sim\eta_\psi}\left[\phi^{*}(W)\right],
    \qquad \phi^{*}(w) = \kl(\RefM^w, \TargetM^w),
\end{equation}
quantifying in the same way the suboptimality of the variational approximation produced by \cref{alg:vi}. When the reference conditional is itself learned, as in \cref{subsec:navier-stokes}, the KL decomposition contains an additional term. From the same decomposition
\begin{equation}
    \kl(\TargetM, \hat{\eta})
    = \kl(\TargetM_\psi, \hat{\eta}_\psi)
    + \EE_{W\sim\TargetM_\psi}\!\left[\kl\big(\TargetM^W, \hat{\RefM}(\cdot \mid W)\big)\right],
\end{equation}
the fiber term can be further decomposed relative to the idealized reference conditional. Assuming absolute continuity of the conditionals, \(\TargetM^w \ll \RefM^w \ll \hat{\RefM}(\cdot \mid w)\), we have
\begin{equation}
    \kl\big(\TargetM^w, \hat{\RefM}(\cdot \mid w)\big)
    = \kl(\TargetM^w, \RefM^w)
    + \EE_{\TargetM^w}\!\left[\log \frac{\dd\RefM^w}{\dd\hat{\RefM}(\cdot \mid w)}\right].
\end{equation}
In the remainder of this subsection, we focus on bounding the irreducible error \cref{eqn:kl-irreducible}.

The following proposition bounds the irreducible error by the
expected symmetrized (Jeffreys) divergence of the conditionals, which
reduces to an explicit difference of conditional expectations of
\(\Phi\).
\begin{proposition}[Jeffreys bound]\label{prop:jeffreys}
For \(\TargetM_\psi\)-almost every \(w\),
\begin{equation}\label{eq:jeffreys-fiber}
    \kl(\TargetM^w, \RefM^w) + \kl(\RefM^w, \TargetM^w)
    = \EE_{\RefM}\left[\Phi \mid \psi(U) = w\right]
    - \EE_{\TargetM}\left[\Phi \mid \psi(U) = w\right],
\end{equation}
and consequently
\begin{equation}\label{eq:jeffreys-bound}
    \kl\left(\TargetM, \pma\right)
    \leq \EE_{W\sim\TargetM_\psi}\left[
    \kl(\TargetM^W, \RefM^W) + \kl(\RefM^W, \TargetM^W)\right]
    = \EE_{\TargetM}\left[\EE_{\RefM}\left[\Phi(U)\mid \psi(U)\right]-\Phi(U)\right].
\end{equation}
\end{proposition}

\begin{proof}
Since \(\dd\TargetM^w/\dd\RefM^w = e^{\tilde{\phi}(w) - \Phi}\),
taking expectations of the log-density ratio under \(\TargetM^w\) and
under \(\RefM^w\) gives
\begin{equation*}
    \kl(\TargetM^w, \RefM^w)
    = \tilde{\phi}(w) - \EE_\TargetM\left[\Phi \mid \psi(U) = w\right],
    \qquad
    \kl(\RefM^w, \TargetM^w)
    = \EE_\RefM\left[\Phi \mid \psi(U) = w\right] - \tilde{\phi}(w),
\end{equation*}
and summing gives \cref{eq:jeffreys-fiber}. The inequality in
\cref{eq:jeffreys-bound} follows from \cref{eqn:kl-irreducible} and
the nonnegativity of \(\kl(\RefM^W, \TargetM^W)\). The equality
follows from averaging \cref{eq:jeffreys-fiber} over
\(W \sim \TargetM_\psi\), since, by iterated expectation,
\begin{equation*}
    \EE_{W\sim\TargetM_\psi}\left[\EE_{\TargetM}\left[\Phi\mid \psi(U)
    = W\right]\right] = \EE_{\TargetM}\left[\Phi(U)\right].
\end{equation*}
\end{proof}

By \cref{eqn:kl-irreducible}, the middle term of
\cref{eq:jeffreys-bound} exceeds \(\kl(\TargetM,\pma)\) by exactly
\(\EE_{W\sim\TargetM_\psi}\left[\phi^{*}(W)\right]\), the reverse
suboptimality term of \cref{eqn:kl-chain-reverse}; the bounds derived
below therefore control both terms simultaneously. Equivalently, the
excess \(\phi^{*}(w) = \EE_\RefM\left[\Phi \mid \psi(U) = w\right] -
\tilde{\phi}(w)\) is the gap in Jensen's inequality
\(\tilde{\phi}(w) = -\log\EE_\RefM\left[e^{-\Phi} \mid \psi(U) =
w\right] \leq \EE_\RefM\left[\Phi \mid \psi(U) = w\right]\). Bounds
in total variation and Hellinger distance follow from the
Kullback--Leibler bounds here and below via Pinsker's inequality and
\(d_H^2 \leq \kl\), with \(d_H\) the Hellinger distance defined in \cref{subsec:functional}.

When \(\Phi\) is Lipschitz along the fibers of \(\psi\), \cref{prop:jeffreys} already yields an explicit rate, requiring no linear structure on \(\psi\) and no assumptions on the reference conditionals beyond first moments.

\begin{corollary}\label{cor:kl-lipschitz}
    Suppose that \(|\Phi(u) - \Phi(u')| \leq C\|u - u'\|\) whenever \(\psi(u) = \psi(u')\). Then for any measurable section \(u^*\) of \(\psi\), that is, any measurable map \(u^*:\R^d\to\U\) satisfying \(\psi(u^*(w)) = w\),
    \begin{equation*}
        \kl(\TargetM, \pma) \leq \frac{2C}{Z}\EE_\RefM\left[\|U - u^*(\psi(U))\|\right].
    \end{equation*}
\end{corollary}
\begin{proof}
    Fix \(u\) and write \(w = \psi(u)\). For \(\RefM^w\)-almost every \(u'\), we have \(\psi(u') = w = \psi(u^*(w))\), so the fiber-Lipschitz hypothesis and the triangle inequality through \(u^*(w)\) give \(\Phi(u') - \Phi(u) \leq C\left(\|u' - u^*(w)\| + \|u - u^*(w)\|\right)\), and hence
    \begin{equation*}
        \EE_\RefM\left[\Phi \mid \psi(U) = w\right] - \Phi(u)
        \leq C\left(\int_{\U} \|u' - u^*(w)\|\RefM^w(\dd u') + \|u - u^*(w)\|\right).
    \end{equation*}
    By \cref{prop:jeffreys}, the bound \(\dd\TargetM/\dd\RefM \leq 1/Z\), and the tower property,
    \begin{align*}
        \kl(\TargetM,\pma)
        &\leq \frac{C}{Z}\EE_\RefM\left[\EE_\RefM\left[\|U - u^*(\psi(U))\| \mid \psi(U)\right] + \|U - u^*(\psi(U))\|\right] \\
        &= \frac{2C}{Z}\EE_\RefM\left[\|U - u^*(\psi(U))\|\right].
    \end{align*}
\end{proof}

\begin{corollary}[Second-moment bound]\label{cor:kl-jeffreys}
    Under the same assumptions as \cref{cor:kl-lipschitz}, and with
    \(\sigma_w^2\) as in \cref{prop:fiber-w1},
    \begin{equation*}
        \kl(\TargetM, \pma)
        \leq \EE_{W\sim\TargetM_\psi}\left[
        \kl(\TargetM^W, \RefM^W) + \kl(\RefM^W, \TargetM^W)\right]
        \leq \frac{C^2}{Z}
        \EE_{W\sim\RefM_\psi}\left[\sigma_W^2\right].
    \end{equation*}
\end{corollary}
\begin{proof}
    The first inequality is \cref{prop:jeffreys}. By
    \cref{eq:jeffreys-fiber}, Kantorovich--Rubinstein duality
    (\(\Phi/C\) is \(1\)-Lipschitz on the fiber \(\{\psi(u) = w\}\)),
    and \cref{prop:fiber-w1}, with
    \(Z_w = \EE_{\RefM^w}[e^{-\Phi}]\) as defined there,
    \begin{equation*}
        \kl(\TargetM^w, \RefM^w) + \kl(\RefM^w, \TargetM^w)
        \leq C W_1(\RefM^w, \TargetM^w)
        \leq \frac{C^2 \sigma_w^2}{Z_w}.
    \end{equation*}
    Averaging over \(W \sim \TargetM_\psi\) and using
    \(\dd\TargetM_\psi/\dd\RefM_\psi = Z_w/Z\) completes the proof.
\end{proof}

The hypotheses are identical to those of \cref{cor:kl-lipschitz}, but
this bound is quadratic in \(C\) and uses second moments in place
of first. The two are complementary, with the quadratic bound being sharper
when \(C\sigma_W \lesssim 1\). When the norm is induced by an inner
product, \(\sigma_w^2\) is the conditional variance and the conditional
mean minimizes the squared deviation over sections, so
\(\EE_{\RefM_\psi}[\sigma_W^2] \leq
\EE_\RefM\|U - u^*(\psi(U))\|^2\) for every section \(u^*\), and the
moment computations of \cref{ex:jump-rate} below apply verbatim.

A different route to such bounds is through general stability
estimates for perturbed posterior distributions, as found in
\cite{sprungk2020local}. For instance, applying
\cite[Theorem~11]{sprungk2020local} and noting that replacing
\(\Phi\) by \(\tilde\phi \circ \psi\) changes neither the reference
measure nor the normalizing constant gives the bound
\(\kl(\TargetM, \pma) \leq \frac{2}{Z}\|\Phi - \tilde\phi\circ\psi\|_{L^1_\RefM}\).
Using this bound requires two-sided control of
\(\Phi - \tilde\phi\circ\psi\), where the direction opposite to
Jensen's inequality asks for exponential integrability of the fiber
oscillations of \(\Phi\). Another strategy would be to apply the same
theorem to each fiber in the decomposition \cref{eqn:kl-irreducible},
where the approximation replaces \(e^{-\Phi}/Z_w\) with a constant
likelihood as in \cref{subsec:wasserstein}. In this case, the
perturbation has a single sign, so exponential integrability is not
needed, but the bound still carries an extra factor of order \(1/Z_w\) compared to
\cref{eq:jeffreys-fiber}. In either form the bound
is linear in the perturbation, while the divergence is quadratic when
the perturbation is small, so this route recovers at best the
first-moment rate of \cref{cor:kl-lipschitz} and not the improvement
of \cref{cor:kl-jeffreys}.

\begin{example}[Rates for the Brownian Poisson jump prior]\label{ex:jump-rate}
    Consider the setting of \cref{subsec:jump-deconv}: \(\RefM\) is the Brownian Poisson jump process of \cref{subsubsec:bpjp} on \([0,1]\), \(\psi(u) = (u(z_1),\ldots,u(z_d))\) is point evaluation on a grid of spacing \(h\), and
    \(\Phi(u) = \frac{1}{2\sigma^2}\sum_{i=1}^m (y_i - F_i(u))^2\) with \(F_i(u) = (\varphi * \tanh(\beta u))(x_i)\), where \(\beta = 3\) in \cref{subsec:jump-deconv}.
    Since \(|\tanh| \leq 1\) and \(\tanh(\beta\cdot)\) is \(\beta\)-Lipschitz, we have \(|F_i(u)| \leq \|\varphi\|_{L^1}\) and \(|F_i(u) - F_i(u')| \leq \beta\|\varphi\|_{L^2}\|u-u'\|_{L^2}\), so \(\Phi\) is globally Lipschitz on both \(L^2\) and \(\mathcal{H}_{\psi}\) (see \cref{eqn:special-space}) with the same constant
    \(C = \frac{\beta\|\varphi\|_{L^2}}{\sigma^2}\sum_{i=1}^m(|y_i| + \|\varphi\|_{L^1})\).
    For the section, take \(u^*(w) = \ell(w)\), where \(\ell(w)\) is
    the piecewise-linear interpolant of the grid values \(w\), which
    is also the conditional expectation of the process given
    \(\psi(U) = w\). Since \(\psi(\ell(w)) = w\), \(\|u - \ell(w)\|_{\mathcal{H}_\psi} = \|u - \ell(w)\|_{L^2}\) whenever \(\psi(u) = w\). 
    
    Set
    \(c := \sigma_{\mathrm{bm}}^2 + \lambda\sigma_{\mathrm{jump}}^2\),
    the variance rate of the underlying L\'evy process, whose increments
    over an interval of length \(t\) are independent, mean-zero, and of
    variance \(ct\). On a grid interval of length
    \(h\), writing \(Y_s\) for the increment from the left node, the
    deviation from the chord \(\ell\) satisfies
    \begin{equation*}
        \EE\left|Y_s - \tfrac{s}{h}Y_h\right|^2 = c\frac{s(h-s)}{h}, \qquad s \in [0,h],
    \end{equation*}
    exactly as for a Brownian bridge, as \(\EE[Y_s Y_h] = \Var(Y_s) = cs\). Integrating over each interval and summing the \(1/h\) intervals gives
    \(\EE\|U - \ell(\psi(U))\|_{L^2}^2 = ch/6\), and \cref{cor:kl-lipschitz} with Jensen's inequality gives
    \begin{equation*}
        \kl(\TargetM, \pma) \leq \frac{2C}{Z}\sqrt{\frac{c}{6}} h^{1/2}.
    \end{equation*}
    By contrast, \cref{cor:kl-jeffreys} uses the same moment
    computation at second order, giving
    \begin{equation*}
        \kl(\TargetM, \pma) \leq \frac{C^2}{Z}
        \frac{c}{6} h,
    \end{equation*}
    of order \(h\), while controlling the reverse divergence
    simultaneously. Similarly, consider taking \(\RefM\) to be a Mat\'ern-\(\nu\) Gaussian process defined on a compact set \(\Omega \subset \mathbb{R}^n\) with \(\nu > n/2\), and \(\psi\) to be point evaluations on a mesh of width \(h\). Then writing \(u^*\) for the conditional mean, \(\EE\|U - u^*(\psi(U))\|_{L^2}^2 \leq C' h^{2\nu}\) \cite[eq.~65]{kanagawa2018gaussianprocesseskernelmethods}, and \cref{cor:kl-lipschitz} yields a rate of order \(h^{\nu}\), which \cref{cor:kl-jeffreys} improves to \(h^{2\nu}\). These rates describe the decay of the approximation error within \(\psiclass\).
\end{example}

For the remainder of this section, we assume that \(\psi:\V\to\R^d\) is a bounded linear map, surjective without loss of generality, that \(\Phi\) is Fr\'echet differentiable on \(\V\), and we work entirely in a Hilbert space \(\V \subset \U\) which contains the support of \(\TargetM\), taking gradients, inner products, and orthogonal projections with respect to the \(\V\) inner product.
Define \(\W^\perp := \ker\psi\), which is a closed subspace of \(\V\), and \(\W = (\ker\psi)^\perp\).
For the space \(\H_{\psi}\) of \cref{ex:jump-rate}, gradient projected onto \(\W^{\perp}\) may be identified with the usual \(L^2\) gradient, as the \(\mathcal{H}_{\psi}\) inner product restricted to \(\ker \psi\) is the \(L^2\) inner product. 

Under a growth condition on the projected gradients of \(\Phi\), the bound of \cref{prop:jeffreys} yields an explicit rate in terms of the conditional moments of the reference measure.

\begin{proposition}\label{prop:PMA-grad-lipschitz}
    Let \(u^*:\R^d\to\V\) be a measurable section of \(\psi\), suppose \(\Phi\) is Fr\'echet differentiable on \(\V\), and suppose that for all \(u \in \V\),
    \begin{equation*}
    \|P_{\W^\perp}\nabla\Phi(u)\| \leq \|P_{\W^\perp}\nabla\Phi(u^*(\psi(u)))\| + L(\psi(u))\|u - u^*(\psi(u))\|.
    \end{equation*}
    Then
    \begin{equation*}
    \begin{split}
        \kl(\TargetM,\pma)
        \leq \frac{2}{Z}\EE_\RefM\Big[&\|P_{\W^\perp}\nabla\Phi(u^*(\psi(U)))\|\|U - u^*(\psi(U))\| \\
        &+ \frac{L(\psi(U))}{2}\|U - u^*(\psi(U))\|^2\Big].
    \end{split}
    \end{equation*}
\end{proposition}
\begin{proof}
    Fix \(u \in \V\) and write \(w = \psi(u)\). Since \(u - u^*(w) \in \ker\psi\), the segment \(u_t := u^*(w) + t(u - u^*(w))\), \(t\in[0,1]\), satisfies \(\psi(u_t) = w\), with \(u_t - u^*(\psi(u_t)) = t(u - u^*(w))\).
    Using that \(u - u^*(w) \in \ker\psi\) and that \(P_{\W^\perp}\) is self-adjoint,
    \begin{align*}
    |\Phi(u) - \Phi(u^*(w))| &\leq
    \int_{0}^{1} \left|\left\langle
        u-u^*(w), \nabla\Phi(u_t)
    \right\rangle\right|
    \dd t\\
    &= \int_{0}^{1} \left|\left\langle
        P_{\W^\perp}\left(u-u^*(w)\right), \nabla\Phi(u_t)
    \right\rangle\right| \dd t\\
    &= \int_{0}^{1} \left|\left\langle
        u-u^*(w), P_{\W^\perp}\nabla\Phi(u_t)
    \right\rangle\right| \dd t\\
    &\leq \|u-u^*(w)\|\int_{0}^{1}\left\|P_{\W^\perp}\nabla\Phi(u_t)\right\|\dd t
    \\
    & \leq \|P_{\W^\perp}\nabla\Phi(u^*(w))\|\|u-u^*(w)\| + \frac{L(w)}{2}\|u-u^*(w)\|^2,
    \end{align*}
    where the hypothesis was applied at \(u_t\), using \(\psi(u_t) = w\) and \(u_t - u^*(w) = t(u-u^*(w))\).
    Writing \(D(u) := \Phi(u) - \Phi(u^*(\psi(u)))\), we have \(|D(u)| \leq \|P_{\W^\perp}\nabla\Phi(u^*(w))\|\|u-u^*(w)\| + \frac{L(w)}{2}\|u-u^*(w)\|^2\). Since \(u^*(\psi(u')) = u^*(w)\) for \(\RefM^{w}\)-almost every \(u'\),
    \begin{align*}
        \EE_\RefM\left[\Phi \mid \psi(U) = w\right] - \Phi(u)
        &= \int_{\V} D(u')\RefM^{w}(\dd u') - D(u) \\
        &\leq \int_{\V} \Big(\|P_{\W^\perp}\nabla\Phi(u^*(w))\|\|u'-u^*(w)\| \\
        &\hspace{5em} + \frac{L(w)}{2}\|u'-u^*(w)\|^2\Big)\RefM^{w}(\dd u') \\
        &\quad + \|P_{\W^\perp}\nabla\Phi(u^*(w))\|\|u-u^*(w)\| + \frac{L(w)}{2}\|u-u^*(w)\|^2.
    \end{align*}
    By \cref{prop:jeffreys}, the bound \(\dd\TargetM/\dd\RefM \leq 1/Z\), and the tower property applied to the conditional term,
    \begin{align*}
        \kl(\TargetM,\pma)
        &\leq \EE_\TargetM\left[\EE_\RefM[\Phi \mid \psi(U)] - \Phi(U)\right] \\
        &\leq \frac{2}{Z}\EE_\RefM\Big[\|P_{\W^\perp}\nabla\Phi(u^*(\psi(U)))\|\|U - u^*(\psi(U))\| \\
        &\qquad\qquad\quad + \frac{L(\psi(U))}{2}\|U - u^*(\psi(U))\|^2\Big].
    \end{align*}
\end{proof}

Since \(\psi(u^*(\psi(U))) = \psi(U)\), the difference \(U - u^*(\psi(U))\) lies in \(\W^\perp\), so that \(\|U - u^*(\psi(U))\| = \|P_{\W^\perp}(U - u^*(\psi(U)))\|\); the distances appearing in the bound are measured within the subspace.
When \(u^*(w)\) is chosen to be a critical point of the restriction of \(\Phi\) to the fiber, such as a local minimum, the first term vanishes and the bound reduces to \(\frac{1}{Z}\EE_\RefM[L(\psi(U))\|U-u^*(\psi(U))\|^2]\). On the other hand, choosing \(u^*(w)\) to be the conditional mean of the prior minimizes the second term, which then reduces to an averaged conditional variance. The hypothesis follows from the triangle inequality whenever the projected gradient is Lipschitz along fibers, i.e., \(\|P_{\W^\perp}(\nabla\Phi(u)-\nabla\Phi(u'))\| \leq L(w)\|u-u'\|\) whenever \(\psi(u) = \psi(u') = w\); for twice differentiable \(\Phi\), this is a bound on the projected Hessian. In contrast to the bounds based on functional inequalities recalled in \cref{subsec:functional}, \cref{prop:PMA-grad-lipschitz} requires only conditional moments of the reference measure, not Poincar\'e or log-Sobolev constants.

\subsection{Error bounds assuming functional inequalities}\label{subsec:functional}
The bounds of \cref{subsec:PMA-error} place regularity
conditions on the potential \(\Phi\) that are uniform in \(u\), while requiring only low-order
conditional moments of the reference measure. In this subsection we
cover the guarantees available when the reference conditionals satisfy
functional inequalities, in which case the error is controlled by
averaged gradient norms projected to the null space of \(\psi\). The
bounds assuming a log-Sobolev inequality follow the results in
\cite{certified}, and those assuming a Poincar\'e inequality largely
follow \cite{cui-LIS}. Both of these works state their bounds on
\(\mathbb{R}^d\), but there is no complication in extending them to the
Hilbert space setting. The identity \cref{eqn:kl-irreducible} takes the
place of the marginal--conditional decompositions used there, and the
arguments are otherwise unchanged. The functional inequalities
themselves are native to infinite dimensions. The log-Sobolev
inequality was introduced in the infinite-dimensional Gaussian setting
\cite{gross1975logarithmic}, Hilbert-space statements of both
inequalities in the form used below may be found in
\cite[Section~10.5]{daprato2002second} and \cite{bogachev2015gaussian},
and the Sobolev classes
\(H^1(\RefM^w)\) over general reference measures are those of
\cite[Chapter~8]{bogachev2010differentiable}.

\begin{definition}[Conditional functional inequalities]\label{def:cond-fi}
    The reference \(\RefM\) satisfies a \emph{conditional log-Sobolev
    inequality} with constant \(\kappa\) if, for \(\RefM_\psi\)-almost
    every \(w\) and all \(h \in H^1(\RefM^w)\),
    \begin{equation*}
        \mathrm{Ent}_{\RefM^w}\left(h^2\right) \leq 2\kappa
        \EE_{\RefM^w}\left[\|P_{\W^\perp}\nabla h\|^2\right],
        \qquad
        \mathrm{Ent}_\rho(g) := \EE_\rho[g \log g] -
        \EE_\rho[g]\log\EE_\rho[g],
    \end{equation*}
    and a \emph{conditional Poincar\'e inequality} with constant
    \(\kappa\) if
    \begin{equation*}
        \Var_{\RefM^w}[h] \leq \kappa
        \EE_{\RefM^w}\left[\|P_{\W^\perp}\nabla h\|^2\right].
    \end{equation*}
\end{definition}

A conditional log-Sobolev inequality implies a conditional Poincar\'e
inequality with the same constant, but the converse is not necessarily
true.

\begin{proposition}[{Kullback--Leibler bound, following \cite{certified}}]\label{prop:kl-lsi}
    Suppose \(\RefM\) satisfies a conditional log-Sobolev inequality
    with constant \(\kappa\). Then
    \begin{equation*}
        \kl(\TargetM, \pma) \leq \frac{\kappa}{2} \EE_\TargetM\left[\|P_{\W^\perp}\nabla\Phi(U)\|^2\right].
    \end{equation*}
\end{proposition}
\begin{proof}
    On the fiber over \(w\), \(\dd\TargetM^w/\dd\RefM^w = e^{\tilde\phi(w) - \Phi}\), whose logarithm has fiber gradient \(-P_{\W^\perp}\nabla\Phi\) (the term \(\tilde\phi(w)\) is constant on the fiber). Applying the conditional log-Sobolev inequality with \(h^2 = \dd\TargetM^w/\dd\RefM^w\) yields the standard implication
    \begin{equation*}
        \kl(\TargetM^w, \RefM^w) \leq \frac{\kappa}{2}\EE_{\TargetM^w}\left[\|P_{\W^\perp}\nabla\Phi\|^2\right],
    \end{equation*}
    and integrating over \(W \sim \TargetM_\psi\) with the identity \cref{eqn:kl-irreducible} gives the result.
\end{proof}

A log-Sobolev inequality is a strong requirement. Under the weaker
conditional Poincar\'e inequality, the same quantity controls the
squared Hellinger distance
\(d_H^2(\rho, \sigma) := 1 - \int \sqrt{\dd\rho \dd\sigma}\).

\begin{proposition}[{Hellinger bound, following \cite{cui-LIS}}]\label{prop:hellinger}
    Suppose \(\RefM\) satisfies a conditional Poincar\'e inequality
    with constant \(\kappa\). Then
    \begin{equation*}
        d_H^2(\TargetM, \pma)
        \leq \frac{\kappa}{4} \EE_\TargetM\left[\|P_{\W^\perp}\nabla\Phi(U)\|^2\right].
    \end{equation*}
\end{proposition}
\begin{proof}
    Since \(\TargetM\) and \(\pma\) share the \(\psi\)-marginal and the
    conditionals of \(\pma\) are \(\RefM^w\), the Hellinger affinity
    disintegrates, and in parallel to \cref{eqn:kl-irreducible},
    \(d_H^2(\TargetM, \pma) =
    \EE_{W\sim\TargetM_\psi}[d_H^2(\TargetM^W, \RefM^W)]\).
    On the fiber over \(w\), let
    \(f := (\dd\TargetM^w/\dd\RefM^w)^{1/2} = e^{(\tilde\phi(w) - \Phi)/2}\),
    so that \(\EE_{\RefM^w}[f^2] = 1\) and
    \(P_{\W^\perp}\nabla f = -\tfrac12 f P_{\W^\perp}\nabla\Phi\).
    Since \(\EE_{\RefM^w}[f] \leq 1\), the conditional Poincar\'e
    inequality gives
    \begin{equation*}
        d_H^2(\TargetM^w, \RefM^w)
        = 1 - \EE_{\RefM^w}[f]
        \leq 1 - \left(\EE_{\RefM^w}[f]\right)^2
        = \Var_{\RefM^w}(f)
        \leq \frac{\kappa}{4} \EE_{\RefM^w}\left[f^2 \|P_{\W^\perp}\nabla\Phi\|^2\right],
    \end{equation*}
    and \(f^2 \dd\RefM^w = \dd\TargetM^w\) turns the right-hand side
    into \(\frac{\kappa}{4} \EE_{\TargetM^w}[\|P_{\W^\perp}\nabla\Phi\|^2]\);
    integrate over \(W \sim \TargetM_\psi\).
\end{proof}

\begin{remark}[Gaussian and Mat\'ern references]\label{rem:lsi-gauss}
For Gaussian \(\RefM\) and bounded linear \(\psi\), the conditional
log-Sobolev inequality holds with \(\kappa = \|\Sigma_{u \mid w}\|\),
the operator norm of the
conditional covariance
\(\Sigma_{u \mid w} = \Sigma_{uu} - \Sigma_{u\psi}\Sigma_{\psi\psi}^{-1}\Sigma_{\psi u}\)
of \cref{eq:gauss-cond}. Indeed, the
conditionals \(\RefM^w = \RefM(\cdot \mid \psi(u) = w)\) are Gaussian
with this covariance, and for a Gaussian measure an upper bound on the
covariance operator is a lower bound on the strong convexity of the
associated potential, so the inequality follows from Bakry--\'Emery
theory \cite{bakry2014analysis} (for a direct statement via the Cameron--Martin derivative, see
\cite{bogachev2015gaussian}); the fiber gradient \(P_{\W^\perp}\nabla\)
is the natural Dirichlet form here because \(\Sigma_{u \mid w}\) ranges
in \(\ker\psi\), i.e., \(\psi\Sigma_{u \mid w} = 0\). Since
\(\Sigma_{u \mid w} \preceq \Sigma_{uu}\), one always has
\(\kappa \leq \|\Sigma_{uu}\|\), but the conditional constant improves
as \(\psi\) grows. For instance, let \(\U = L^2(\Omega)\) with
\(\Omega \subset \R^n\) bounded, \(\RefM\) be a Mat\'ern-\(\nu\)
Gaussian process with \(\nu > n/2\), and
\(\psi(u) = (u(x_i))_{i=1}^d\) be point evaluations on discrete sets with
mesh norm \(h_d\). The conditional variance function is the kriging
variance, bounded by \(C h_d^{2\nu}\) uniformly on
\(\Omega\) \cite[eq.~65]{kanagawa2018gaussianprocesseskernelmethods}, so
\(\kappa_d \leq \tr \Sigma_{u \mid w} \leq C|\Omega|h_d^{2\nu} \to 0\).
\end{remark}

\begin{remark}[Weighted forms]\label{rem:weighted-fi}
    Both proofs use the conditional inequalities only through the
    Dirichlet form, and apply verbatim when the gradient is weighted
    by a positive operator: if, for \(\RefM_\psi\)-almost every \(w\)
    and a positive operator \(\Gamma\) (possibly depending on \(w\)),
    \begin{equation*}
        \mathrm{Ent}_{\RefM^w}\left(h^2\right) \leq 2
        \EE_{\RefM^w}\left[\langle \nabla h,\ \Gamma \nabla h\rangle\right],
    \end{equation*}
    then
    \begin{equation*}
        \kl(\TargetM, \pma) \leq \tfrac12
        \EE_\TargetM\left[\langle\nabla\Phi, \Gamma\nabla\Phi\rangle\right],
    \end{equation*}
    and the Poincar\'e version likewise gives
    \begin{equation*}
        d_H^2(\TargetM, \pma) \leq \tfrac14
        \EE_\TargetM\left[\langle\nabla\Phi, \Gamma\nabla\Phi\rangle\right].
    \end{equation*}
    The scalar inequalities of \cref{def:cond-fi} are the case
    \(\Gamma = \kappa P_{\W^\perp}\). For Gaussian \(\RefM\) the
    conditionals satisfy the weighted inequality with constant one and
    \(\Gamma = \Sigma_{u \mid w}\), the conditional covariance, by the
    same Bakry--\'Emery argument as in \cref{rem:lsi-gauss}. Since
    \(\Sigma_{u \mid w}\) does not depend on \(w\),
    \begin{equation*}
        \kl(\TargetM, \pma) \leq \tfrac12
        \EE_\TargetM\left[\langle\nabla\Phi,\ \Sigma_{u\mid w}
        \nabla\Phi\rangle\right].
    \end{equation*}
\end{remark}

Both bounds depend on \(\psi\) only through a quadratic form in
\(\nabla\Phi\), which is therefore the natural objective for choosing
\(\psi\). Suppose that \(\RefM\) is Gaussian with covariance \(\Gamma\)
or, more generally, has a log-concave density with respect to such a
Gaussian, so that Bakry--\'Emery theory gives a log-Sobolev inequality
with weight operator \(\Gamma\).
The Bakry--\'Emery condition survives conditioning, since restricting
that density to a fiber of a linear \(\psi\) preserves its
log-concavity and the conditionals then satisfy the
weighted inequalities of \cref{rem:weighted-fi} with the weight
\(\Gamma_\psi := \Gamma - \Gamma\psi^*(\psi\Gamma\psi^*)^{-1}\psi\Gamma\),
the Schur complement of \(\Gamma\) onto \(\ker\psi\). For Gaussian
\(\RefM\) this is the conditional covariance \(\Sigma_{u \mid w}\) of
\cref{eq:gauss-cond}. The resulting bound is a residual trace. Writing
\(\Gamma_\psi = \Gamma^{1/2}(I - \tilde{P})\Gamma^{1/2}\), where
\(\tilde{P}\) is the orthogonal projection onto
\(\Range(\Gamma^{1/2}\psi^*)\),
\begin{equation}\label{eqn:diagnostic}
    \EE_\TargetM\left[\langle\nabla\Phi,\ \Gamma_\psi
    \nabla\Phi\rangle\right]
    = \tr\left(\Gamma_\psi M_\TargetM\right)
    = \tr\left( (I - \tilde{P}) D (I - \tilde{P}) \right),
\end{equation}
where \(M_\TargetM := \EE_\TargetM[\nabla\Phi \nabla\Phi^*]\), assumed
bounded, and \(D := \Gamma^{1/2} M_\TargetM \Gamma^{1/2}\) is its
prior-whitened form, which is then trace class with
\(\tr D = \EE_\TargetM\|\Gamma^{1/2}\nabla\Phi\|^2\). As
\(\psi\) ranges over rank-\(d\) bounded linear maps, \(\tilde{P}\)
ranges over the rank-\(d\) orthogonal projections with range in
\(\Range(\Gamma^{1/2})\), so the Ky Fan principle
gives
\begin{equation*}
    \inf_{\operatorname{rank}\psi = d}
    \tr\left(\Gamma_\psi M_\TargetM\right)
    = \sum_{j > d} \lambda_j(D),
\end{equation*}
where \(\lambda_1 \geq \lambda_2 \geq \cdots\) are the eigenvalues of
\(D\). Eigenvectors \(D v_j = \lambda_j v_j\) with \(\lambda_j > 0\)
lie automatically in \(\Range(\Gamma^{1/2})\), so the infimum is
attained by taking
\(\psi_j(u) = \langle \Gamma^{-1/2} v_j, u\rangle\), or equivalently, by
the leading solutions of the generalized eigenproblem
\(M_\TargetM \varphi = \lambda \Gamma^{-1}\varphi\). This
prior-weighted eigenproblem is precisely the subspace-selection
criterion of \cite{certified}, and the original likelihood-informed
subspace construction \cite{lis} solves the same eigenproblem with the
Gauss--Newton Hessian of the data misfit in place of \(M_\TargetM\).
When \(\Phi\) is \(C\)-Lipschitz, \(M_\TargetM \preceq C^2
I\) and \cref{eqn:diagnostic} is at most \(C^2 \tr\Gamma_\psi\). In the
Mat\'ern example of \cref{rem:lsi-gauss}, this is \(C^2\) times
the integrated kriging variance, which is of order \(h_d^{2\nu}\), matching the
rate of \cref{cor:kl-jeffreys}.

Finally, a Talagrand transport inequality for the conditionals
converts the Kullback--Leibler bound into a Wasserstein bound,
connecting the results of this subsection back to those of
\cref{subsec:wasserstein}.

\begin{proposition}[\(\kl\) to \(W_2\)]\label{prop:t2}
    Suppose that for \(\RefM_\psi\)-almost every \(w\), \(\RefM^w\) satisfies the Talagrand inequality \(W_2^2(\rho, \RefM^w) \leq 2\kappa_T \kl(\rho, \RefM^w)\) for all \(\rho\). Then
    \begin{equation*}
        W_2^2(\pma, \TargetM) \leq 2\kappa_T \kl(\TargetM, \pma).
    \end{equation*}
\end{proposition}
\begin{proof}
    By \cref{cor:marginal-Wp} with \(p = 2\), the Talagrand inequality applied on each fiber with \(\rho = \TargetM^w\), and the identity \cref{eqn:kl-irreducible},
    \begin{equation*}
        W_2^2(\pma, \TargetM)
        \leq \EE_{W\sim\TargetM_\psi}\left[W_2^2\left(\RefM^W, \TargetM^W\right)\right]
        \leq 2\kappa_T \EE_{W\sim\TargetM_\psi}\left[\kl\left(\TargetM^W, \RefM^W\right)\right]
        = 2\kappa_T \kl(\TargetM, \pma).
    \end{equation*}
\end{proof}

For Gaussian conditionals, the Poincar\'e, log-Sobolev, and Talagrand inequalities all hold with the same constant \(\kappa = \kappa_T = \|\Sigma_{u\mid w}\|\), which does not depend on \(w\). In particular, all of the Kullback--Leibler bounds in \cref{subsec:PMA-error} additionally yield Wasserstein bounds with this extra factor of \(2 \kappa_T\) when the assumptions of \cref{prop:t2} hold.

\subsection{Discussion}\label{subsec:analysis-discussion}

The results of this section reduce the function-space error of
\cref{alg:reduced-gen,alg:reduced-cond-gen,alg:vi}
to a combination of error from approximation in \(\psiclass\),
controlled by the moment and functional-inequality bounds of
\cref{subsec:PMA-error,subsec:functional}, and marginal
fitting terms such as \(W_p(\hat{\eta}_\psi, \TargetM_\psi)\), the
error of a generative model on \(\R^d\).
Error bounds for such terms
are the subject of the finite-dimensional generative modeling
literature (for the flow-matching and interpolant models used in
\cref{sec:numerics}, see
\cite{albergo2025stochastic,benton2023error}), and
transfer to function space through
\cref{prop:lift-stability,cor:learned-cond} at the price of the
constant \(K\), or through the chain rule
\cref{eqn:kl-chain-psi} for
Kullback--Leibler control.

For fixed reference distributions,
the marginal terms also set the statistical difficulty of the method.
The dimension that governs this error is not
the dimension of the discretization of \(u\), but the size of \(\psi\),
which can be chosen to match the information given by the
likelihood.
The infinite-dimensional remainder is carried by the
conditioned-sampling step, which for tractable reference measures is
exact and incurs no statistical cost. When the likelihood is
informative only on a low-dimensional subspace, the method thereby avoids errors
that grow with the ambient dimension.
This is consistent with the minimax analysis of
\cite{ponnoprat2025minimax} for the estimation of transport maps
between infinite-dimensional spaces, where error decays only poly-logarithmically in the sample size
without stronger assumptions of low-dimensional structure. 

This structure is particularly relevant when the data itself is
observed at a fixed resolution. Samples seen only through \(\psi\),
such as pixel averages or pointwise measurements, determine the target
only through its \(\psi\)-marginal. Any model of the underlying
function must supply the infinite-dimensional structure by assumption.
Measures in \(\psiclass\) fit the marginal distribution at the observed resolution and complete it to a measure on functions
using the reference conditionals. The resulting model can be sampled at any
resolution by conditioned sampling, and when the underlying law
admits a density with respect to the reference, the results of this
section quantify the error of this completion.

\appendix
\crefname{appendix}{Appendix}{Appendices}
\Crefname{appendix}{Appendix}{Appendices}
\crefalias{section}{appendix}
\section{Proofs for the Gaussian optimal transport results}\label{app:ot-proofs}

The trace norm bound in \cref{thm:ot-structure} relies on the
following inversion bound for positive perturbations of a self-adjoint
involution, a special case of
\cite[Theorem~3]{kostrykin2007perturbation}. We include its short
proof.

\begin{lemma}\label{lem:signed-pencil}
    Let \(J\) be a self-adjoint unitary operator on a Hilbert space with
    spectral projections \(P_\pm\), so that \(J = P_+ - P_-\), and let
    \(K \succeq 0\) be a bounded operator with
    \(P_-KP_- \preceq (1-c) I\) for some \(c \in (0, 1]\). Then
    \(J + K\) is invertible, with
    \begin{equation*}
        \|(J + K)^{-1}\| \leq \frac{1}{c}.
    \end{equation*}
\end{lemma}
\begin{proof}
    Let \(x = x_+ + x_-\) with \(x_{\pm} = P_{\pm}x\). We lower bound \(\|(J+K)x\|\). First, apply Cauchy--Schwarz to see that
    \begin{equation*}
        \left|\left\langle Jx,(J+K)x\right\rangle \right|\leq\|Jx\|\|\left(J+K\right)x\|=\|x\|\|\left(J+K\right)x\|.
    \end{equation*}
    We now lower bound the quadratic form.
    \begin{align*}
        \left|\left\langle Jx,(J+K)x\right\rangle \right| &=\left|\left\langle x_{+}-x_{-},x_{+}-x_{-}+Kx_{+}+Kx_{-}\right\rangle \right| \\
        &= \left|\|x\|^{2}+\left\langle Kx_{+},x_{+}\right\rangle -\left\langle Kx_{-},x_{-}\right\rangle +\left\langle x_{+},Kx_{-}\right\rangle -\left\langle x_{-},Kx_{+}\right\rangle \right| \\
        &\geq \operatorname{Re}\left(\|x\|^{2}+\left\langle Kx_{+},x_{+}\right\rangle -\left\langle Kx_{-},x_{-}\right\rangle +\left\langle x_{+},Kx_{-}\right\rangle -\left\langle x_{-},Kx_{+}\right\rangle \right) \\
        &=
        \operatorname{Re}\left(\|x\|^{2}+\left\langle Kx_{+},x_{+}\right\rangle -\left\langle Kx_{-},x_{-}\right\rangle +\left\langle x_{+},Kx_{-}\right\rangle -\overline{ \left\langle x_{+},Kx_{-}\right\rangle }\right) \\
        &= \|x\|^{2}+\left\langle Kx_{+},x_{+}\right\rangle -\left\langle Kx_{-},x_{-}\right\rangle \\
        &\geq\|x\|^{2}-\left\langle Kx_{-},x_{-}\right\rangle \\
        &=\left\langle x,\left(I-P_{-}KP_{-}\right)x\right\rangle \\
        &\geq c\|x\|^{2},
    \end{align*}
    where we used that \(\left\langle Kx_{+},x_{+}\right\rangle\geq 0\) because \(K\succeq0\), and that \(\operatorname{Re}(z - \overline{z})=0\) for any \(z \in \mathbb{C}\). Thus \(\|(J+K)x\| \geq c\|x\|\). Because \(J+K\) is self-adjoint, we may conclude that it is boundedly invertible with \(\|(J+K)^{-1}\|\leq \frac{1}{c}\).
\end{proof}

\begin{proof}[Proof of \cref{thm:ot-structure}]

    We first prove equivalence of the two measures.
    Since \(\Range(\Sigma^{1/2})\) is dense and \(\Sigma'\) is a covariance operator, we have that \(I + H \succeq 0\).
    Since \(I+H\) is injective and \(H\) is compact, \(I+H\) is boundedly invertible. With \(c = \min\left(\|(I+H)^{-1}\|^{-1},1\right) \in (0,1]\), we have that \(I + H \succeq c\).\footnote{When \(\V\) is infinite dimensional, \(\|(I+H)^{-1}\|\geq 1\) automatically because \(0\) is an accumulation point of the spectrum of \(H\), so that \(\kappa = \|(I+H)^{-1}\|\).}
    Thus,
    \begin{equation}\label{eqn:H-lower-upper}
        c I \preceq I+H \preceq (1+\|H\|) I.
    \end{equation}
    The lower bound implies that \(\Sigma'\) is injective, so \(\eta\) is nondegenerate.
    Conjugating each piece of \cref{eqn:H-lower-upper} and applying Douglas' lemma \cite[Theorem~1]{douglas1966majorization} yields that
    \(\Range(\Sigma^{1/2}) = \Range(\Sigma'^{1/2})\). Finally, since trace class operators are Hilbert--Schmidt, the Feldman--H\'ajek theorem \cite{bogachev2015gaussian} gives that \(\RefM\) and \(\eta\) are equivalent.

    We now construct the OT map \(T(u) = Au\), first defining it on the dense subspace \(\Range(\Sigma^{1/2})\) before extending it to a continuous map on all of \(\V\).
    Set
    \begin{equation}\label{eqn:M-def}
    M := \Sigma^{1/2}\Sigma'\Sigma^{1/2}
    = \Sigma^2 + \Sigma H\Sigma = \Sigma(I + H)\Sigma.
    \end{equation}

    Applying \cref{eqn:H-lower-upper} immediately gives
    \begin{equation*}
        c \Sigma^2 \preceq M \preceq (1 + \|H\|)\,\Sigma^2.
    \end{equation*}

    Since the square root is operator monotone,
    \begin{equation}
    \sqrt{c}\,\Sigma \preceq M^{1/2}
    \preceq \sqrt{1 + \|H\|}\,\Sigma.
    \end{equation}

    Now define
    \(A := \Sigma^{-1/2}M^{1/2}\Sigma^{-1/2}\) on
    \(\Range(\Sigma^{1/2})\).
    In finite dimensions, this is the closed form of the optimal
    transport map between centered Gaussian measures
    \cite{knott1984optimal}, which extends to Gaussian measures on a
    separable Hilbert space as a possibly unbounded map
    \cite{cuesta1996lower}.
    For \(x = \Sigma^{1/2}y\) we have
    \(\langle x, Ax\rangle = \langle y, M^{1/2}y\rangle
    \in [\sqrt{c}\,\|x\|^2, \sqrt{1+\|H\|}\,\|x\|^2]\), so \(A\)
    extends to a bounded positive operator on \(\V\) with a
    bounded inverse.

    In particular, observe that \(\Sigma^{1/2}A\Sigma^{1/2} = M^{1/2}\). Squaring this expression, we obtain
    \begin{equation*}
    \Sigma^{1/2}(A\Sigma A)\Sigma^{1/2} = M
    = \Sigma^{1/2}\Sigma'\Sigma^{1/2},
    \end{equation*}
    which, along with the injectivity of \(\Sigma^{1/2}\) and the density of
    its range, establishes that \(A\Sigma A = \Sigma'\).

    For uniqueness, if \(\tilde A\) is another bounded
    positive solution, then
    \(\Sigma^{1/2}\tilde A\Sigma^{1/2}\) is a positive square root of
    \(M\), hence equal to \(M^{1/2}\), and \(\tilde A = A\) by density.

    For optimality, \(u \mapsto Au\) is the gradient of the convex function
    \(\tfrac12\langle u, Au\rangle\) and pushes \(\RefM\) to \(\eta\),
    so its coupling is cyclically monotone and it is the optimal
    transport map \cite[Section~6.2.3]{ambrosio2008gradient}.

    We now derive the integral representation \cref{eqn:ot-structure-int}.
    Start by applying the integral formula
    \begin{equation*}
        X^{1/2} = \frac{2}{\pi}\int_0^\infty
        \left[I - t^2(X + t^2)^{-1}\right]\dd t
    \end{equation*}
    to \(M\) and \(\Sigma^2\) and subtracting, which gives
    \begin{equation}\label{eqn:first-integral}
    M^{1/2} - \Sigma = \frac{2}{\pi}\int_0^\infty
    t^2\left[R_t - (M+t^2)^{-1}\right]\dd t.
    \end{equation}
    Expanding \(R_t\) and \(M\), and applying the second resolvent identity, we get
    \begin{align*}
        R_t - (M + t^2)^{-1} &= (\Sigma^2 + t^2)^{-1} - (\Sigma^2 + \Sigma H \Sigma + t^2)^{-1} \\
        &= (\Sigma^2 + t^2)^{-1} (\Sigma H \Sigma) (\Sigma^2 + t^2 + \Sigma H \Sigma)^{-1} \\
        &= R_t (\Sigma H \Sigma) \left[(\Sigma^2 + t^2) (I + R_t\Sigma H \Sigma)\right]^{-1} \\
        &= (R_t \Sigma) (H \Sigma) (I + (R_t\Sigma) (H \Sigma))^{-1} R_t.
    \end{align*}
    We now apply the push-through identity \((I + AB)^{-1}A = A(I + BA)^{-1}\) twice, once with \(H\Sigma\) and \(R_t\Sigma\), and again with \(H\) and \(G_t\),
    obtaining
    \begin{align}
        R_t - (M + t^2)^{-1} &= (R_t \Sigma) (I + H \underbrace{\Sigma R_t\Sigma}_{=G_t} )^{-1} (H \Sigma) R_t \nonumber\\
        &= (R_t \Sigma) (I + H G_t )^{-1} H \Sigma R_t \nonumber\\
        &= R_t \Sigma H(I + G_t H )^{-1}\Sigma R_t. \label{eqn:final-diff}
    \end{align}
    Finally, plug \cref{eqn:final-diff} into \cref{eqn:first-integral} and conjugate each integrand by \(\Sigma^{-1/2}\), which commutes with \(R_t\), to obtain the desired integrands of \cref{eqn:ot-structure-int}.
    Define
    \begin{equation}\label{eqn:S-def}
        S := \frac{2}{\pi} \int_{0}^{\infty} t^2
        R_t \Sigma^{1/2}\, H(I + G_t H )^{-1}\,\Sigma^{1/2} R_t \dd t
    \end{equation}
    In order to conclude that \(S = A - I\) and that it is trace class, it remains to establish
    trace-norm convergence of the integral.

    We first bound the middle factor \(H (I + G_t H)^{-1}\).
    Restrict to the closure of \(\Range(H)\),
    on which \(|H|^{1/2}\) is injective so we may write the factorization
    \(H = |H|^{1/2}J|H|^{1/2}\) with \(J\) the unitary sign of \(H\)
    given by the Borel functional calculus
    \cite[Theorem~VII.2]{reedsimon1980functional}, satisfying
    \(J^{-1} = J\) and \(J^2 = I\).
    Set \(K_t := |H|^{1/2}G_t|H|^{1/2} \succeq 0\). Expanding and applying the push-through identity, we have
    \begin{align*}
    H(I + G_t H)^{-1} &= |H|^{1/2}J|H|^{1/2} \left(I + G_t |H|^{1/2}J|H|^{1/2}\right)^{-1}
    \\
    &= |H|^{1/2} \left(I + J|H|^{1/2} G_t |H|^{1/2}\right)^{-1} J|H|^{1/2}\\
    &= |H|^{1/2} \left(I + J K_t\right)^{-1} J|H|^{1/2}\\
    &= |H|^{1/2} \left(J^2 + J K_t\right)^{-1} J|H|^{1/2}\\
    &= |H|^{1/2} \left(J + K_t\right)^{-1} |H|^{1/2}.
    \end{align*}
    Since \(0 \preceq G_t \preceq I\), we have
    \(0 \preceq K_t \preceq |H|\), and on the negative subspace of
    \(J\), the lower bound of \cref{eqn:H-lower-upper} reads
    \(|H| \preceq (1 - c)I\), so \(P_-K_tP_- \preceq (1-c)P_-\).
    \cref{lem:signed-pencil} applied to \(J\) and \(K_t\) then gives
    \(\|(J + K_t)^{-1}\| \leq \tfrac{1}{c}\), uniformly in \(t\). As \(1/c = \kappa\), we obtain
    \(\|(J + K_t)^{-1}\| \leq \kappa\) with \(\kappa\) as in
    \cref{eqn:norm-bound}.

    We now have that each integrand \(t^2
        R_t \Sigma^{1/2}\, H(I + G_t H )^{-1}\,\Sigma^{1/2} R_t\) of \cref{eqn:S-def} is
    a trace class operator depending
    continuously on \(t\) in trace norm, since
    \(\Sigma^{1/2}|H|^{1/2}\) is Hilbert--Schmidt and
    \(t \mapsto R_t\) is norm continuous.
    Applying the H\"older inequality
    \(\|XYZ\|_1 \leq \|X\|_{\mathrm{HS}}\|Y\|\|Z\|_{\mathrm{HS}}\),
    we obtain the bound
    \begin{align*}
    \|
    R_t \Sigma^{1/2} H(I + G_t H )^{-1} \Sigma^{1/2} R_t\|_1
    &=
    \|
    \left(R_t \Sigma^{1/2} |H|^{1/2}\right) (J+K_t)^{-1}
    \left(|H|^{1/2} \Sigma^{1/2} R_t\right)\|_1
    \\
    &\leq \kappa \|R_t\Sigma^{1/2}|H|^{1/2}\|_{\mathrm{HS}}^2.
    \end{align*}

    By the spectral identity
    \(\int_0^\infty t^2\lambda(\lambda^2 + t^2)^{-2}\dd t = \pi/4\)
    for every \(\lambda > 0\), the positive operators
    \(t^2\,|H|^{1/2}\Sigma^{1/2}R_t^2\Sigma^{1/2}|H|^{1/2}\) integrate to
    \(\tfrac{\pi}{4}|H|\) in trace norm. Their traces are nonnegative, so the
    trace and the integral may be exchanged, giving
    \begin{align*}
        \int_0^\infty t^2\,
        \|R_t\Sigma^{1/2}|H|^{1/2}\|_{\mathrm{HS}}^2\,\dd t
        &= \int_0^\infty t^2\,\operatorname{tr}\left(
        |H|^{1/2}\,\Sigma^{1/2}R_t^2\Sigma^{1/2}\,|H|^{1/2}
        \right)\dd t \\
        &= \operatorname{tr}\left(\int_0^\infty t^2\,
        |H|^{1/2}\,\Sigma^{1/2}R_t^2\Sigma^{1/2}\,|H|^{1/2}\,\dd t\right) \\
        &= \frac{\pi}{4}\operatorname{tr}|H|.
    \end{align*}
    The trace norms of the integrands therefore admit an integrable
    majorant, so the integral defining \(S\) converges as a Bochner
    integral in the trace class, with
    \(\|S\|_1 \leq \tfrac{2}{\pi}\,\kappa\,\tfrac{\pi}{4}
    \operatorname{tr}|H| = \tfrac{\kappa}{2}\operatorname{tr}|H|\).

    Finally, to see that \(S = A - I\), bounded operators pass through Bochner integrals, so
    \begin{equation*}
        \Sigma^{1/2}S\,\Sigma^{1/2}
        = \frac{2}{\pi}\int_0^\infty
        t^2\left[R_t - (M+t^2)^{-1}\right]\dd t
        = M^{1/2} - \Sigma
        = \Sigma^{1/2}(A - I)\Sigma^{1/2}.
    \end{equation*}
    Because
    \(\Sigma^{1/2}\) is injective with dense range, we have \(S = A - I\),
    thereby establishing \cref{eqn:ot-structure-int}.
\end{proof}

The proof of \cref{thm:ot-rank} uses the following description of the
invariant subspace \(\mathcal{K}\) through resolvents.

\begin{lemma}\label{lem:resolvent-span}
    Let \(\Sigma\) be a bounded positive self-adjoint injective operator
    on \(\V\), let \(\psi : \V \to \R^d\) be bounded, and let
    \(\mathcal{K}\) be the smallest closed \(\Sigma\)-invariant
    subspace containing \(\Range(\psi^*)\). With
    \(R_t := (\Sigma^2 + t^2)^{-1}\),
    \begin{equation*}
        \mathcal{K} = \overline{\operatorname{span}}
        \{R_t\Sigma\psi^*a : t > 0,\ a \in \R^d\}.
    \end{equation*}
\end{lemma}
\begin{proof}
    Write \(\mathcal{S}\) for the closed span. Since \(\mathcal{K}\) is a closed
    invariant subspace of the self-adjoint operator \(\Sigma\), it is
    reducing \cite[Proposition~II.3.7]{conway1990functional}, so
    \(\Sigma^2 + t^2\) is block diagonal with respect to
    \(\mathcal{K} \oplus \mathcal{K}^\perp\), and its inverse \(R_t\)
    maps \(\mathcal{K}\) into itself. As
    \(\Sigma\psi^*a \in \mathcal{K}\), this gives
    \(\mathcal{S} \subseteq \mathcal{K}\).

    For the reverse inclusion, fix \(a\) and set
    \(b := \Sigma\psi^*a\). First, \(b \in \mathcal{S}\), since
    \(\|t^2R_tb - b\| = \|R_t\Sigma^2b\| \leq \|\Sigma^2b\|/t^2\).
    Second, \(\mathcal{S}\) is \(\Sigma^2\)-invariant: the identity
    \(\Sigma^2R_tb = b - t^2R_tb\) shows that \(\Sigma^2\) maps each
    spanning vector into \(\mathcal{S}\), and boundedness passes invariance to
    the closed span. A closed \(\Sigma^2\)-invariant subspace is
    \(\Sigma\)-invariant, since
    \(\Sigma = (\Sigma^2)^{1/2}\) is a norm limit of
    polynomials in \(\Sigma^2\) by the continuous functional calculus
    \cite[Theorem~VII.1]{reedsimon1980functional}. Thus \(\mathcal{S}\) is a
    closed \(\Sigma\)-invariant subspace containing
    \(\Range(\Sigma\psi^*)\).

    It remains to check that every such subspace contains
    \(\mathcal{K}\), for which it suffices that \(\psi^*a\) lie in the
    closed \(\Sigma\)-invariant subspace generated by
    \(\Sigma\psi^*a\). The polynomials
    \(q_n(\lambda) := 1 - (1 - \lambda/\|\Sigma\|)^n\) vanish at zero,
    so \(q_n(\Sigma)\psi^*a\) lies in that subspace, while
    \(\|q_n(\Sigma)\psi^*a - \psi^*a\|^2 = \int |q_n - 1|^2\,\dd m
    \to 0\) by dominated convergence, where the spectral measure \(m\)
    of \(\psi^*a\) assigns no mass to \(\{0\}\) by the injectivity of
    \(\Sigma\).
\end{proof}

\begin{proof}[Proof of \cref{thm:ot-rank}]
    Define \(\Psi := \Sigma^{1/2}\psi^*\) and \(H := \Psi C \Psi^*\),
    so that \(\Sigma' = \Sigma + \Sigma^{1/2}H\Sigma^{1/2}\) where \(H\) is
    self-adjoint, finite rank, and nonzero since \(\Psi\) is
    injective and \(\Psi^*\) is surjective.

    In order to apply \cref{thm:ot-structure}, we first show that \(I+H\) is injective. If \((I + H)y = 0\), then
    \(y = -\Psi C\Psi^* y \in \Range(\Sigma^{1/2})\), and writing
    \(y = \Sigma^{1/2}x\) gives
    \(\Sigma' x = \Sigma^{1/2}(I + H)y = 0\) with \(x \neq 0\),
    contradicting the nondegeneracy of \(\eta\).
    Thus \(I + H\) is
    injective, and \cref{thm:ot-structure} provides the unique bounded
    positive \(A\) with \(A\Sigma A = \Sigma'\), its
    optimality, that \(S := A - I\) is trace class, and the integral
    representation \cref{eqn:ot-structure-int}.

    We now specialize \cref{eqn:ot-structure-int} to the finite rank
    perturbation to prove the four conclusions. Set \(R_t = (\Sigma^2 + t^2)^{-1}\) and \(G_t = \Sigma R_t \Sigma\) as in
    \cref{thm:ot-structure}, and let
    \(B := \Sigma^{1/2}\Sigma_{u\psi} = \Sigma^{3/2}\psi^*\).
    Restricting \(B\) to the range of \(C\), we may assume that \(C\)
    is invertible, and we set \(W_t := C^{-1} + B^*R_tB\), noting that
    \(\Psi^*G_t\Psi = \psi\Sigma^{3/2}R_t\Sigma^{3/2}\psi^* = B^*R_tB\).
    We show that the matrices \(W_t\) are invertible for every \(t\). From the proof of \cref{thm:ot-structure}, \(I + G_t H\) is invertible for every \(t\), and \(G_t H = G_t \Psi C \Psi^*\).
    Invertibility implies that \(-1 \not\in \sigma(G_t H)\), and cyclicity of the nonzero spectrum implies that \(-1 \not\in \sigma(\Psi^* G_t \Psi C )\), so \(W_t = (I + \Psi^* G_t \Psi C)\,C^{-1}\) is invertible.

    We now derive the integral representation \cref{eqn:ot-Delta}.
    Applying the push-through identity,
    \begin{equation*}
        H(I + G_tH)^{-1}
        = \Psi C\Psi^*\left(I + G_t\Psi C\Psi^*\right)^{-1}
        = \Psi C\left(I + \Psi^*G_t\Psi C\right)^{-1}\Psi^*
        = \Psi\, W_t^{-1}\, \Psi^*,
    \end{equation*}
    and since \(R_t\Sigma^{1/2}\Psi = R_t\Sigma\psi^*\),
    \cref{eqn:ot-structure-int} becomes
    \begin{equation}\label{eqn:ot-Delta}
        S = \frac{2}{\pi}\int_0^\infty t^2\,
        \left(R_t\Sigma\psi^*\right) W_t^{-1} \left(\psi\Sigma R_t\right)\dd t,
        \qquad
        A = I + S.
    \end{equation}
    Each integrand of \cref{eqn:ot-Delta} is self-adjoint of rank at most
    \(d\).

    We now prove the four claims. For the first,
    \(\Range(\Sigma\psi^*) \subseteq \mathcal{K}\) and \(R_t\) preserves
    \(\mathcal{K}\), so each integrand of \cref{eqn:ot-Delta} maps into
    \(\mathcal{K}\). As \(S\) is self-adjoint, it therefore vanishes on
    \(\mathcal{K}^\perp\).

    For the second claim, suppose \(C\) is definite. We show that the matrices \(W_t\), and hence
    their inverses in \cref{eqn:ot-Delta}, are definite with a
    common sign, so that no cancellation occurs in the integral. If \(C \succ 0\),
    then \(W_t \succ 0\) for every
    \(t\). If \(C \prec 0\), then the positivity of \(\Sigma'\) is
    equivalent, by a Schur complement, to
    \(C^{-1} + \Sigma_{\psi\psi} \prec 0\), while
    \(\Sigma^{3/2}R_t\Sigma^{3/2} \preceq \Sigma\) gives
    \begin{equation*}
        B^*R_tB = \psi\left(\Sigma^{3/2}R_t\Sigma^{3/2}\right)\psi^*
        \preceq \psi\Sigma\psi^* = \Sigma_{\psi\psi},
        \qquad\text{so}\qquad
        W_t \prec 0 \text{ for every } t.
    \end{equation*}
    In either case, \(\langle x, S x\rangle = 0\) forces
    \(\psi\Sigma R_tx = 0\) for all \(t\), that is, \(x\) is orthogonal
    to the closed span of
    \(\{R_t\Sigma\psi^*a : t > 0,\ a \in \R^d\}\). By
    \cref{lem:resolvent-span}, this span is \(\mathcal{K}\), so the
    closure of \(\Range(S)\) is \(\mathcal{K}\).

    For the third claim,
    \begin{equation*}
        \Sigma_{u\psi}C\Sigma_{\psi u} = A\Sigma A - \Sigma
        = S\Sigma + \Sigma S + S\Sigma S
    \end{equation*}
    has range contained in \(\Range(S) + \Sigma\Range(S)\), and
    therefore \(\operatorname{rank} C \leq 2\operatorname{rank} S\).

    For the final claim, assume the normalization \(\psi\psi^* = I\), which makes \(\psi^*\) an isometry
    onto \(\Range(\psi^*)\) and \(\psi^*\psi\) the orthogonal projection
    onto it.\footnote{If \(\psi\psi^* \neq I\), then the marginal cost is taken with respect to a norm induced by \(\psi\psi^*\), and we do not have an isometry on the subspace.} If \(\Range(\psi^*)\) is \(\Sigma\)-invariant, so is its
    orthogonal complement, and
    \(\Range(\Sigma_{u\psi}) = \Sigma\Range(\psi^*) = \Range(\psi^*)\), so
    both \(\RefM\) and \(\eta\) factor as products of Gaussians on
    \(\Range(\psi^*)\) with a common Gaussian factor on its complement. If
    \(T_1 = \nabla\phi_1\) is the OT map between the first factors,
    then \((u_1, u_2) \mapsto (T_1u_1, u_2)\) is the gradient of the
    convex function \(\phi_1(u_1) + \tfrac12\|u_2\|^2\) and pushes
    \(\RefM\) to \(\eta\), so it is the OT map. Since invariance
    gives
    \(\Sigma\psi^* = \psi^*\psi\Sigma\psi^* = \psi^*\Sigma_{\psi\psi}\),
    we have \(\Sigma_{u\psi}\Sigma_{\psi\psi}^{-1} = \psi^*\), and this
    map is exactly
    \begin{equation*}
        T_F(u) = u + \psi^*\left(F(\psi(u)) - \psi(u)\right),
        \qquad
        F = \psi \circ T_1 \circ \psi^*,
    \end{equation*}
    with \(F\) the OT map between the marginals.
\end{proof}

\section*{Acknowledgments}

AWH was supported by a Carl E. Pearson Fellowship and by the U.S. National Science Foundation under Award No. 2602390.
BH and AWH acknowledge support from the National Science Foundation grants DMS-2337678, ``Gaussian Processes for Scientific Machine Learning: Theoretical Analysis and Computational Algorithms'' and DMS-2208535, ``Machine Learning for Bayesian Inverse Problems''.

RB acknowledges support from the NSERC Discovery Grant (award RGPIN-2026-07896), the von K\'arm\'an instructorship at Caltech and a Department of Defense Vannevar Bush Faculty Fellowship (award N00014-22-1-2790) held by Andrew M. Stuart.

\section*{Use of large language models}
The authors used large language models to aid in performing the research and preparing this article. A variety of models were used within Claude Code to help with implementing the numerical experiments. The analysis of \cref{sec:analysis} and \cref{thm:ot-structure,thm:ot-rank} were refined and partly developed through interactive use of Fable 5 and GPT-5.6 Sol. Fable 5 through Claude Code was used for formatting and proofreading. The authors take full responsibility for the content and correctness of this article.

\bibliographystyle{preamble/mathrefs}
\bibliography{content/references}

\begin{thebibliography}{10}

\bibitem{albergo2025stochastic}
{\sc M.~Albergo, N.~M. Boffi, and E.~Vanden-Eijnden}, {\em Stochastic
  interpolants: A unifying framework for flows and diffusions}, Journal of
  Machine Learning Research, 26 (2025), pp.~1--80,
  \url{https://jmlr.org/papers/v26/23-1605.html}.

\bibitem{ambrosio2008gradient}
{\sc L.~Ambrosio, N.~Gigli, and G.~Savar{\'e}}, {\em Gradient Flows in Metric
  Spaces and in the Space of Probability Measures}, Lectures in Mathematics ETH
  Z{\"u}rich, Birkh{\"a}user, Basel, second~ed., 2008,
  \url{https://doi.org/10.1007/978-3-7643-8722-8}.

\bibitem{bakry2014analysis}
{\sc D.~Bakry, I.~Gentil, and M.~Ledoux}, {\em Analysis and Geometry of
  {M}arkov Diffusion Operators}, vol.~348 of Grundlehren der mathematischen
  Wissenschaften, Springer, Cham, 2014,
  \url{https://doi.org/10.1007/978-3-319-00227-9}.

\bibitem{baptista2025knothe}
{\sc R.~Baptista, F.~Hoffmann, M.~V.~H. Nguyen, and B.~Zhang}, {\em
  Knothe-{R}osenblatt maps via soft-constrained optimal transport}, 2025,
  \url{https://arxiv.org/abs/2511.04579}.

\bibitem{mgan}
{\sc R.~Baptista, B.~Hosseini, N.~B. Kovachki, and Y.~Marzouk}, {\em
  {Conditional Sampling with Monotone GANs: from Generative Models to
  Likelihood-Free Inference}}, SIAM/ASA Journal on Uncertainty Quantification,
  12 (2024), pp.~868--900, \url{https://doi.org/10.1137/23M1581546}.

\bibitem{baptista2025conditionalSimulation}
{\sc R.~Baptista, A.-A. Pooladian, M.~Brennan, Y.~Marzouk, and J.~Niles-Weed},
  {\em Conditional simulation via entropic optimal transport: Toward
  non-parametric estimation of conditional {B}renier maps}, in Proceedings of
  The 28th International Conference on Artificial Intelligence and Statistics,
  vol.~258 of Proceedings of Machine Learning Research, PMLR, 2025,
  pp.~4807--4815, \url{https://proceedings.mlr.press/v258/baptista25a.html}.

\bibitem{beckermann2021rational}
{\sc B.~Beckermann, A.~Cortinovis, D.~Kressner, and M.~Schweitzer}, {\em
  Low-rank updates of matrix functions {II}: Rational {K}rylov methods}, SIAM
  Journal on Numerical Analysis, 59 (2021), pp.~1325--1347,
  \url{https://doi.org/10.1137/20M1362553}.

\bibitem{beckermann2018low}
{\sc B.~Beckermann, D.~Kressner, and M.~Schweitzer}, {\em Low-rank updates of
  matrix functions}, SIAM Journal on Matrix Analysis and Applications, 39
  (2018), pp.~539--565, \url{https://doi.org/10.1137/17M1140108}.

\bibitem{benton2023error}
{\sc J.~Benton, G.~Deligiannidis, and A.~Doucet}, {\em Error bounds for flow
  matching methods}, Transactions on Machine Learning Research,  (2024),
  \url{https://openreview.net/forum?id=uqQPyWFDhY}.

\bibitem{bernstein2000rational}
{\sc D.~S. Bernstein and C.~F. Van~Loan}, {\em Rational matrix functions and
  rank-1 updates}, SIAM Journal on Matrix Analysis and Applications, 22 (2000),
  pp.~145--154, \url{https://doi.org/10.1137/S0895479898333636}.

\bibitem{blei2017variational}
{\sc D.~M. Blei, A.~Kucukelbir, and J.~D. McAuliffe}, {\em {Variational
  inference: A review for statisticians}}, Journal of the American Statistical
  Association, 112 (2017), pp.~859--877,
  \url{https://doi.org/10.1080/01621459.2017.1285773}.

\bibitem{bogachev2015gaussian}
{\sc V.~Bogachev}, {\em Gaussian Measures}, vol.~62 of Mathematical Surveys and
  Monographs, American Mathematical Society, Providence, RI, 1998,
  \url{https://doi.org/10.1090/surv/062}.

\bibitem{bogachev2010differentiable}
{\sc V.~I. Bogachev}, {\em Differentiable Measures and the {M}alliavin
  Calculus}, vol.~164 of Mathematical Surveys and Monographs, American
  Mathematical Society, Providence, RI, 2010,
  \url{https://doi.org/10.1090/surv/164}.

\bibitem{bogachev2012monge}
{\sc V.~I. Bogachev and A.~V. Kolesnikov}, {\em The {M}onge--{K}antorovich
  problem: achievements, connections, and perspectives}, Russian Mathematical
  Surveys, 67 (2012), pp.~785--890,
  \url{https://doi.org/10.1070/RM2012v067n05ABEH004808}.

\bibitem{bogachev2005triangular}
{\sc V.~I. Bogachev, A.~V. Kolesnikov, and K.~V. Medvedev}, {\em Triangular
  transformations of measures}, Sbornik: Mathematics, 196 (2005), pp.~309--335,
  \url{https://doi.org/10.1070/SM2005v196n03ABEH000882}.

\bibitem{bonnotte2013knothe}
{\sc N.~Bonnotte}, {\em {From Knothe's rearrangement to Brenier's optimal
  transport map}}, SIAM Journal on Mathematical Analysis, 45 (2013),
  pp.~64--87, \url{https://doi.org/10.1137/120874850}.

\bibitem{bouveyron2026scaling}
{\sc C.~Bouveyron and M.~Corneli}, {\em Scaling optimal transport to
  high-dimensional {G}aussian distributions with application to domain
  adaptation}, Statistics and Computing, 36 (2026), p.~88,
  \url{https://doi.org/10.1007/s11222-026-10851-7}.

\bibitem{lazy-map}
{\sc M.~C. Brennan, D.~Bigoni, O.~Zahm, A.~Spantini, and Y.~Marzouk}, {\em
  {Greedy inference with structure-exploiting lazy maps}}, in Advances in
  Neural Information Processing Systems, vol.~33, 2020,
  \url{https://proceedings.neurips.cc/paper/2020/hash/5ef20b89bab8fed38253e98a12f26316-Abstract.html}.

\bibitem{cabezas2024blackjax}
{\sc A.~Cabezas, A.~Corenflos, J.~Lao, and R.~Louf}, {\em Blackjax: Composable
  {B}ayesian inference in {JAX}}, 2024, \url{https://arxiv.org/abs/2402.10797}.

\bibitem{carere2024optimal}
{\sc G.~Carere and H.~C. Lie}, {\em Optimal low-rank posterior covariance
  approximation in linear {G}aussian inverse problems on {H}ilbert spaces},
  2024, \url{https://arxiv.org/abs/2411.01112}.

\bibitem{carlier2016vector}
{\sc G.~Carlier, V.~Chernozhukov, and A.~Galichon}, {\em {Vector quantile
  regression: An optimal transport approach}}, Annals of Statistics, 44 (2016),
  pp.~1165--1192, \url{https://doi.org/10.1214/15-AOS1401}.

\bibitem{carlier2010knothe}
{\sc G.~Carlier, A.~Galichon, and F.~Santambrogio}, {\em {From Knothe's
  transport to Brenier's map and a continuation method for optimal transport}},
  SIAM Journal on Mathematical Analysis, 41 (2010), pp.~2554--2576,
  \url{https://doi.org/10.1137/080740647}.

\bibitem{chemseddine2025conditional}
{\sc J.~Chemseddine, P.~Hagemann, G.~Steidl, and C.~Wald}, {\em Conditional
  {W}asserstein distances with applications in {B}ayesian {OT} flow matching},
  Journal of Machine Learning Research, 26 (2025), pp.~1--47,
  \url{https://www.jmlr.org/papers/v26/24-0586.html}.

\bibitem{chen2025gaussian}
{\sc Y.~Chen, B.~Hosseini, H.~Owhadi, and A.~M. Stuart}, {\em Gaussian measures
  conditioned on nonlinear observations: consistency, {MAP} estimators, and
  simulation}, Statistics and Computing, 35 (2025), p.~10,
  \url{https://doi.org/10.1007/s11222-024-10535-0}.

\bibitem{constantine2015active}
{\sc P.~G. Constantine}, {\em {Active Subspaces: Emerging Ideas for Dimension
  Reduction in Parameter Studies}}, vol.~2 of SIAM Spotlights, SIAM,
  Philadelphia, 2015, \url{https://doi.org/10.1137/1.9781611973860}.

\bibitem{ConstantineActiveSubspace}
{\sc P.~G. Constantine, E.~Dow, and Q.~Wang}, {\em {Active subspace methods in
  theory and practice: Applications to kriging surfaces}}, SIAM Journal on
  Scientific Computing, 36 (2014), pp.~A1500--A1524,
  \url{https://doi.org/10.1137/130916138}.

\bibitem{conway1990functional}
{\sc J.~B. Conway}, {\em A Course in Functional Analysis}, vol.~96 of Graduate
  Texts in Mathematics, Springer, New York, second~ed., 1990,
  \url{https://doi.org/10.1007/978-1-4757-4383-8}.

\bibitem{cotter2013mcmc}
{\sc S.~L. Cotter, G.~O. Roberts, A.~M. Stuart, and D.~White}, {\em {MCMC
  methods for functions: modifying old algorithms to make them faster}},
  Statistical Science, 28 (2013), pp.~424--446,
  \url{https://doi.org/10.1214/13-STS421}.

\bibitem{cuesta1996lower}
{\sc J.~A. Cuesta-Albertos, C.~Matr\'an-Bea, and A.~Tuero-D\'iaz}, {\em On
  lower bounds for the {$L^2$}-{W}asserstein metric in a {H}ilbert space},
  Journal of Theoretical Probability, 9 (1996), pp.~263--283,
  \url{https://doi.org/10.1007/BF02214649}.

\bibitem{lis}
{\sc T.~Cui, J.~Martin, Y.~M. Marzouk, A.~Solonen, and A.~Spantini}, {\em
  {Likelihood-informed dimension reduction for nonlinear inverse problems}},
  Inverse Problems, 30 (2014), p.~114015,
  \url{https://doi.org/10.1088/0266-5611/30/11/114015}.

\bibitem{cui-LIS}
{\sc T.~Cui and X.~T. Tong}, {\em {A unified performance analysis of
  likelihood-informed subspace methods}}, Bernoulli, 28 (2022), pp.~2788--2815,
  \url{https://doi.org/10.3150/21-BEJ1437}.

\bibitem{daprato2002second}
{\sc G.~Da~Prato and J.~Zabczyk}, {\em Second Order Partial Differential
  Equations in {H}ilbert Spaces}, vol.~293 of London Mathematical Society
  Lecture Note Series, Cambridge University Press, Cambridge, 2002,
  \url{https://doi.org/10.1017/CBO9780511543210}.

\bibitem{douglas1966majorization}
{\sc R.~G. Douglas}, {\em On majorization, factorization, and range inclusion
  of operators on {H}ilbert space}, Proceedings of the American Mathematical
  Society, 17 (1966), pp.~413--415,
  \url{https://doi.org/10.1090/S0002-9939-1966-0203464-1}.

\bibitem{moselhy2012bayesian}
{\sc T.~A. El~Moselhy and Y.~M. Marzouk}, {\em {Bayesian inference with optimal
  maps}}, Journal of Computational Physics, 231 (2012), pp.~7815--7850,
  \url{https://doi.org/10.1016/j.jcp.2012.07.022}.

\bibitem{fasi2023square}
{\sc M.~Fasi, N.~J. Higham, and X.~Liu}, {\em Computing the square root of a
  low-rank perturbation of the scaled identity matrix}, SIAM Journal on Matrix
  Analysis and Applications, 44 (2023), pp.~156--174,
  \url{https://doi.org/10.1137/22M1471559}.

\bibitem{feyel2004monge}
{\sc D.~Feyel and A.~S. {\"U}st{\"u}nel}, {\em Monge--{K}antorovitch measure
  transportation and {M}onge--{A}mp{\`e}re equation on {W}iener space},
  Probability Theory and Related Fields, 128 (2004), pp.~347--385,
  \url{https://doi.org/10.1007/s00440-003-0307-x}.

\bibitem{franzese2023functional}
{\sc G.~Franzese, G.~Corallo, S.~Rossi, M.~Heinonen, M.~Filippone, and
  P.~Michiardi}, {\em {Continuous-time functional diffusion processes}}, in
  Advances in Neural Information Processing Systems, vol.~36, 2023,
  \url{https://doi.org/10.52202/075280-1625}.

\bibitem{gneiting2007strictly}
{\sc T.~Gneiting and A.~E. Raftery}, {\em Strictly proper scoring rules,
  prediction, and estimation}, Journal of the American Statistical Association,
  102 (2007), pp.~359--378, \url{https://doi.org/10.1198/016214506000001437}.

\bibitem{gorham2017measuring}
{\sc J.~Gorham and L.~Mackey}, {\em Measuring sample quality with kernels}, in
  Proceedings of the 34th International Conference on Machine Learning, vol.~70
  of Proceedings of Machine Learning Research, PMLR, 2017, pp.~1292--1301,
  \url{https://proceedings.mlr.press/v70/gorham17a.html}.

\bibitem{gross1975logarithmic}
{\sc L.~Gross}, {\em Logarithmic {S}obolev inequalities}, American Journal of
  Mathematics, 97 (1975), pp.~1061--1083,
  \url{https://doi.org/10.2307/2373688}.

\bibitem{ho2020ddpm}
{\sc J.~Ho, A.~Jain, and P.~Abbeel}, {\em Denoising diffusion probabilistic
  models}, in Advances in Neural Information Processing Systems (NeurIPS),
  2020,
  \url{https://proceedings.neurips.cc/paper/2020/hash/4c5bcfec8584af0d967f1ab10179ca4b-Abstract.html}.

\bibitem{NUTS}
{\sc M.~D. Hoffman and A.~Gelman}, {\em The {N}o-{U}-{T}urn sampler: adaptively
  setting path lengths in {H}amiltonian {M}onte {C}arlo.}, J. Mach. Learn.
  Res., 15 (2014), pp.~1593--1623,
  \url{https://jmlr.org/papers/v15/hoffman14a.html}.

\bibitem{hosseini2025conditional}
{\sc B.~Hosseini, A.~W. Hsu, and A.~Taghvaei}, {\em Conditional optimal
  transport on function spaces}, SIAM/ASA Journal on Uncertainty
  Quantification, 13 (2025), pp.~304--338,
  \url{https://doi.org/10.1137/23M1618922}.

\bibitem{2025singleseedgenerationbrownianpaths}
{\sc A.~Jelinčič, J.~Foster, and P.~Kidger}, {\em Single-seed generation of
  {B}rownian paths and integrals for adaptive and high order {SDE} solvers},
  2024, \url{https://arxiv.org/abs/2405.06464}.

\bibitem{jordan1999introduction}
{\sc M.~I. Jordan, Z.~Ghahramani, T.~S. Jaakkola, and L.~K. Saul}, {\em {An
  introduction to variational methods for graphical models}}, Machine Learning,
  37 (1999), pp.~183--233, \url{https://doi.org/10.1023/A:1007665907178}.

\bibitem{kanagawa2018gaussianprocesseskernelmethods}
{\sc M.~Kanagawa, P.~Hennig, D.~Sejdinovic, and B.~K. Sriperumbudur}, {\em
  Gaussian processes and kernel methods: A review on connections and
  equivalences}, 2018, \url{https://arxiv.org/abs/1807.02582}.

\bibitem{dynamic_cot}
{\sc G.~Kerrigan, G.~Migliorini, and P.~Smyth}, {\em {Dynamic conditional
  optimal transport through simulation-free flows}}, in Advances in Neural
  Information Processing Systems 37 (NeurIPS 2024), 2024,
  \url{https://doi.org/10.52202/079017-2968}.

\bibitem{functional_flow}
{\sc G.~Kerrigan, G.~Migliorini, and P.~Smyth}, {\em {Functional Flow
  Matching}}, in Proceedings of the 27th International Conference on Artificial
  Intelligence and Statistics, vol.~238 of Proceedings of Machine Learning
  Research, PMLR, 2024, pp.~3934--3942,
  \url{https://proceedings.mlr.press/v238/kerrigan24a.html}.

\bibitem{knott1984optimal}
{\sc M.~Knott and C.~S. Smith}, {\em On the optimal mapping of distributions},
  Journal of Optimization Theory and Applications, 43 (1984), pp.~39--49,
  \url{https://doi.org/10.1007/BF00934745}.

\bibitem{kostrykin2007perturbation}
{\sc V.~Kostrykin, K.~A. Makarov, and A.~K. Motovilov}, {\em Perturbation of
  spectra and spectral subspaces}, Transactions of the American Mathematical
  Society, 359 (2007), pp.~77--89,
  \url{https://doi.org/10.1090/S0002-9947-06-03930-4}.

\bibitem{li2024slice}
{\sc S.~Li and C.~Moosm{\"u}ller}, {\em Approximation properties of
  slice-matching operators}, Sampling Theory, Signal Processing, and Data
  Analysis, 22 (2024), \url{https://doi.org/10.1007/s43670-024-00089-7}.

\bibitem{pmlr-scalable-gradients}
{\sc X.~Li, T.-K.~L. Wong, R.~T.~Q. Chen, and D.~Duvenaud}, {\em Scalable
  gradients for stochastic differential equations}, in Proceedings of the
  Twenty Third International Conference on Artificial Intelligence and
  Statistics, vol.~108 of Proceedings of Machine Learning Research, PMLR, 2020,
  pp.~3870--3882, \url{https://proceedings.mlr.press/v108/li20i.html}.

\bibitem{score_function_space}
{\sc J.~H. Lim, N.~B. Kovachki, R.~Baptista, C.~Beckham, K.~Azizzadenesheli,
  J.~Kossaifi, V.~Voleti, J.~Song, K.~Kreis, J.~Kautz, C.~Pal, A.~Vahdat, and
  A.~Anandkumar}, {\em {Score-based Diffusion Models in Function Space}},
  Journal of Machine Learning Research, 26 (2025), pp.~1--62,
  \url{https://jmlr.org/papers/v26/23-1472.html}.

\bibitem{lipman2023flow}
{\sc Y.~Lipman, R.~T.~Q. Chen, H.~Ben-Hamu, M.~Nickel, and M.~Le}, {\em Flow
  matching for generative modeling}, in The Eleventh International Conference
  on Learning Representations, 2023,
  \url{https://openreview.net/forum?id=PqvMRDCJT9t}.

\bibitem{manupriya2024consistent}
{\sc P.~Manupriya, R.~K. Das, S.~Biswas, and S.~N~Jagarlapudi}, {\em Consistent
  optimal transport with empirical conditional measures}, in Proceedings of The
  27th International Conference on Artificial Intelligence and Statistics,
  vol.~238 of Proceedings of Machine Learning Research, PMLR, 2024,
  pp.~3646--3654, \url{https://proceedings.mlr.press/v238/manupriya24a.html}.

\bibitem{marzouk2016sampling}
{\sc Y.~Marzouk, T.~Moselhy, M.~Parno, and A.~Spantini}, {\em {Sampling via
  measure transport: An introduction}}, in Handbook of Uncertainty
  Quantification, R.~Ghanem, D.~Higdon, and H.~Owhadi, eds., Springer
  International Publishing, Cham, 2016, pp.~1--41,
  \url{https://doi.org/10.1007/978-3-319-11259-6_23-1}.

\bibitem{masarotto2019procrustes}
{\sc V.~Masarotto, V.~M. Panaretos, and Y.~Zemel}, {\em Procrustes metrics on
  covariance operators and optimal transportation of {G}aussian processes},
  Sankhy\=a A, 81 (2019), pp.~172--213,
  \url{https://doi.org/10.1007/s13171-018-0130-1}.

\bibitem{masarotto2024transportation}
{\sc V.~Masarotto, V.~M. Panaretos, and Y.~Zemel}, {\em Transportation-based
  functional {ANOVA} and {PCA} for covariance operators}, Electronic Journal of
  Statistics, 18 (2024), pp.~1887--1916,
  \url{https://doi.org/10.1214/24-EJS2240}.

\bibitem{Merton1976JumpDiffusion}
{\sc R.~C. Merton}, {\em Option pricing when underlying stock returns are
  discontinuous}, Journal of Financial Economics, 3 (1976), pp.~125--144,
  \url{https://doi.org/10.1016/0304-405X(76)90022-2}.

\bibitem{minh2022finite}
{\sc H.~Q. Minh}, {\em Finite sample approximations of exact and entropic
  {W}asserstein distances between covariance operators and {G}aussian
  processes}, SIAM/ASA Journal on Uncertainty Quantification, 10 (2022),
  pp.~96--124, \url{https://doi.org/10.1137/21M1410488}.

\bibitem{moosmuller2023linear}
{\sc C.~Moosm{\"u}ller and A.~Cloninger}, {\em Linear optimal transport
  embedding: provable {W}asserstein classification for certain rigid
  transformations and perturbations}, Information and Inference: A Journal of
  the IMA, 12 (2023), pp.~363--389,
  \url{https://doi.org/10.1093/imaiai/iaac023}.

\bibitem{muzellec2019subspace}
{\sc B.~Muzellec and M.~Cuturi}, {\em Subspace detours: Building transport
  plans that are optimal on subspace projections}, in Advances in Neural
  Information Processing Systems, vol.~32, 2019,
  \url{https://proceedings.neurips.cc/paper/2019/hash/f9beb1e831faf6aaec2a5cecaf1af293-Abstract.html}.

\bibitem{niles2022estimation}
{\sc J.~Niles-Weed and P.~Rigollet}, {\em Estimation of {W}asserstein distances
  in the spiked transport model}, Bernoulli, 28 (2022), pp.~2663--2688,
  \url{https://doi.org/10.3150/21-BEJ1433}.

\bibitem{inf_dim_diffusion}
{\sc J.~Pidstrigach, Y.~Marzouk, S.~Reich, and S.~Wang}, {\em
  {Infinite-dimensional diffusion models}}, Journal of Machine Learning
  Research, 25 (2024), pp.~1--52,
  \url{https://jmlr.org/papers/v25/23-1271.html}.

\bibitem{pinski2015algorithms}
{\sc F.~J. Pinski, G.~Simpson, A.~M. Stuart, and H.~Weber}, {\em {Algorithms
  for Kullback--Leibler approximation of probability measures in infinite
  dimensions}}, SIAM Journal on Scientific Computing, 37 (2015),
  pp.~A2733--A2757, \url{https://doi.org/10.1137/14098171X}.

\bibitem{pinski2015kullback}
{\sc F.~J. Pinski, G.~Simpson, A.~M. Stuart, and H.~Weber}, {\em
  {Kullback--Leibler approximation for probability measures on infinite
  dimensional spaces}}, SIAM Journal on Mathematical Analysis, 47 (2015),
  pp.~4091--4122, \url{https://doi.org/10.1137/140962802}.

\bibitem{ponnoprat2025minimax}
{\sc D.~Ponnoprat and M.~Imaizumi}, {\em Minimax rates of estimation for
  optimal transport map between infinite-dimensional spaces}, 2025,
  \url{https://arxiv.org/abs/2505.13570}.

\bibitem{rahimi2007randomfeatures}
{\sc A.~Rahimi and B.~Recht}, {\em Random features for large-scale kernel
  machines}, in Advances in Neural Information Processing Systems 20,
  Proceedings of the Twenty-First Annual Conference on Neural Information
  Processing Systems, Vancouver, British Columbia, Canada, December 3--6, 2007,
  J.~C. Platt, D.~Koller, Y.~Singer, and S.~T. Roweis, eds., Curran Associates,
  Inc., 2007, pp.~1177--1184,
  \url{https://proceedings.neurips.cc/paper/2007/hash/013a006f03dbc5392effeb8f18fda755-Abstract.html}.

\bibitem{friendly-triangular}
{\sc M.~Ramgraber, D.~Sharp, M.~Le~Provost, and Y.~Marzouk}, {\em {A friendly
  introduction to triangular transport}}, 2025,
  \url{https://arxiv.org/abs/2503.21673}.

\bibitem{reedsimon1980functional}
{\sc M.~Reed and B.~Simon}, {\em Methods of Modern Mathematical Physics {I}:
  Functional Analysis}, Academic Press, New York, revised and enlarged~ed.,
  1980.

\bibitem{rezende2015variational}
{\sc D.~Rezende and S.~Mohamed}, {\em {Variational inference with normalizing
  flows}}, in International Conference on Machine Learning, PMLR, 2015,
  pp.~1530--1538, \url{https://proceedings.mlr.press/v37/rezende15.html}.

\bibitem{santoro2025large}
{\sc L.~V. Santoro and V.~M. Panaretos}, {\em Large sample theory for
  {B}ures--{W}asserstein barycentres}, The Annals of Applied Probability, 35
  (2025), pp.~3215--3241, \url{https://doi.org/10.1214/25-AAP2192}.

\bibitem{shmueli2024lowrank}
{\sc S.~Shmueli, P.~Drineas, and H.~Avron}, {\em Low-rank updates of matrix
  square roots}, Numerical Linear Algebra with Applications, 31 (2024),
  p.~e2528, \url{https://doi.org/10.1002/nla.2528}.

\bibitem{sohl-dickstein2015noneq}
{\sc J.~Sohl-Dickstein, E.~A. Weiss, N.~Maheswaranathan, and S.~Ganguli}, {\em
  Deep unsupervised learning using nonequilibrium thermodynamics}, in
  Proceedings of the 32nd International Conference on Machine Learning (ICML),
  2015, pp.~2256--2265,
  \url{https://proceedings.mlr.press/v37/sohl-dickstein15.html}.

\bibitem{song2019score}
{\sc Y.~Song and S.~Ermon}, {\em Generative modeling by estimating gradients of
  the data distribution}, in Advances in Neural Information Processing Systems
  (NeurIPS), vol.~32, 2019,
  \url{https://proceedings.neurips.cc/paper_files/paper/2019/file/3001ef257407d5a371a96dcd947c7d93-Paper.pdf}.

\bibitem{song2021score_sde}
{\sc Y.~Song, J.~Sohl-Dickstein, D.~P. Kingma, A.~Kumar, S.~Ermon, and
  B.~Poole}, {\em Score-based generative modeling through stochastic
  differential equations}, in International Conference on Learning
  Representations (ICLR), 2021,
  \url{https://openreview.net/forum?id=PxTIG12RRHS}.

\bibitem{sprungk2020local}
{\sc B.~Sprungk}, {\em On the local {L}ipschitz stability of {B}ayesian inverse
  problems}, Inverse Problems, 36 (2020), p.~055015,
  \url{https://doi.org/10.1088/1361-6420/ab6f43}.

\bibitem{stuart2010inverse}
{\sc A.~M. Stuart}, {\em Inverse problems: a {B}ayesian perspective}, Acta
  Numerica, 19 (2010), pp.~451--559,
  \url{https://doi.org/10.1017/S0962492910000061}.

\bibitem{tabak2021conditional}
{\sc E.~G. Tabak, G.~Trigila, and W.~Zhao}, {\em {Data driven conditional
  optimal transport}}, Machine Learning, 110 (2021), pp.~3135--3155,
  \url{https://doi.org/10.1007/s10994-021-06060-0}.

\bibitem{villani2009optimal}
{\sc C.~Villani}, {\em {Optimal Transport: Old and New}}, vol.~338 of
  Grundlehren der mathematischen Wissenschaften, Springer, Berlin, 2009,
  \url{https://doi.org/10.1007/978-3-540-71050-9}.

\bibitem{debias-coarsely}
{\sc Z.~Y. Wan, R.~Baptista, Y.-f. Chen, J.~Anderson, A.~Boral, F.~Sha, and
  L.~Zepeda-N{\'u}{\~n}ez}, {\em {Debias Coarsely, Sample Conditionally:
  Statistical Downscaling through Optimal Transport and Probabilistic Diffusion
  Models}}, in Advances in Neural Information Processing Systems, vol.~36,
  2023, \url{https://doi.org/10.52202/075280-2069}.

\bibitem{wang2025conditional}
{\sc Z.~O. Wang, R.~Baptista, Y.~Marzouk, L.~Ruthotto, and D.~Verma}, {\em
  {Efficient neural network approaches for conditional optimal transport with
  applications in Bayesian inference}}, SIAM Journal on Scientific Computing,
  47 (2025), \url{https://doi.org/10.1137/24M1678659}.

\bibitem{pathwise_gp}
{\sc J.~T. Wilson, V.~Borovitskiy, A.~Terenin, P.~Mostowsky, and M.~P.
  Deisenroth}, {\em Pathwise conditioning of {G}aussian processes}, J. Mach.
  Learn. Res., 22 (2021), pp.~1--47,
  \url{https://jmlr.org/papers/v22/20-1260.html}.

\bibitem{yun2025gaussian}
{\sc H.~Yun and Y.~Zemel}, {\em Gaussian optimal transport beyond {B}renier's
  theorem}, 2025, \url{https://arxiv.org/abs/2512.21464}.

\bibitem{certified}
{\sc O.~Zahm, T.~Cui, K.~Law, A.~Spantini, and Y.~Marzouk}, {\em {Certified
  dimension reduction in nonlinear Bayesian inverse problems}}, Mathematics of
  Computation, 91 (2022), pp.~1789--1835,
  \url{https://doi.org/10.1090/mcom/3737}.

\bibitem{func_norm_flow}
{\sc Y.~Zhao, H.~Lu, J.~Jia, and T.~Zhou}, {\em {Functional normalizing flow
  for statistical inverse problems of partial differential equations}}, 2024,
  \url{https://arxiv.org/abs/2411.13277}.

\end{thebibliography}

\end{document}